\documentclass[a4paper, reqno,12pt]{amsart}

\usepackage{tikz, tikz-3dplot}
\usepackage{amsmath, amsthm, amssymb,graphics }

\theoremstyle{plain}
\newtheorem{te}{Theorem}[section]
\newtheorem{lem}[te]{Lemma}
\newtheorem{co}[te]{Corollary}
\newtheorem{pr}[te]{Proposition}
\newtheorem{de}[te]{Definition}

\newtheorem{con}[te]{Conjecture}
\theoremstyle{remark}
\newtheorem{re}[te]{Remark}
\newtheorem*{ack*}{Acknowledgment}

\def\z{{\bf z}}

\def\n{{\bf n}}
\def\b{{\bf b}}
\def\t{{\bf t}}

\def\e{{\bf e}}

\def\0{{\bf 0}}

\def\I{{\mathbb I}}

\def\T{{\mathbb T}}
\def\R{{\mathbb R}}
\def\E{{\mathbb E}}

\def\C{{\mathbb C}}
\def\S{{\mathbb S}}
\def\Z{{\mathbb Z}}
\def\P{{\mathbb P}}

\def\Dec{{\operatorname{Dec}}}
\def\TD{{\operatorname{TriDec}}}
\def\supp{{\operatorname{supp}\,}}
\def\nint{\mathop{\diagup\kern-13.0pt\int}}

\def\les{{\;\lessapprox}\;}
\def\dist{{\operatorname{dist}\,}}

\def\Ic{{\mathcal I}}\def\Fc{{\mathcal F}}
\def\Nc{{\mathcal N}}
\def\Tc{{\mathcal T}}\def\Mc{{\mathcal M}}
\def\Bc{{\mathcal B}}
\def\Qc{{\mathcal Q}}
\def\Pc{{\mathcal P}}

\begin{document}
\author{Ciprian Demeter}
\address{Department of Mathematics, Indiana University,  Bloomington IN}
\email{demeterc@indiana.edu}
\author{Larry Guth}
\address{Department of Mathematics, MIT, Cambridge MA}
\email{lguth@math.mit.edu}
\author{Hong Wang}
\address{Department of Mathematics, MIT, Cambridge MA}
\email{hongwang@mit.edu}

\dedicatory{To the memory of Jean Bourgain}

\thanks{The first author is partially supported by the  NSF grant DMS-1800305.
	The second author is partially supported by a Simons Investigator Award. The third author was partially supported by the Simons Foundation grant of David Jerison while she was at MIT, and supported by the S.S. Chern Foundation for Mathematics Research Fund and by the NSF while at IAS. }

\begin{abstract}
We develop a toolbox for proving decouplings into boxes with diameter smaller than the canonical scale. As an application of this new technique,  we solve three problems for which earlier methods have failed. We start by verifying  the small cap decoupling for the parabola.
Then we find sharp estimates for exponential sums with small frequency separation  on the moment curve in $\R^3$.  This part of the work relies on recent improved Kakeya-type estimates for planar tubes, as well as on new multilinear incidence bounds for plates and planks.

 We also combine our method with the recent advance on the reverse square function estimate, in order to prove small cap   decoupling  into square-like caps for the two dimensional cone.

The Appendix by Roger Heath-Brown contains an application of the new exponential sum estimates for the moment curve, to the Riemann zeta-function.
\end{abstract}
\title[Small cap decouplings]{Small cap decouplings}

\maketitle

\section{A brief overview of ``old" and ``new" decouplings}

In this paper we will address three related problems, one for the parabola, one for the twisted cubic and another one for the cone.
\smallskip

All functions we work with will implicitly be assumed to be in the Schwartz class. For each positive measure  set $B\subset \R^n$ and each $F:\R^n\to\C$ we will denote by
$$\Pc_BF(x)=\int_{B}\widehat{F}(\xi)e(\xi\cdot x)d\xi$$
the Fourier projection of $F$ onto $L^2(B)$.

\begin{de}
Let us assume that we have a family $\Bc$ consisting of  $N$
pairwise disjoint sets $B_1,\ldots,B_N$ in $\R^n$. Given $p,r\ge 2$, the $l^r(L^p)$ decoupling  (or simply $l^r$ decoupling if we do not want to emphasize the Lebesgue exponent $p$) constant $\Dec(\Bc,p,r)$ is the smallest number for which the inequality
$$\|F\|_{L^p(\R^n)}\le \Dec(\Bc,p,r) N^{\frac12-\frac1r}(\sum_{i=1}^N\|\Pc_{B_i}F\|^r_{L^p(\R^n)})^{\frac1r}$$
holds uniformly for all functions $F$ with spectrum in $\cup_{B\in\Bc} B$.
\end{de}
It is easy to see that   $1\lesssim \Dec(\Bc,p,r_2)\le  \Dec(\Bc,p,r_1)$ whenever $r_2\ge r_1$ (Exercise 9.7 in \cite{Dembook}). Moreover, $\Dec(\Bc,2,r)=1$ for each $r\ge 2$. For $p>2$, the smallness of $\Dec(\Bc,p,r)$ is a measure of $L^p$ orthogonality associated with $\Bc$.

\smallskip

In recent years, decoupling constants have been investigated in the context when the sets $B_i$ are almost rectangular boxes covering $\delta$-neighborhoods of various manifolds $\Mc$. For each $\delta<1$, the collection $\Bc_\Mc(\delta)$ consists of $N_\delta$ such boxes, and $\lim_{\delta\to 0}N_\delta=\infty$. An ideal result is of the form
\begin{equation}
\label{1}
\Dec(\Bc_\Mc(\delta),p,r)\lesssim_\epsilon \delta^{-\epsilon}
\end{equation}
for various values of $p,r$. We will refer to this type of favorable estimate as $l^r(L^p)$ decoupling. Using interpolation and the fact that $\Dec(\Bc_\Mc(\delta),2,r)=1$, the range for which \eqref{1} holds for a fixed $r$ is of the form $2\le p\le p_r$.

What the previous results of this type have in common is the fact that the sets $B\in\Bc_\Mc(\delta)$ are maximal with respect to the property of being  almost rectangular (or essentially flat; see the beginning of the next section for a precise definition). By that we mean the fact that if the diameter of  $B$ were significantly  larger, $B$ would end up being a curved box. This maximal property was essential in the previous arguments; it is precisely the feature that allows for the use of the fundamental tool called (generalized) parabolic rescaling. In short, this consists of the use of affine transformations to map  caps on the manifold to the whole manifold. We will refer to this particular scale (diameter) of the boxes $B$ as the {\em canonical scale}. This is of course a function of both $\delta$ and $\Mc$.

Two families of manifolds have proved particularly useful for applications. The first one consists of the hypersurfaces in $\R^n$ with nonzero Gaussian curvature. The canonical scale in this case is $\delta^{1/2}$. The other one consists of the curves with torsion in $\R^n$, whose canonical scale turns out to be $\delta^{1/n}$.

In this paper we initiate a systematic study of decoupling into boxes with diameter smaller than the canonical scale. We will refer to this as {\em small cap decoupling}. Part of the motivation for addressing this new class of problems comes from Number Theory, via the following rather elementary principle, first proposed by Jean Bourgain in \cite{Bor2}. We state it somewhat loosely at this point, but will later revisit concrete examples in more detail.

 For a set $S$ and for $1\le p<\infty$, we will use the normalized $L^p$ integral
$$\|F\|_{L^p_\sharp(S)}=(\frac1{|S|}\int_S|F|^p)^{1/p}$$

\begin{pr}[Reverse H\"older's inequality for exponential sums]
	\label{revHolforexpssums}	
	
	Let $\Mc$ be a manifold in $\R^n$ and let $\Bc_{\Mc}(\delta)$
	be a pairwise disjoint cover of the $\delta$-neighborhood of  $\Mc$ with boxes $B$ of thickness $\delta$. The diameter of $B$ need not be the same as the canonical scale associated with $(\delta,\Mc)$. For each $B\in\Bc_\Mc(\delta)$, let $\xi_B$ be a point in $B\cap \Mc$.

	Assume  that for some $p,r\ge 2$ we have $\Dec(\Bc_{\Mc}(\delta),p,r)\lesssim _\epsilon\delta^{-\epsilon}$.
	Then for each family of complex coefficients $(a_B)_{B\in\Bc_{\Mc}(\delta)}$ with essentially constant magnitude (say $1\le |a_B|\le 2$) and for each cube $Q_R\subset \R^n$ with diameter  $R\ge \delta^{-1}$ we have
	\begin{equation}
	\label{ewfdrw44}
	\|\sum_{B\in\Bc_{\Mc}(\delta)}a_B e(\xi_B\cdot x)\|_{L^p_\sharp(Q_R)}\lesssim_\epsilon \delta^{-\epsilon}\|a_B\|_{l^2}.
	\end{equation}
\end{pr}

In all applications, the points $\xi_B$ will be $\delta$-separated. Apart from the term $\delta^{-\epsilon}$, inequality \eqref{ewfdrw44} is sharp, since simple orthogonality considerations show that  $$\|a_B\|_{l^2}\approx \|\sum_{B\in\Bc_{\Mc}(\delta)}a_B e(\xi_B\cdot x)\|_{L^2_\sharp(Q_R)}.$$
Our work here addresses the case when there is no $l^2(L^p)$ decoupling for the sets in $\Bc_\Mc(\delta)$. The above proposition shows that when the  coefficients have essentially constant magnitude, an $l^p(L^p)$ decoupling is just as good for applications to exponential sum estimates.

As one of the main applications, small cap decoupling allows us to investigate moments of periodic exponential sums  over subsets smaller than the full domain of periodicity (e.g. a major arc, in the number theoretic terminology). Bourgain's papers \cite{Bo34} and \cite{Bo13} contain a few examples of this nature. The small cap decouplings there are handled with ad hoc arguments that are restricted to specific exponents, and are ultimately reduced to canonical scale decouplings for higher dimensional manifolds (see, e.g. \cite{BD13}).

As mentioned earlier, a key obstacle in proving small cap decouplings is the inefficiency of rescaling. Our new method avoids parabolic rescaling in the main body of the argument, and only makes use of it in the multilinear-to-linear reduction. Instead, it relies on a two-step decoupling, which amounts to refining and carefully combining two previously known ingredients. One is the decoupling for the canonical scale, the other one is the so-called flat (or trivial) decoupling.  Executing this strategy will rely crucially on new  Kakeya-type estimates for boxes exhibiting a  wide range of shapes.

The phenomenon described in this paper is very broad. For reasons of brevity, we illustrate it here with only three conjectures. They are described in Section \ref{afewcon}, with some words about the underlying motivation and the necessity of new methods. In Section \ref{s2} we give some details on how we solve or make progress on such problems. The rest of the paper will be devoted to proofs.

\begin{ack*}
This paper has benefitted from motivating conversations between the first author, Jean Bourgain and Shaoming Guo. We are grateful to  a very careful referee, whose many suggestions led to the improvement of our presentation.
\end{ack*}
\smallskip

\section{A few conjectures}
\label{afewcon}

An almost rectangular (or essentially flat) box $B$ is a set in $\R^n$ for which there is a genuine rectangular box (parallelepiped) $R$ such that $C^{-1}R\subset B\subset CR$ for some $C=O(1)$. The orientation and dimensions of $B$ are (somewhat loosely) defined to be the same as those of $R$.
All boxes considered in this paper will be quantitatively far from being degenerate, in other words, they will be almost rectangular.
\medskip

To describe the first problem,  let $\frac12\le \alpha\le 1$.
Let $\Gamma_\alpha(\delta)$ be a partition of the vertical  $\delta$-neighborhood $\Nc_{\P^1}(\delta)$ of the parabola
$$\P^1=\{(\xi,\xi^2):\;|\xi|\le 1\}$$
into almost rectangular boxes $\gamma$ with diameter $\sim \delta^{\alpha}$ and thickness $\sim \delta$. The case $\alpha=\frac12$ is rather special, as  $\delta^{\frac12}$ is the canonical scale for the parabola.  To emphasize this we will denote $\Gamma_{\frac12}(\delta)$ by $\Theta(\delta)$ and the elements of $\Theta(\delta)$ by $\theta$.

\smallskip

\begin{con}[Small cap $l^p$ decoupling for the parabola]
\label{c1}	
Assume $F:\R^2\to\C$ has the Fourier transform supported on $\Nc_{\P^1}(R^{-1})$. Then for each $2\le p\le 2+\frac2\alpha$ we have
$$\|F\|_{L^p(\R^2)}\lesssim_\epsilon R^{\alpha(\frac12-\frac1p)+\epsilon}(\sum_{\gamma\in \Gamma_\alpha(R^{-1})}\|\Pc_\gamma F\|^p_{L^p(\R^2)})^{\frac1p}.$$
In other words,
$$\Dec(\Gamma_\alpha(R^{-1}),p,p)\lesssim_\epsilon R^{\epsilon}.$$
\end{con}
 The range for $p$, as well as the upper bound $R^{\alpha(\frac12-\frac1p)}$ are sharp. Indeed, assume that each
 $\widehat{\Pc_\gamma F}$ is a smooth approximation of $1_\gamma$. Then
$$|F(x)|\sim |\int_{\Nc_{\P^1}(R^{-1})}e(\xi\cdot x)d\xi|\sim \frac{1}R$$
when $|x|\lesssim 1$. In particular, $\|F\|_{L^p(\R^2)}\gtrsim R^{-1}$. Also,  $\|\Pc_\gamma F\|_{L^p(\R^2)}\sim R^{(1+\alpha)\frac{1-p}{p}}$. Note also that there are $\sim R^\alpha$ boxes $\gamma$ in $\Gamma_\alpha(R^{-1})$.
\smallskip

The cases $\alpha=\frac12$ and $\alpha=1$ of this conjecture were known. When $\alpha=\frac12$, it is an immediate consequence of the following $l^2$ decoupling proved by Bourgain and the first author in \cite{BD}.

\begin{te}[$l^2$ decoupling for boxes of canonical scale]
\label{t8}	
Assume $F:\R^2\to\C$ has Fourier transform supported on $\Nc_{\P^1}(R^{-1})$. Then for each $2\le p\le 6$ we have
$$\|F\|_{L^p(\R^2)}\lesssim_\epsilon R^{\epsilon}(\sum_{\theta\in \Theta(R^{-1})}\|\Pc_\theta F\|^2_{L^p(\R^2)})^{\frac12}.$$
\end{te}

We state and sketch a simple proof of the case $\alpha=1$ of the conjecture.

\begin{te}
	Assume $F:\R^2\to\C$ has Fourier transform supported on $\Nc_{\P^1}(R^{-1})$. Then for each $2\le p\le 4$ we have
	$$\|F\|_{L^p(\R^2)}\lesssim_\epsilon R^{\frac12-\frac1p+\epsilon}(\sum_{\gamma\in \Gamma_{1}(R^{-1})}\|\Pc_\gamma F\|^p_{L^p(\R^2)})^{\frac1p}.$$
\end{te}
\begin{proof}
The result is immediate for $p=2$, due to orthogonality. Using interpolation (Exercise 9.21 in \cite{Dembook}), it suffices to prove the case $p=4$. The bilinear version of this is an immediate consequence of Cordoba's classical square function estimate (Exercise 3.5. in \cite{Dembook}). The bilinear-to-linear reduction is also standard. See e.g. subsection \ref{btol} in this paper. 

\end{proof}

It is easy to see that the $l^2$ decoupling that holds in the case $\alpha=\frac12$ cannot hold for any other value of $\alpha>\frac12$, unless $p=2$. This is due to the following result, and the fact that each $\theta\in \Theta(R^{-1})$  contains $L\sim R^{\alpha-\frac12}\gg 1$ boxes $\gamma\in\Gamma_{\alpha}(R^{-1})$.
\begin{pr}[Flat  decoupling]
	\label{p6}
	Let $B$ be a rectangular box in $\R^n$, and let $B_1,\ldots, B_L$ be a partition of $B$ into  congruent rectangular boxes that are  translates of each other.
	
	Then for each $2\le p,r\le \infty$ we have
	$$ \|\Pc_BF\|_{L^p(\R^n)}\lesssim L^{1-\frac1p-\frac1r} (\sum_{i=1}^L\|\Pc_{B_i}F\|_{L^p(\R^n)}^r)^{\frac1r}.$$
	Moreover, apart from universal multiplicative constants, the upper bound $L^{1-\frac1p-\frac1r}$ is sharp.
\end{pr}
\begin{proof}
The result is clear when $p=2$, invoking orthogonality and H\"older's inequality in $r$. It is also clear for $p=\infty$ (and all $r\ge 1$). All other cases follow using special interpolation, see Exercise 9.21 in \cite{Dembook}.

The lower bound can be obtained by testing with $F$ having Fourier transform equal to a smooth approximation of $1_B$.

\end{proof}

We explain the difficulty of Conjecture \ref{c1} by describing a naive approach to it that fails. It is tempting to  attack this conjecture via a two-step decoupling as follows. First, Theorem \ref{t8} gives
\begin{equation}
\label{e5}
\|F\|_{L^p(\R^2)}\lesssim_\epsilon R^{\frac12(\frac12-\frac1p)+\epsilon}(\sum_{\theta\in \Theta(R^{-1})}\|\Pc_\theta F\|^p_{L^p(\R^2)})^{\frac1p}.
\end{equation}

It remains to decouple each $\theta$ in smaller boxes $\gamma\subset \theta$. Since $\theta$ is essentially a flat box, Proposition \ref{p6} gives
\begin{equation}
\label{e7}
\|\Pc_\theta F\|_{L^p(\R^2)}\lesssim R^{(\alpha-\frac12)(1-\frac2p)}(\sum_{\gamma\in\Gamma_\alpha(R^{-1})\atop{\gamma\subset \theta}}\|\Pc_\gamma F\|^p_{L^p(\R^2)})^{\frac1p}.
\end{equation}
Combining \eqref{e5} and \eqref{e7} we arrive at the inequality
$$\|F\|_{L^p(\R^2)}\lesssim_\epsilon R^{(2\alpha-\frac12)(\frac12-\frac1p)+\epsilon}(\sum_{\gamma\in \Gamma_\alpha(R^{-1})}\|\Pc_\gamma F\|^p_{L^p(\R^2)})^{\frac1p}.$$
Comparing this with Conjecture \ref{c1} reveals that the intended exponent of $R$ is too large. Here is the explanation for this discrepancy. Both inequalities \eqref{e5} and \eqref{e7} are sharp, in the sense that the precise exponents of $R$ in the two upper bounds  can be realized for specific choices of functions $F$. However, the point of the stronger inequality in Conjecture \ref{c1} is that these bounds cannot be simultaneously realized by the same function.

We also observe that parabolic rescaling, a tool that was so vital for the proof of Theorem \ref{t8}, is no longer appropriate for attacking Conjecture \ref{c1}, when $\alpha>\frac12$. Roughly speaking, parabolic rescaling amounts to stretching by some factor $\sigma$ in the $\xi_1$ frequency direction, and by $\sigma^2$ in the $\xi_2$ direction. While $\gamma\in\Gamma_\alpha(R^{-1})$ has dimensions $\sim (R^{-\alpha},R^{-1})$, the rescaled version of $\gamma$ has dimensions $\sim (\sigma R^{-\alpha},\sigma^2 R^{-1})$, and thus it is never in a collection of the type $\Gamma_\alpha(\delta)$.
\medskip

Let us now state a conjecture for the moment curve.

\begin{con}
	\label{confgght}	
	For each $n\ge 2$,  $0\le \beta\le n-1$ and $s\ge 1$ we have
	$$\int_{[0,1]^{n-1}\times [0,\frac1{N^\beta}]}|\sum_{k=1}^Ne(kx_1+k^2x_2+\ldots+k^nx_n)|^{2s}dx\lesssim_\epsilon N^{\epsilon}(N^{s-\beta}+N^{2s-\frac{n(n+1)}2}). $$	
\end{con}
This inequality is easily seen to be true for $s=1$ ($L^2$ orthogonality) and $s=\infty$ (triangle inequality).
If $\beta$ and $n$ are fixed, it suffices to verify the conjecture for the critical exponent $s_c=\frac{n(n+1)}{2}-\beta$. Indeed, the remaining values of $s$ are addressed using H\"older's inequality with  indices $1$, $s_c$ and $\infty$. It is worth observing that in the range $2\le 2s\le 2s_c$, the conjecture is the same as the reverse H\"older's inequality
$$\|\sum_{k=1}^Ne(kx_1+k^2x_2+\ldots+k^nx_n)\|_{L^{2s}_\sharp([0,1]^{n-1}\times [0,\frac1{N^\beta}])}\lesssim_{\epsilon}N^{\frac12+\epsilon}.$$

When $n=2$, the conjecture can easily be verified using standard  Gauss sum estimates. There is also an alternative argument, as a consequence of our solution to Conjecture \ref{c1}.

The case $\beta=0$, $n\ge 3$ is known as Vinogradov's Mean Value Theorem, solved recently in \cite{Woo} ($n=3$) and \cite{BDG} ($n\ge 4$). We are aware of two arguments for $n=3$ and $\beta=2$, one by Bombieri-Iwaniec \cite{BI} and another one by Bourgain \cite{Bo34}. No other cases seem to have been  known when $n\ge 3$.

We will see that this conjecture can be approached using a small cap decoupling.
\medskip

Before stating a third conjecture, we motivate it  with the following result proved in \cite{BD3}.

\begin{te}[$l^p$ decoupling for boxes of canonical scale]
	
	Let  $\Mc\subset \R^3$ be the  graph of a smooth function on some compact subset of $\R^2$.	Assume $\Mc$ has nowhere zero Gaussian curvature.  Let $\Theta_{\Mc}(\delta)$ be a partition of  the $\delta$-neighborhood $\Nc_{\Mc}(\delta)$ of $\Mc$ into almost rectangular boxes $\theta$ with dimensions $\sim (\delta,\delta^{\frac12},\delta^{\frac12})$.
    There are $\sim \delta^{-1}$ such boxes.

 Assume  $F:\R^3\to\C$ satisfies $\supp(\widehat{F})\subset \Nc_{\Mc}(R^{-1})$.
Then for each $2\le p\le 4$ we have
	$$\|F\|_{L^p(\R^3)}\lesssim_\epsilon R^{\frac12-\frac1p+\epsilon}(\sum_{\theta\in\Theta_{\Mc}(R^{-1})}\|\Pc_\theta F\|_{L^p(\R^3)}^p)^{\frac1p}.$$
	\end{te}
One expects that the same result  holds for the (truncated) cone, which has everywhere zero Gaussian curvature

$${\C}o^2=\{(\xi_1,\xi_2,\sqrt{\xi_1^2+\xi_2^2}):1\le \xi_1^2+\xi_2^2\le 2 \}.$$
Let $\Theta_{{\C}o^2}(\delta)$ be a partition of its $\delta$-neighborhood $\Nc_{{\C}o^2}(\delta)$ into almost rectangular boxes $\theta$ with dimensions $\sim (\delta^{\frac12},\delta,1)$. It is easy to see that these boxes have canonical scale; if they were wider than $\delta^{1/2}$, they would no longer be almost rectangular. We can make $\theta$ smaller in two ways, either narrower or shorter. We illustrate small cap decoupling in the latter case, as it has a different flavor than the previously considered case of the parabola.

 To this end,  we partition each $\theta$ into almost rectangular boxes $\gamma$ with dimensions $\sim (\delta^{\frac12},\delta,\delta^{\frac12})$. Let $\Gamma(\delta)$ the collection of all these $\gamma$.

\bigskip

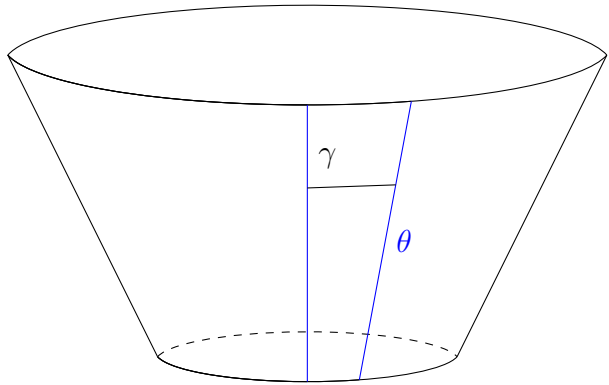
\textbf{Figure 1: caps $\theta$ and $\gamma$ for the cone}

\begin{center}
	
\begin{tikzpicture}[scale=2]
\draw (0,0) arc(170:10:2cm and 0.4 cm) coordinate[pos=0] (a);
\draw (0,0) arc(-170:-10:2cm and 0.4cm) coordinate (b);
\draw (0,0) arc(-170: -90: 2cm and 0.4cm) coordinate(c);
\draw (0,0) arc(-170: -70: 2cm and 0.4cm) coordinate(d);
\coordinate(mid) at ([yshift=-4cm]$(a)!0.5!(b)$); 
\coordinate(start) at ($(mid)!0.5!(a)$);
\draw (start)--(a);
\draw[dashed] (start) arc(170:10:1cm and 0.2cm) coordinate[pos=0] (aa);
\draw (start) arc(-170: -10: 1cm and 0.2cm) coordinate (bb);
\draw (start) arc(-170: -90: 1cm and 0.2cm) coordinate (cc);
\draw(start) arc(-170: -70: 1cm and 0.2cm) coordinate (dd);
\draw (bb)--(b);
\draw[blue](c)--(cc);
\draw[blue](d)--(dd);
\node[right, blue] at ($(d)!0.5!(dd)$) {$\theta$};
\coordinate(g) at ($(c)! 0.3!(cc)$);
\coordinate(gg) at ($(d)! 0.3! (dd)$);
\draw(g)--(gg);
\node[above, right] at ([yshift=0.2cm] g) {$\gamma$};
\end{tikzpicture}
\end{center}

\bigskip

We recall the following conjecture stated at the end of \cite{BDK}.
\begin{con}[Small cap  decoupling for the cone]
\label{c2}
Let $F:\R^3\to\C$ with the Fourier transform supported inside $\Nc_{{\C}o^2}(R^{-1})$. Then for each $2\le p\le 4$ we have
$$\|F\|_{L^p(\R^3)}\lesssim_{\epsilon}R^{\frac12-\frac1p+\epsilon}(\sum_{\gamma\in \Gamma(R^{-1})}\|\Pc_\gamma F\|^p_{L^p(\R^3)})^{\frac1p}.$$
\end{con}

Testing with $\widehat{F}$ equal to  a smooth approximation of $\Nc_{{\C}o^2}(R^{-1})$ shows that the range $2\le p\le 4$ is optimal for an $l^p(L^p)$ decoupling.
Using flat decoupling (to pass from $\theta$ to $\gamma$), the trilinear restriction theorem and the Bourgain--Guth method \cite{BG}, one can easily verify the conjecture for $2\le p\le 3$. We omit the details.
\smallskip

If proved, this conjecture would have immediate implications for (not necessarily periodic) exponential sums, via Proposition \ref{revHolforexpssums}. See Lemma 2.1 and Lemma 2.2 in \cite{Woo1} for related results in the periodic case.
Similar issues have recently surfaced  in the works \cite{BW} and \cite{BW1} of Bourgain and Watt  on  estimating the averages of the zeta function on short subintervals of the critical line, as well as on getting new estimates on the Gauss circle problem. The key exponential sum estimate involves
$$\|\sum_{l\sim L}\sum_{k\sim K}a_{kl}e(lx_1+klx_2+\omega(k,l)x_3)\|_{L^p(|x_1|,|x_2|<1,|x_3|<(\eta LK^{\frac12})^{-1})}$$
where
\begin{equation}
\label{huhcfiyfiwo[]}
\omega(k,l)=\frac13((k+l)^{\frac32}-(k-l)^{3/2})=k^{1/2}l+ck^{-3/2}l^3+\ldots
\end{equation}
and  $\eta>0$ is a small parameter. Note that the points $(l,kl,k^{\frac12}l)$ lie on the cone $x=\frac{z^2}{y}$, and the points $(l,kl,\omega(k,l))$ lie on a slight perturbation of it. Solving Conjecture \ref{c2} would be a first step towards understanding better this class of problems.

\smallskip

Conjecture \ref{c2} should  be compared with the following  result about boxes with canonical scale.

\begin{te}[\cite{BD}]
	\label{3}
Let $F:\R^3\to\C$ with the Fourier transform supported inside $\Nc_{{\C}o^2}(R^{-1})$. Then for each $2\le p\le 6$ we have
$$\|F\|_{L^p(\R^3)}\lesssim_{\epsilon}R^{\epsilon}(\sum_{\theta\in \Theta_{{\C}o^2}(R^{-1})}\|\Pc_\theta F\|^2_{L^p(\R^3)})^{\frac12}.$$	
\end{te}

We  gauge the difficulty of Conjecture \ref{c2} in a similar way we did with the Conjecture \ref{c1} for the parabola. Combining Theorem \ref{3} with flat decoupling (Proposition \ref{p6}) leads to the less than optimal inequality
$$
\|F\|_{L^p(\R^3)}\lesssim_\epsilon R^{\frac32(\frac12-\frac1p)+\epsilon}(\sum_{\gamma\in\Gamma(R^{-1})}\|\Pc_\gamma F\|_{L^p(\R^3)}^p)^{\frac1p}.
$$
The Lorentz transformations rescale the cone in the direction of nonzero principal curvature, but do not stretch the cone in the direction of zero principal curvature. In other words, the square-like caps $\gamma$ are mapped into caps of  rectangular-like shape. This renders Lorentz rescaling inefficient for attacking Conjecture \ref{c2}.

\section{The new results and the methods of proof}
\label{s2}

Our first result verifies   Conjecture \ref{c1}.

\begin{te}
\label{t9}Assume $F:\R^2\to\C$ has the Fourier transform supported on $\Nc_{\P^1}(R^{-1})$. Then for each $2\le p\le 2+\frac2\alpha$ we have
$$\|F\|_{L^p(\R^2)}\lesssim_\epsilon R^{\alpha(\frac12-\frac1p)+\epsilon}(\sum_{\gamma\in \Gamma_\alpha(R^{-1})}\|\Pc_\gamma F\|^p_{L^p(\R^2)})^{\frac1p}.$$	
\end{te}
As an immediate consequence, we get the following essentially sharp exponential sum estimate for frequency points that do not necessarily belong to a lattice.

\begin{co}
	\label{t7}	
	Let $\frac12\le \alpha\le 1$ and $2\le p\le 2+\frac2\alpha$. The inequality
	$$(\frac1{R^2}\int_{Q_R}|\sum_{\xi\in\Xi}a_\xi e(x\cdot (\xi,\xi^2))|^{p}dx)^{\frac1p}\lesssim_\epsilon R^{\frac\alpha2+\epsilon}$$
	holds true for each collection $\Xi\subset [-1,1]$ with $|\Xi|\sim R^\alpha$ consisting  of $R^{-\alpha}$-separated points, each square $Q_R\subset \R^2$ with diameter $R$ and  each $a_\xi\in\C$ with magnitude $\sim 1$.
\end{co}

\medskip	

We also verify Conjecture \ref{confgght} for $n=3$ in the range $0\le \beta\le \frac32$. As remarked earlier, for each $\beta$ it suffices to consider the critical exponent $s_c=6-\beta$.	Our method allows for a more generous result, with unit modulus coefficients and arbitrary intervals $H$ of length $\frac1{N^{\beta}}$.
\begin{te}
	\label{n=3partialrange}	
	For each $0\le \beta\le \frac32$, each interval $H$ of length $\frac1{N^{\beta}}$ and each $a_j\in\C$ with $|a_j|=1$ we have
	$$\int_{[0,1]^{2}\times H}|\sum_{k=1}^Na_je(kx_1+k^2x_2+k^3x_3)|^{12-2\beta}dx\lesssim_\epsilon N^{6-2\beta+\epsilon}. $$	
\end{te}	
We did not make serious efforts to extend this result to the full range $0\le \beta\le 2$.   The Appendix contains an application of the preceding theorem to the Riemann zeta-function.	
\begin{re}
The term $k^3$ may be replaced with $N^3\varphi(\frac{k}{N})$, where $\varphi$ is a $C^3([0,1])$ function satisfying $\min_{t\in[0,1]}|\varphi^{'''}(t)|>0$. This is because periodicity in the variable $x_3$ is never used in our argument. 	
\end{re}
\medskip

Regarding Conjecture \ref{c2}, we will verify it by combining our  two-step decoupling method with the following very recent result due to the last two authors and Zhang.

\begin{te}[Reverse square function estimate, \cite{GWZ}]
\label{c3}	
Assume $F:\R^3\to \C$ has Fourier transform supported on $\Nc_{{\C}o^2}(R^{-1})$. Then
$$\|F\|_{L^4(\R^3)}\lesssim_\epsilon R^{\epsilon}\|(\sum_{\theta\in \Theta_{{\C}o^2}(R^{-1})}|\Pc_\theta F|^2)^{\frac12}\|_{L^4(\R^3)}.$$	
\end{te}
More precisely, in Section \ref{cone} we prove the following result.
\begin{te}
\label{4}
Theorem \ref{c3} implies Conjecture \ref{c2}.
\end{te}
We have the following immediate consequence.
\begin{co}
Let $\Lambda$ be a collection of $\delta$-separated points on the cone. For each $\lambda\in \Lambda$, let $a_\lambda$ have magnitude $\sim 1$. Then for each cube $Q_R$ with diameter $R\ge \delta^{-2}$ we have
$$(\frac1{R^3}\int_{Q_R}|\sum_{\lambda\in\Lambda}a_\lambda e(x\cdot\lambda)|^4dx)^{\frac{1}4}\lesssim_\epsilon \delta^{-1-\epsilon}.$$
In particular, if $\Lambda$ has maximum size $|\Lambda|\sim \delta^{-2}$, then (by letting $R\to\infty$) we find the following sharp estimate on the additive energy of $\Lambda$
$$\E_2(\Lambda)=|\{(\lambda_1,\ldots,\lambda_4)\in\Lambda^4:\;\lambda_1+\lambda_2=\lambda_3+\lambda_4\}|\lesssim_\epsilon |\Lambda|^{2+\epsilon}.$$
\end{co}

\medskip

Let us close this section by commenting on the methods used to prove Theorem \ref{t9}, Theorem \ref{n=3partialrange}	 and Theorem \ref{4}.

The naive, failed argument we have described in the previous section  for both the parabola and the cone relied on  combining decoupling at the canonical scale with flat decoupling. We will instead apply an improved version of this argument, by proving and eventually combining refined versions of these two types of decoupling. The refinements will reflect a gain over a certain parameter that counts the statistics of thin boxes inside fat boxes.

A first example of refined flat decoupling is presented in Section \ref{reftri}, that only relies on $L^2$ orthogonality.
This is good enough for the parabola and the cone. The moment curve will require more sophisticated  versions of this principle, that will combine $L^2$ orthogonality with lower dimensional decoupling. These  are proved in  Propositions \ref{ccc9} and \ref{ddd9} and rely in part on the new small cap decoupling we prove here for the parabola.

The refined decoupling for boxes of canonical scale will rely on new Kakeya-type information for boxes (tubes, plates and planks) that satisfy certain structural assumptions. In the case of the parabola, the refined Kakeya estimates we need are in the spirit of those from \cite{GSW}. In this case, the structural assumption amounts to control over the (upper)  density of the thin tubes with respect to  certain fat tubes. We refer to this as ``statistical" assumption. To prove Theorem  \ref{n=3partialrange} we derive new incidence estimates  for plates and planks that are adapted to the  geometry of the moment curve. In this case we rely both on statistical and on periodicity assumptions for our boxes, the latter being inherited from the periodicity of the exponential sum in the first two variables. In the case of the cone, the required Kakeya-type input for planks is rather standard, but this convenience is made possible by the use of the powerful reverse square function estimate (Theorem \ref{c3}). It is worth pointing out that the proof of this square function estimate relies on more delicate geometric localization for planks.

Our arguments rely substantially on wave packet decompositions, more so than the previous work on decoupling. They facilitate the reduction of oscillatory problems to incidence geometric estimates between  boxes and small scale-cubes.  The local analysis on each such cube will be handled using two mechanisms: refined decoupling (see Theorems \ref{t4} and \ref{ccc24})  and multilinear restriction estimates (the classical bilinear $L^4$  Cordoba inequality for the parabola and the analogous trilinear $L^6$ estimate for the twisted cubic).

\section{A refined  flat decoupling}
\label{reftri}
In this section we will show how to improve the flat decoupling in Proposition \ref{p6}, subject to a statistical  assumption on the wave packets of $F$. Small variations of the next result will be used later for all three conjectures mentioned in the previous section.

Two rectangular boxes are said to be dual to each other if they share the axes, and if their corresponding edges have reciprocal lengths.

\begin{pr}
\label{p2}
Let $B$ be an almost rectangular box in $\R^n$, and let $B_1,\ldots, B_L$ be a partition of $B$ into  almost  rectangular boxes of volume $\sim V$ which are (essentially) translates of each other.  Let $\Tc$ be a tiling of $\R^n$ with rectangular boxes $\tau$ dual to $B_i$. Then for each $2\le p\le \infty$ we have
$$ (\sum_{\tau\in\Tc}\|\Pc_BF\|_{L^2(\tau)}^p)^{\frac1p}\lesssim(\frac{L}V)^{\frac12-\frac1p} (\sum_{i=1}^L\|\Pc_{B_i}F\|_{L^p(\R^n)}^p)^{\frac1p}.$$
\end{pr}
\begin{proof}
Let $\phi_\tau$ be a positive-valued smooth approximation of $1_\tau$ such that
$$\sum_{\tau\in\Tc}(\phi_\tau)^p\sim 1_{\R^n},\;\;
\supp(\widehat{\phi_\tau})\subset B_0\;\;\text{
and }
\phi_\tau\ge 1_\tau.$$
We use H\"older's inequality twice to write
\begin{align*}
(\sum_{i=1}^L\|\Pc_{B_i}F\|_{L^p(\R^n)}^p)^{\frac1p}&\gtrsim(\sum_{\tau\in\Tc}\sum_{i=1}^L\|\phi_\tau\Pc_{B_i}F\|_{L^p(\R^n)}^p)^{\frac1p}\\&\gtrsim V^{\frac12-\frac1p}(\sum_{\tau\in\Tc}\sum_{i=1}^L\|\phi_\tau^2\Pc_{B_i}F\|_{L^2(\R^n)}^p)^{\frac1p}\\&\ge (\frac{V}{L})^{\frac12-\frac1p}	(\sum_{\tau\in\Tc}(\sum_{i=1}^L\|\phi_\tau^2\Pc_{B_i}F\|_{L^2(\R^n)}^2)^{\frac{p}2})^{\frac1p}.
\end{align*}	
Next, note that for each fixed $\tau$ the functions $\phi_\tau^2\Pc_{B_i}F$ are almost orthogonal. Thus, the last expression is
$$\gtrsim (\frac{V}{L})^{\frac12-\frac1p}(\sum_{\tau\in\Tc}\|\Pc_BF\|_{L^2(\tau)}^p)^{\frac1p},$$
as needed.

\end{proof}

For later use, we rewrite this proposition in a slightly different way, that shows  a gain of $N^{\frac12-\frac1p}$ over the flat $l^p(L^p)$ decoupling bound in Proposition \ref{p6}.
\smallskip

Let $$\chi(x)=\frac1{(1+|x|)^{100n}}.$$
For a box $\tau$ we denote by $\chi_\tau$ the  $L^\infty$ rescaled version of $\chi$ adapted to $\tau$.

Consider the wave packet decomposition
\begin{equation}
\label{wapadec}
\Pc_BF=\sum_{T\in\T_B} w_TW_T
\end{equation}
arising as follows. 
Let $\eta_B$ be a  real Schwartz function satisfying  $1_B\le  \eta_B\le 1_{2B}$. Consider the Fourier expansion of $\widehat{\Pc_BF}$ on $2B$ and note that
$$\widehat{\Pc_B F}=|B|^{-1}\sum_{T\in\T_B}\langle \widehat{\Pc_BF}, e(c_T\;\cdot)\rangle e(c_T\;\cdot)\eta_B.$$
The collection $\T_B$ represents a tiling of $\R^n$ with boxes $T$ centered at $c_T$ and dual to $2B$. We define
$$W_T(x)=\frac1{|B|}\widehat{\eta_B}(x-c_T)$$
and $w_T=\langle \widehat{\Pc_BF}, e(c_T\;\cdot)\rangle.$

The function $W_T$ is an $L^{\infty}$ normalized smooth approximation of $1_T$ with the Fourier support inside a slight enlargement of $B$. In particular, $|W_T|\lesssim \chi_T$ and $\|W_T\|_{L^p(\lambda T)}\sim |T|^{1/p}$ for each $\lambda\ge 1$.  If $|w_T|\sim w$ for $T\in\T_{B}'\subset \T_B$, then 
$$\|\sum_{T\in\T_B'}w_TW_{T}\|_{L^p(\R^n)}\sim (\sum_{T\in\T_B'}\|w_TW_{T}\|_{L^p(\R^n)}^p)^{1/p}\sim w(\frac{|\T_B'|}{|B|})^{1/p}.$$
  See Chapter 2, especially Exercise 2.7, in \cite{Dembook}.

We next use $\lessapprox$  to hide arbitrarily small power losses with respect to the parameter $N$.

\begin{co}
\label{8}	
Let $\T_B'\subset \T_B$ be such that $|w_{T}|\sim w$ for $T\in\T_{B}'$
and such that each $\tau\in\Tc$ contains either $\sim N$ tubes $T$ or no tubes at all. Then
\begin{equation}
\label{xx1}
\|\sum_{T\in\T_B'}w_TW_{T}\|_{L^p(\R^n)}\lessapprox(\frac{L^2}N)^{\frac12-\frac1p} (\sum_{i=1}^L\|\Pc_{B_i}F\|_{L^p(\R^n)}^p)^{\frac1p}.
\end{equation}

\end{co}
\begin{proof}

For $\epsilon>0$ we split $\T_B'$ into $O(N^{2\epsilon})$ collections $\T_B''$ as follows. If $\tau\in\Tc$ contributes to $\T_B'$, then all $T\subset \tau$ will be placed in the same collection $\T_B''$. Also, if $\tau,\tau'$ contribute to $\T_B''$ then $N^\epsilon\tau$ and $N^\epsilon \tau'$ are disjoint. It suffices to prove \eqref{xx1} with $\T_{B}'$ replaced with a collection $\T_B''$.

  Let $\Tc'$ be the collection of those $\tau\in\Tc$ that contribute to $\T_B''$. Thus $|\T_B''|\sim N|\Tc'|$.
It is easy to see that for each $p$ we have
$$\|\sum_{T\in\T_B''}w_TW_{T}\|_{L^p(N^{\epsilon/2}\tau)}\sim  w(\frac {N}{LV})^{\frac1p}$$
if $\tau\in \Tc'$. Indeed, the left hand side differs from  $\|\sum_{T\in\T_B''\atop{T\subset \tau}}w_TW_{T}\|_{L^p(\R^n)}$ by a negligible term, due to Schwartz tail considerations. On the other hand, this term is $\sim  w(\frac {N}{LV})^{\frac1p}$, since $|w_T|\sim w$. Also, 
$$\|\sum_{T\in\T_B''}w_TW_{T}\|_{L^p(\R^n)}\sim w(\frac {N|\Tc'|}{LV})^{\frac1p}.$$

Thus
$$ \|\sum_{T\in\T_B''}w_TW_{T}\|_{L^p(\R^n)}\sim (\sum_{\tau\in\Tc'}\|\sum_{T\in\T_B''}w_TW_{T}\|_{L^p(N^{\epsilon/2}\tau)}^p)^{1/p}\sim $$$$(\frac{LV}N)^{\frac12-\frac1p}(\sum_{\tau\in\Tc'}\|\sum_{T\in\T_B''}w_TW_{T}\|_{L^2(N^{\epsilon/2}\tau)}^p)^{\frac1p}.$$
We may write	
\begin{align*}\|\sum_{T\in\T_B''}w_TW_{T}\|_{L^p(\R^n)}&\lesssim (\frac{LV}{N})^{\frac12-\frac1p} (\sum_{\tau\in\Tc'}\|\sum_{T\in\T_B''}w_TW_{T}\|_{L^2(N^{\epsilon/2}\tau)}^p)^{\frac1p}\\&\lesssim N^{O(\epsilon)}(\frac{LV}{N})^{\frac12-\frac1p} (\sum_{\tau\in\Tc'}\|\sum_{T\in\T_B}w_TW_{T}\|_{L^2(\chi_\tau)}^p)^{\frac1p}\;\;\;\;\text{(by $L^2$ orthogonality)}\\&\sim  N^{O(\epsilon)}(\frac{LV}{N})^{\frac12-\frac1p} (\sum_{\tau\in\Tc}\|\sum_{T\in\T_B}w_TW_{T}\|_{L^2(\tau)}^p)^{\frac1p}
\\&\lesssim N^{O(\epsilon)}(\frac{L^2}N)^{\frac12-\frac1p} (\sum_{i=1}^L\|\Pc_{B_i}F\|_{L^p(\R^n)}^p)^{\frac1p}. \;\;\;\;\text{(by Proposition \ref{p2})}
\end{align*}
\end{proof}

\section{The proof of Theorem \ref{t9}}

We split the proof into several smaller steps.

\subsection{An initial bilinear reduction for Theorem \ref{t9}}
\label{btol}
\smallskip

Let $\Theta_1(R^{-1})$ and $\Theta_2(R^{-1})$ be the subsets of $\Theta(R^{-1})$ consisting of boxes $\theta$ sitting above  $[-1,-\frac12]$ and  $[\frac12,1]$, respectively. The only important thing about $[-1,-\frac12]$ and  $[\frac12,1]$ is that they are disjoint.
\medskip

Write for $i\in\{1,2\}$

$$F_i=\sum_{\theta\in \Theta_i(R^{-1})}\Pc_{\theta}F.$$

Here we show that Theorem \ref{t9} is a consequence of the following bilinear result.

\begin{te}
\label{t1}	If $2\le p\le 2+\frac2{\alpha}$ and if $\widehat{F}$ is supported on  $\Nc_{\P^1}(R^{-1})$ then
$$\|(F_1F_2)^{\frac12}\|_{L^p(\R^2)}\lesssim_\epsilon R^{\alpha(\frac12-\frac1p)+\epsilon}(\sum_{\gamma\in\Gamma_\alpha(R^{-1})}\|\Pc_\gamma F\|^p_{L^p(\R^2)})^{\frac1p}.$$
\end{te}

\begin{proof}(Theorem \ref{t1} implies Theorem \ref{t9}) This is the only part in  the proof where we use parabolic rescaling. This tool will not be used in the proof of Theorem \ref{t1}.
\smallskip

	Let $K\sim \log R$ and let $m$ be such that $K^m=R^{1/2}$. We denote by $\vartheta$ a set of the form $\Nc_I(R^{-1})$, with $I$ a dyadic interval in $[-1,1]$. We write $l(\vartheta)=|I|$.
The starting point is the following elementary inequality, with $C$ a universal constant (such as 100)

\begin{align*}|F(x)|&\le \sum_{l(\vartheta)=\frac1K} |\Pc_\vartheta F(x)| \\&\le C\max_{l(\vartheta)=\frac1K} |\Pc_\vartheta F(x)|+K^C\max_{l(\vartheta_1)=l(\vartheta_2)=\frac1K\atop{ \dist(\vartheta_1,\vartheta_2)\gtrsim \frac1K}}|\Pc_{\vartheta_1} F(x)\Pc_{\vartheta_2} F(x)|^{1/2}.
\end{align*}	
If we iterate this $m$ times (always for the first term) and integrate, we find (with a different, still universal $C$)
\begin{align*}
\|F\|_{L^p(\R^2)}^p&\lesssim C^{m}\sum_{l(\vartheta_0)=R^{-1/2}}\|\Pc_{\vartheta_0} F\|^p_{L^p(\R^2)}\\&+C^mK^C\sum_{\frac1{R^{1/2}}\lesssim \Delta\lesssim 1\atop{\Delta\in K^{\Z}}}\sum_{I:\;|I|\sim \Delta}\max_{\vartheta_1,\vartheta_2\subset \Nc_I(\frac1R)\atop{ l(\vartheta_1)=l(\vartheta_2)=K^{-1}\Delta\atop{\dist(\vartheta_1,\vartheta_2)\gtrsim K^{-1}\Delta}}}\|(\Pc_{\vartheta_1} F\Pc_{\vartheta_2} F)^{1/2}\|_{L^p(\R^n)}^p.
\end{align*}
Note first that $C^mK^C\lesssim_\epsilon R^{\epsilon}$, for each $\epsilon>0$.

We estimate each term from the first sum using flat decoupling  (Proposition \ref{p6})
$$\|\Pc_{\vartheta_0} F\|^p_{L^p(\R^2)}\lesssim R^{(2\alpha-1)(\frac{p}{2}-1)}\sum_{\gamma\in\Gamma_\alpha(R^{-1})\atop{\gamma\subset \vartheta_0}}\|\Pc_{\gamma} F\|_{L^p(\R^2)}^p.$$
The contribution from all these terms is acceptable, since $2\alpha-1\le \alpha$.

Let us now analyze a term from the second sum. Note that $R\Delta^2\ge 1$. We use parabolic rescaling, mapping $\Nc_{I}(\frac1R)$ to $\Nc_{\P^1}(\frac1{R\Delta^2})$, and $F$ to a function $G$ with spectrum inside $\Nc_{\P^1}(\frac1{R\Delta^2})$. We apply Theorem \ref{t1} to this $G$, with $R$ replaced by $R\Delta^2$, and then reinterpret this inequality for $F$, via a change of variables. We get
$$\|(\Pc_{\vartheta_1} F\Pc_{\vartheta_2} F)^{1/2}\|_{L^p(\R^n)}\lesssim_\epsilon (R\Delta^2)^{\alpha(\frac12-\frac1p)+\epsilon}(\sum_{\omega\subset \Nc_{I}(\frac1R)\atop{l(\omega)=\frac{\Delta}{(R\Delta^2)^\alpha}}}\|\Pc_\omega F\|^p_{L^p(\R^2)})^{\frac1p}.$$
Note that the tubes $\omega$ are essentially flat, and are longer than the tubes $\gamma\in\Gamma_\alpha(R^{-1})$. We apply flat decoupling (Proposition \ref{p6})
$$\|\Pc_\omega F\|_{L^p(\R^2)}\lesssim (\frac{l(\omega)}{l(\gamma)})^{2(\frac12-\frac1p)}(\sum_{\gamma\in \Gamma_\alpha(\frac1R)\atop{\gamma\subset \omega}}\|\Pc_\gamma F\|^p_{L^p(\R^2)})^{\frac1p}.$$	
If we combine the last four displayed inequalities, we finish the proof as follows
\begin{align*}
\|F\|_{L^p(\R^2)}&\lesssim_\epsilon R^{\alpha(\frac12-\frac1p)+\epsilon}(\sum_{\frac1{R^{1/2}}\lesssim \Delta\lesssim 1\atop{\Delta\in K^\Z}}\Delta^{(p-2)(1-\alpha)})^{1/p}(\sum_{\gamma\in\Gamma_\alpha(R^{-1})}\|\Pc_\gamma F\|^p_{L^p(\R^2)})^{1/p}\\&\lesssim_\epsilon R^{\alpha(\frac12-\frac1p)+\epsilon}(\sum_{\gamma\in\Gamma_\alpha(R^{-1})}\|\Pc_\gamma F\|^p_{L^p(\R^2)})^{1/p}.
\end{align*}

\end{proof}	

We will next focus on proving Theorem \ref{t1}.

\subsection{A refined planar  Kakeya inequality}

We first review two  classical estimates that we will find useful in the next sections. The notation $\lessapprox$ will hide logarithmic losses. 
\smallskip

We call a family $\T$ of congruent tubes in $\R^n$ $\delta$-separated, if the collection of their directions forms a $\delta$-separated set on $\S^{n-1}$,
and if any two parallel tubes are disjoint.

\begin{pr}[Linear Kakeya]
	\label{E6}
	Consider a finite collection $\T$ of $\delta$-separated  congruent tubes (rectangles) in $\R^2$ with eccentricity $\delta^{-1}$ and with at most $m$ tubes in each direction. Then
	$$\|\sum_{T\in\T}1_T\|_2\lesssim \log(\delta^{-1})^{\frac12} m^{\frac12}(\sum_{T\in\T}|T|)^{\frac12}.$$
	
\end{pr}
\begin{proof}
	The angle between any two  rectangles $T_1,T_2\in\T$ is   $\sim j\delta$, for some $0\le j\lesssim \delta^{-1}$. It is easy to see that in this case we have the estimate  $$|T_1\cap T_2|\lesssim \frac{|T_1|}{j+1}.$$ We may thus write
	$$\sum_{T_1\in\T}\sum_{T_2\in\T}|T_1\cap T_2|\lesssim m\sum_{T_1\in\T}\sum_{j=1}^{O(\delta^{-1})}j^{-1}|T_1|\lesssim m\log(\delta^{-1})\sum_{T_1\in\T}|T_1|. $$
\end{proof}

The proof of the following result is immediate, and will be omitted.
\begin{pr}[Bilinear Kakeya]
\label{vvv7}
Let $\T_1,\T_2$ be two families of congruent tubes in the plane with eccentricity $\delta^{-1}$, so that the angle between any pair $(T_1,T_2)\in\T_1\times\T_2$ is $\gtrsim 1$. We allow  both $\T_1$ and $\T_2$ to contain multiple copies of a given tube.  Then
$$\|(\sum_{T_1\in\T_1}1_{T_1}\sum_{T_2\in\T_2}1_{T_2})^{1/2}\|_2\lesssim \delta(|\T_1||\T_2|)^{1/2}.$$
In particular, the collection $\Qc_{r_1,r_2}(\T_1,\T_2)$ of bilinear $(r_1,r_2)$-rich $\delta$-squares (those squares intersecting at least $r_1$ tubes from $\T_1$ and at least $r_2$ tubes from $\T_2$) satisfies the bound
\begin{equation}
\label{euyf7yr8fur9fu98u}
|\Qc_{r_1.r_2}(\T_1,\T_2)|\lesssim \frac{|\T_1||\T_2|}{r_1r_2}.
\end{equation}
\end{pr}

The critical new  input from incidence geometry  is provided by the following  refined Kakeya estimate, in the spirit of \cite{GSW}.

\begin{te}
\label{12}	
Let $\T$ be a collection of $R^{-1/2}$-separated   $(R^{\frac12},R)$-tubes in $\R^2$. Assume that the following statistics assumption is satisfied: there are at most $N$ parallel tubes  inside each fat $(R^\alpha,R)$-tube $\tau$ with the same orientation.

Let $r\ge 1$. Let $\Qc_r$ be a collection of pairwise disjoint squares $q$ with side length $\sim R^{1/2}$ that intersect at least $r$ tubes $T\in\T$. Then there is a dyadic scale $1\le s\le R^{1/2}$ and an integer $M_s$ such that the following properties hold:

\begin{equation}
\label{ur8438t754tuijgjotgj}
|\Qc_r|\les\frac{|\T|M_sR^{\frac12}}{sr^2},
\end{equation}

\begin{equation}
\label{ur8438t754tuijgjotgj2}
r\les \frac{M_sR^{\frac12}}{s^2}
\end{equation}
and
\begin{equation}
\label{ur8438t754tuijgjotgj3}
M_s\lesssim Ns\max(1,sR^{\frac12-\alpha}).
\end{equation}	
\end{te}
\bigskip

Before we prove this result, we put it into perspective by presenting  an immediate consequence.

\begin{co}
\label{sdgesyhfioewyrt7yeidcyre7}
Assume $\T$ lies in $[-R,R]^2$ and satisfies the statistics assumption in Theorem \ref{12}. Write $W=R^{1-\alpha}$ and write $|\T_{max}|$ for the maximum possible size $\sim NWR^{\frac12}$ of such a collection $\T$.

Assume also that $r\ge C(\epsilon)R^\epsilon NW$, for some $\epsilon>0$ and some large enough $C(\epsilon)$. Then
\begin{equation}
\label{ewfpiou589tu85u}
|\Qc_r|\les \frac{|\T||\T_{max}|}{r^2W}.
\end{equation}

\end{co}
\begin{proof}
We apply Theorem \ref{12} to get a scale $s$. We claim that the lower bound on $r$ forces $s< R^{\alpha-\frac12}$. Indeed, let us assume for contradiction that $s\ge R^{\alpha-\frac12}$. Then \eqref{ur8438t754tuijgjotgj2} and  \eqref{ur8438t754tuijgjotgj3} lead to
$$C(\epsilon)R^\epsilon NR^{1-\alpha}\le r\le (\log R)^{C_1}\frac{M_sR^{\frac12}}{s^2}\le C_2(\log R)^{C_1}NR^{1-\alpha},$$
or $C(\epsilon)R^\epsilon\le C_2(\log R)^{C_1}$. This is of course false for all $R\ge 1$, if $C(\epsilon)$ is chosen large enough.

Inequality \eqref{ewfpiou589tu85u} follows immediately by combining \eqref{ur8438t754tuijgjotgj} with \eqref{ur8438t754tuijgjotgj3}.

\end{proof}
\bigskip

When $\T$ has maximal size $|\T|\sim NWR^{\frac12}$, this corollary coincides with Theorem 1.2 in \cite{GSW}. For our application to the parabola, we need this slightly more general version, that accommodates  collections $\T$ with smaller size.  One should also compare \eqref{ewfpiou589tu85u} with \eqref{euyf7yr8fur9fu98u}.

\bigskip

\begin{proof}(of Theorem \ref{12})
The proof will show that the statistics assumption in $N$ will not be needed for the proof of either \eqref{ur8438t754tuijgjotgj} or 	\eqref{ur8438t754tuijgjotgj2}.

For each $T$ we let $\upsilon_T$ be a positive smooth approximation of $1_T$,  with the  Fourier transform supported on the dual box to $T$ through the origin.  	
Write
$$K(t)=\sum_{T\in \T}\upsilon_T(t).$$
For each dyadic $1\le s\le R^{1/2}$, let $\eta_s$ be such that $\widehat{\eta_s}$ is a smooth bump on  ${|\xi|\sim (sR^{\frac12})^{-1}}$ if $s< R^{1/2}$, and on  ${|\xi|\lesssim R^{-1}}$ if $s=R^{1/2}$. Moreover, we ask that
$$\sum_s \widehat{\eta_s}(\xi)\equiv 1$$
on the support $|\xi|\le R^{-1/2}$ of $\widehat{K}$.

Note that
$$K=\sum_sK*\eta_s.$$
For each $s$ we consider a maximum collection of $sR^{-1/2}$-separated directions. For each direction, we tile the plane with $(sR^{1/2},R)$-tubes $T_s$ pointing in this direction. Essentially, each $T\in\T$ fits  inside a unique such $T_s$, and we will write $T\subset T_s$.  Each given $T_s$ may contain as many as $\sim s$ tubes $T$ pointing in different directions, and a maximum of $\sim s^2$ tubes from $\T$.

For a dyadic parameter $1\le m\lesssim s^2$, we call $\Tc_{s,m}$ the collection of those $T_s$ that contain $\sim m$ tubes $T\in \T$. We write
$$K_{s,m}=(\sum_{T_s\in\Tc_{s,m}}\sum_{T\subset T_s}\upsilon_T)*\eta_s.$$
Note that
$$K=\sum_s\sum_{m}K_{s,m}.$$
We may pick $s$ and some $m$ --that we denote by $M_s$-- such that
$$K(x)\les |K_{s,M_s}(x)|$$
for all $x$ in a subset $E$  of  $\cup_{q\in\Qc_r}q$ with comparable area $ R|\Qc_r|\les |E|$.
\smallskip

Write $f_{T_s}=(\sum_{T\subset T_s}\upsilon_T)*\eta_s$.
Here is the key observation. Due to both space and frequency support considerations, it is easy to see that for each $s$ and for each pair of  fat tubes $T_s,T_{s}'$, the functions
$f_{T_s}$ and $f_{T_s'}$ are almost orthogonal.

Each function $f_{T_s}$ is essentially supported on (a slight enlargement of) $T_s$, and moreover
\begin{equation}
\label{ur8438t754tuijgjotgj4}
\|f_{T_s}\|_{\infty}\lesssim \frac{M_s}{s}.
\end{equation}

Let us now derive a few consequences.
First, since $K$ is essentially constant on $\cup_{q\in\Qc_r}q$, we have
\begin{align*}
R|\Qc_r|r^2&\lesssim \int_{E}|K|^2\les \int_{\R^2}|K_{s,M_s}|^2\\&\lesssim \sum_{T_s\in\Tc_{s,M_s}}\int_{\R^2}|f_{T_s}|^2\\&\lesssim \sum_{T_s\in\Tc_{s,M_s}}|T_s|(\frac{M_s}{s})^2\\&\lesssim \frac{M_s|\T|R^{3/2}}{s}.
\end{align*}
This proves \eqref{ur8438t754tuijgjotgj}. Let us next see why  \eqref{ur8438t754tuijgjotgj2} holds. Write for some $x\in E$, noting that there are $\lesssim \frac{R^{1/2}}s$ tubes $T_s$ passing through $x,$ and using \eqref{ur8438t754tuijgjotgj4} in the end
$$r\lesssim K(x)\les |K_{s,M_s}(x)|\le \sum_{T_s\in\Tc_{s,M_s}}|f_{T_s}(x)|\les \frac{R^{1/2}}{s}\frac{M_s}{s}.$$

We separate the proof of \eqref{ur8438t754tuijgjotgj3} into two cases.

 Let us start with $s\le R^{\alpha-\frac12}$. Pick $T_s\in\Tc_{s,M_s}$. By our hypothesis, $T_s$ contains at most $N$ tubes $T\in\T$ of each given direction. There are $\lesssim s$ possible directions for these tubes, so $T_s$ can contain at most $Ns$ tubes $T$. We conclude that $M_s\lesssim Ns$.

We next assume that $s\ge R^{\alpha-\frac12}$. Pick $T_s\in\Tc_{s,M_s}$. There are $\sim \frac{R^{\frac12}s}{R^\alpha}$ fat $(R^\alpha,R)$-tubes $\tau$ inside $T_s$, with the same orientation as $T_s$. Our hypothesis implies that $T_s$ contains at most $NR^{\frac12-\alpha}s$ tubes $T\in\T$ of each given direction. There are $\lesssim s$ possible directions for these tubes, so $T_s$ can contain at most $NR^{\frac12-\alpha}s^2$ tubes $T$. We conclude that $M_s\lesssim NR^{\frac12-\alpha}s^2$.

\end{proof}

\smallskip

\subsection{A refined $l^p$ decoupling for boxes of canonical scale}

	Assume $F:\R^2\to\C$, with $\widehat{F}$ supported on $\Nc_{\P^1}(R^{-1})$. Consider the wave packet decomposition (see \eqref{wapadec})
$$F=\sum_{\theta\in\Theta(R^{-1})}\Pc_\theta F=
\sum_{T\in\T_R(F)}w_TW_T.$$
The family $\T_R(F)$ contains the tubes corresponding to all boxes $\theta$.  We will write $F_T=w_TW_T$.
\medskip

The following result is a particular case of Theorem 4.2 proved in \cite{GIOW}. It refines the $l^p(L^p)$ decoupling \eqref{e5} by replacing $R^{\frac14-\frac1{2p}}$ with the smaller quantity $M^{\frac12-\frac1p}$.
\begin{te}
	\label{t4}

	Let $\Qc$ be a collection of pairwise disjoint squares $q$ in $\R^2$, with side length $R^{1/2}$.
	Assume that each $q$ intersects at most $M$ fat tubes $R^\delta T$ with $T\in\T_R(F)$, for some $M\ge 1$ and $\delta>0$.
	
	Then for each $2\le p\le 6$ and $\epsilon>0$ we have	
	
	$$\|F\|_{L^p(\cup_{q\in\Qc}q)}\lesssim_{\delta,\epsilon}R^{\epsilon}M^{\frac12-\frac1p}(\sum_{T\in\T_R(F)}\|F_T\|^p_{L^p(\R^2)})^{\frac1p}.$$
\end{te}

\subsection{Proof of Theorem \ref{t1}}
\bigskip

For $i\in\{1,2\}$, recall the definition of $F_i$, and let  $\T_i\subset \T_R(F)$ be such that
$$F_i=\sum_{T\in\T_i}F_T=\sum_{T\in\T_i}w_TW_T.$$
The directions of $T_1\in\T_1$ and $T_2\in\T_2$ are $\sim 1$-separated.
We split the tubes into $O(\log R)$ many significant collections  with $|w_T|\sim constant$ for each $T$ within each collection. Accordingly, each $F_i$ may be written as the sum of functions $F_{i}^{(k)}$ with $k\lesssim \log R$, and a small error term whose contribution is negligible. 
\smallskip

For each $\theta\in\Theta_1(R^{-1})\cup \Theta_2(R^{-1})$ we tile $\R^2$ with $(R^\alpha,R)$-tubes $\tau$ having the same orientation.
Each $T\in\T_1\cup\T_2$ lies inside a unique  $\tau$, with the same orientation. Each $\tau$ is naturally associated with some $i$. Using another pigeonholing, we may restrict attention to those $\tau$ containing roughly  $N_i$  tubes $T\in\T_i$ inside that are parallel to $\tau$, for some dyadic numbers $N_1,N_2$. We only consider the corresponding tubes $T$ lying inside such $\tau$. This gives rise to a further decomposition of each $F_i^{(k)}$ as a sum of $O(\log R)$ many functions $F_i^{k,N_i}$ associated with families of tubes $\T_i^{k,N_i}$, for each dyadic parameter $N_i$ as described above. It suffices to prove that for each $k_i$ and $N_i$ we have
\begin{equation}\label{fgygcoeioifjiochfuvhj}
\|(F_1^{k_1,N_1}F_2^{k_2,N_2})^{\frac12}\|_{L^p(\R^2)}\lesssim_\epsilon R^{\alpha(\frac12-\frac1p)+\epsilon}(\sum_{\gamma\in\Gamma_\alpha(R^{-1})}\|\Pc_\gamma F\|^p_{L^p(\R^2)})^{\frac1p}.
\end{equation}
Note that the right hand side is independent of $k_1,k_2,N_1,N_2$.

For the rest of the argument, let us fix $k_1,k_2,N_1,N_2$. To ease notation, we will continue to call $\T_i$ the families of the restricted tubes  $\T_i^{k_i,N_i}$, and will call $F_i$ the restricted  functions $F_i^{k_i,N_i}$. Using scaling we may also assume that $|w_T|\sim 1$ for each $T\in\T_1\cup\T_2$.
\smallskip

We begin by restating Corollary \ref{8} for $F_1$ and $F_2$. This is the first step in our two-step decoupling approach. We point out an important subtlety. Corollary \ref{8} allows us to use the original function $F$ on the right hand side of the inequality below.
\begin{co} \label{8restated}
For $i\in\{1,2\}$ and $p\ge 2$
$$(\sum_{\theta\in \Theta_i(R^{-1})}\|\Pc_\theta F_i\|^p_{L^p(\R^2)})^{\frac1{p}}\lesssim (\frac{R^{2\alpha-1}}{N_i})^{\frac12-\frac1p}(\sum_{\gamma\in \Gamma_\alpha(R^{-1})}\|\Pc_\gamma F\|^p_{L^p(\R^2)})^{\frac1{p}}.$$
\end{co}
\medskip

We now describe the second step.
For each dyadic $1\le r_1,r_2\lesssim R^{1/2}$, we let $\Qc_{r_1,r_2}$ be the collection of all dyadic  squares $q$ with side length $R^{1/2}$, whose slight enlargements intersect $\sim r_1$ tubes $T\in\T_1$ and $\sim r_2$ tubes $T\in\T_2$.

Write $$S_{r_1,r_2}=\bigcup_{q\in \Qc_{r_1,r_2}}q.$$
\smallskip

The next result handles the contribution from $S_{r_1,r_2}$.
\begin{te}
	\label{t5}
	We have for each $4\le p\le 6$
	$$\|(F_1F_2)^{\frac12}\|_{L^p(S_{r_1,r_2})}$$$$\lesssim_\epsilon (\frac{R^{-\frac12+\epsilon}|\Qc_{r_1,r_2}|}{|\T_1|^{1/2}|\T_2|^{1/2}})^{\frac3p-\frac12}(r_1r_2)^{\frac1p}(\sum_{\theta\in \Theta_1(R^{-1})}\|\Pc_\theta F_1\|^p_{L^p(\R^2)})^{\frac1{2p}}(\sum_{\theta\in \Theta_2(R^{-1})}\|\Pc_\theta F_2\|^p_{L^p(\R^2)})^{\frac1{2p}}.$$	
\end{te}
\begin{proof}
We note that 	
$$(\sum_{\theta\in \Theta_i(R^{-1})}\|\Pc_\theta F_i\|^p_{L^p(\R^2)})^{\frac1{p}}\sim (\sum_{T\in \T_i}\| F_T\|^p_{L^p(\R^2)})^{\frac1{p}}\sim (|\T_i|R^{\frac32})^{\frac1p}.$$		
The inequality we need to prove  will follow by ``interpolating" between $L^4$ and $L^6$, as follows.
\smallskip

First, by the classical Cordoba's inequality  we have for each square $q$ with side length $\sim R^{\frac12}$
\begin{align*}
\|(F_1F_2)^{\frac12}\|_{L^4(q)}&\lesssim \|(\sum_{\theta\in\Theta_1}|\Pc_{\theta}F_1|^2\sum_{\theta\in\Theta_2}|\Pc_{\theta}F_2|^2)^{\frac14}\|_{L^4(\chi_q)}\\&\le \|\sum_{\theta\in\Theta_1}|\Pc_{\theta}F_1|^2\|_{L^2(\chi_q)}^{\frac14}\|\sum_{\theta\in\Theta_2}|\Pc_{\theta}F_2|^2\|_{L^2(\chi_q)}^{\frac14}.
\end{align*}
Summing over all $q\in \Qc_{r_1,r_2}$, using Cauchy--Schwarz  leads to
\begin{align*}
\|(F_1F_2)^{\frac12}\|_{L^4(S_{r_1,r_2})}&\lesssim \|\sum_{\theta\in\Theta_1}|\Pc_{\theta}F_1|^2\|_{L^2(\sum\chi_q)}^{\frac14}\|\sum_{\theta\in\Theta_2}|\Pc_{\theta}F_2|^2\|_{L^2(\sum_{}\chi_q)}^{\frac14}\\&\lesssim  \|\sum_{T\in\T_1}\chi_{T}\|_{L^2(\sum\chi_q)}^{\frac14}\|\sum_{T\in\T_2}\chi_{T}\|_{L^2(\sum_{}\chi_q)}^{\frac14}\\&\lesssim_\epsilon (r_1r_2R^{1+\epsilon}|\Qc_{r_1,r_2}|)^{\frac14}\\&\sim (
\frac{r_1r_2R^{-\frac12+\epsilon}|\Qc_{r_1,r_2}|}{|\T_1|^{1/2}|\T_2|^{1/2}})^{\frac14}(\sum_{\theta\in \Theta_1}\|\Pc_\theta F_1\|^4_{L^4(\R^2)})^{\frac1{8}}(\sum_{\theta\in \Theta_2}\|\Pc_\theta F_2\|^4_{L^4(\R^2)})^{\frac1{8}}.
\end{align*}

Using the refined $l^6L^6$ decoupling (Theorem \ref{t4}) we may write
\begin{align*}
\|(F_1F_2)^{\frac12}\|_{L^6(S_{r_1,r_2})}&\le (\|F_1\|_{L^6(S_{r_1,r_2})}\|F_2\|_{L^6(S_{r_1,r_2})})^{\frac12}\\&\lesssim_\epsilon R^\epsilon (r_1r_2)^{\frac16}(\sum_{\theta\in\Theta_1}\|\Pc_\theta F_1\|_{L^6(\R^2)}^6)^{\frac1{12}} (\sum_{\theta\in\Theta_2}\|\Pc_\theta F_2\|_{L^6(\R^2)}^6)^{\frac1{12}}.
\end{align*}
Write $$\frac1p=\frac{\beta}4+\frac{1-\beta}6;\;\;\;\beta=\frac{12}{p}-2.$$
Combining the last two inequalities we get
$$
\|(F_1F_2)^{\frac12}\|_{L^p(S_{r_1,r_2})}\le\|(F_1F_2)^{\frac12}\|_{L^4(S_{r_1,r_2})}^\beta\|(F_1F_2)^{\frac12}\|_{L^6(S_{r_1,r_2})}^{1-\beta}$$$$\lesssim_\epsilon (
\frac{r_1r_2R^{-\frac12+\epsilon}|\Qc_{r_1,r_2}|}{|\T_1|^{1/2}|\T_2|^{1/2}})^{\frac3p-\frac12}(r_1r_2)^{\frac12-\frac2p}(\sum_{\theta\in \Theta_1}\|\Pc_\theta F_1\|^p_{L^p(\R^2)})^{\frac1{2p}}(\sum_{\theta\in \Theta_2}\|\Pc_\theta F_2\|^p_{L^p(\R^2)})^{\frac1{2p}}
$$
$$\sim(\frac{R^{-\frac12+\epsilon}|\Qc_{r_1,r_2}|}{|\T_1|^{1/2}|\T_2|^{1/2}})^{\frac3p-\frac12}(r_1r_2)^{\frac1p}(\sum_{\theta\in \Theta_1}\|\Pc_\theta F_1\|^p_{L^p(\R^2)})^{\frac1{2p}}(\sum_{\theta\in \Theta_2}\|\Pc_\theta F_2\|^p_{L^p(\R^2)})^{\frac1{2p}}.$$

\end{proof}
\medskip

We combine Corollary \ref{8restated} and Theorem \ref{t5} to write
$$
\|(F_1F_2)^{\frac12}\|_{L^p(S_{r_1,r_2})}\lesssim_\epsilon(\frac{R^{-\frac12+\epsilon}|\Qc_{r_1,r_2}|}{|\T_1|^{1/2}|\T_2|^{1/2}})^{\frac3p-\frac12}(r_1r_2)^{\frac1p}(\frac{R^{2\alpha-1}}{(N_1N_2)^{\frac12}})^{\frac12-\frac1p}(\sum_{\gamma\in \Gamma_\alpha(R^{-1})}\|\Pc_\gamma F\|^p_{L^p(\R^2)})^{\frac1{p}}.$$

Thus, \eqref{fgygcoeioifjiochfuvhj} and Theorem \ref{t1} will follow if we prove that for $p=\frac{2(1+\alpha)}{\alpha}$
$$(\frac{R^{-\frac12}|\Qc_{r_1,r_2}|}{|\T_1|^{1/2}|\T_2|^{1/2}})^{\frac3p-\frac12}(r_1r_2)^{\frac1p}(\frac{R^{2\alpha-1}}{(N_1N_2)^{\frac12}})^{\frac12-\frac1p}\les R^{\alpha(\frac12-\frac1p)}.
$$
This can be rewritten as
$$|\Qc_{r_1,r_2}|\les \frac{(RN_1N_2)^{\frac1{2(2\alpha-1)}}(|\T_1||\T_2|)^{\frac12}}{(r_1r_2)^{\frac\alpha{2\alpha-1}}}.$$

This inequality is an immediate consequence of the following proposition.
\begin{pr}
For each $i\in\{1,2\}$
$$|\Qc_{r_1,r_2}|\les \frac{R^{\frac1{2(2\alpha-1)}}N_i^{\frac1{2\alpha-1}}|\T_i|}{(r_i)^{\frac{2\alpha}{2\alpha-1}}}.$$
\end{pr}
\begin{proof}
We apply Theorem \ref{12} with $\T=\T_i$, $r=r_i$ and $N=N_i$. We split the analysis into two cases.

Let us assume first that $s\le R^{\alpha-\frac12}$.
Using \eqref{ur8438t754tuijgjotgj3}, it suffices to prove that

$$|\Qc_{r_1,r_2}|\les \frac{R^{\frac1{2(2\alpha-1)}}(\frac{M_s}{s})^{\frac1{2\alpha-1}}|\T_i|}{(r_i)^{\frac{2\alpha}{2\alpha-1}}}.$$
Comparing this with the known upper bound \eqref{ur8438t754tuijgjotgj}, it further suffices to prove that
$$\frac{M_sR^{\frac12}}{sr_i^2}\les \frac{R^{\frac1{2(2\alpha-1)}}(\frac{M_s}{s})^{\frac1{2\alpha-1}}}{(r_i)^{\frac{2\alpha}{2\alpha-1}}},$$
or, after a rearrangement, that
$$1\les (\frac{M_sR^{\frac12}}{sr})^{\frac{2-2\alpha}{2\alpha-1}}.$$
This however follows from \eqref{ur8438t754tuijgjotgj2}, since $s\gtrsim 1$.

We next assume that $s\ge R^{\alpha-\frac12}$. Using \eqref{ur8438t754tuijgjotgj3}, it suffices to verify that

$$|\Qc_{r_1,r_2}|\les \frac{R^{\frac1{2(2\alpha-1)}}(\frac{M_sR^{\alpha-\frac12}}{s^2})^{\frac1{2\alpha-1}}|\T_i|}{(r_i)^{\frac{2\alpha}{2\alpha-1}}}.$$
Comparing this with the known upper bound \eqref{ur8438t754tuijgjotgj}, it further suffices to prove that
$$\frac{M_sR^{\frac12}}{sr_i^2}\les \frac{R^{\frac1{2(2\alpha-1)}}(\frac{M_sR^{\alpha-\frac12}}{s^2})^{\frac1{2\alpha-1}}}{(r_i)^{\frac{2\alpha}{2\alpha-1}}},$$
or, after a rearrangement, that
$$s^{3-2\alpha}\les (\frac{M_s}{r})^{2-2\alpha}R^{1/2}.$$
Using \eqref{ur8438t754tuijgjotgj2}, this will follow if we prove that
$$s^{3-2\alpha}\les (\frac{s^2}{R^{\frac12}})^{2-2\alpha}R^{1/2}.$$
This however is equivalent with the known bound $s\les R^{1/2}$.

\end{proof}

\section{Improved incidences for Vinogradov plates}
\label{Ch:improvedplates}

The material in this rather substantial section will be used in the proof of Theorem \ref{n=3partialrange}. We believe it is also of independent interest and of potential applicability to other problems.  	
\smallskip

Throughout this section we will encounter boxes of various shapes in $\R^3$. Each box is a  parallelepiped with one face parallel
to the $xy$ plane and one edge of that face parallel to the $x$ axis. The dimensions $(d_1,d_2,d_3)$ of the box will be understood as follows: $d_1$ is the length of the edge parallel to the $x$-axis,  $d_2$ is the length of the other edge parallel to the $xy$ plane, and $d_3$ is the length of the remaining edge. All our boxes $B$ are almost rectangular. The dimensions $d_1,d_2,d_3$ of $B$ are comparable to those of a genuinely rectangular box $R$ satisfying $C^{-1}R\subset B\subset CR$.

If the box has the longest two  edges of comparable length, we will call it a {\em plate}. If $d_1,d_2,d_3$ are all substantially  different, we will call it a {\em plank}. The direction/orientation of a plate is completely determined by its normal vector, the vector perpendicular to the face having the two long edges. We will in general not distinguish between two boxes $B_1,B_2$ satisfying $\frac1CB_1\subset B_2\subset CB_1$ for some $C=O(1)$.
\smallskip

For $\delta\in(0,1)$ and $J\subset [0,1]$, let $\I_{\delta}(J)$ be the partition of $J$ into intervals $I$ of length $\delta$. When $I=[0,1]$, we will simply write $\I_\delta$.

Given $t\in[0,1]$, the vectors $\t(t),\n(t),\b(t)$ will denote the unit tangent, normal and binormal vectors at the point $(t,t^2,t^3)$ on the moment curve. The angle between $\t(t_1)$ and $\t(t_2)$ is $O(\delta)$ when $|t_1-t_2|\le \delta$. Thus,  when an angular uncertainty of order $O(\delta)$ is tolerated, we will simply write $\t(I)$ for the tangent corresponding an arbitrary point in some interval $I\in\I_\delta$.


Let us now describe a special type of plates that will play a central role in our investigation.
\begin{de}[Vinogradov plates]

For each  $I\in\I_\delta$, a Vinogradov plate associated with $I$ is a $(\delta,1,1)$-plate $S$ inside $[0,1]^3$, with normal vector $\t(I)$. This definition is unambiguous due to eccentricity considerations.

We will refer to $\t(I)$ as the direction of the plate.
Note that the directions of two plates corresponding to distinct intervals $I$ are $O(\delta)$-separated. The angle between two plates is the angle between their directions.

\end{de}
Throughout this section, the notation $\lessapprox$ is equivalent with $\lesssim (\log \frac1\delta)^{O(1)}$. Also, $A\approx B$ is equivalent to the double inequality $A\lessapprox B\lessapprox A$.

The Vinogradov plates will arise as truncations of Vinogradov planks in the following chapters.

The main feature of the Vinogradov plates is  the fact that their directions are restricted to a curve on $\S^2$.
In some sense, they behave like two dimensional rectangles. To make this more precise, we start with a few geometric lemmas, quantifying the intersections of plates.

\begin{lem}[Volume of intersection]
\label{vv4}	
Let $S_1,S_2$ be Vinogradov plates associated with distinct intervals $I_1,I_2\in\I_\delta$.  Let  $\delta\lesssim D \lesssim 1$ be their angle, so $D\sim\dist(I_1,I_2)$. Assume their centers coincide. Their intersection is an almost rectangular box with dimensions $(\delta,D^{-1}\delta,1)$. In particular, the intersection of any two Vinogradov $(\delta,1,1)$-plates with angle $D$ has volume
$O(\frac{\delta^2}{D})$.

The long side has direction $\t(I_1)\times \t(I_2)$. The short side has direction $\t(I_i)$, where (due to eccentricity reasons) $i$ may be chosen either 1 or 2.

\end{lem}
The proof  is left to the reader. 
\smallskip

\begin{lem}[Small angle]
\label{vv4secondversion}
Let $\delta\lesssim D \lesssim \delta^{1/2}$ and let $J=[t_0,t_0+D]$. Let $P_J$ be a plank centered at the origin, with dimensions $(c\delta,cD^{-1}\delta,c)$, $c$ a small enough constant independent of $\delta,D$. The short side points in the direction $\t(t_0)$, while the long side points in the direction $\t(t_0)\times \t(t_0+D)$.

Then for each Vinogradov plate $S$ centered at the origin and  associated with some $I\in\I_\delta$, $I\subset J$,  we have $P_J\subset S$.
\end{lem}
\begin{proof}
Recall that the normal to $S$ is $\t(t_1)$, for some arbitrary $t_1\in J$. The main concern is with the face of $P_J$ having the smallest eccentricity, namely the one with dimensions $(c\delta,c)$. We have to make sure that the change in angle from ${\t(t_0)}$ to $\t(t_1)$ does not rotate this face with an angle greater than the eccentricity $\delta$. We project  ${\t(t_0)}-{\t(t_1)}$ onto the plane spanned by $\t(t_0)$ and $\t(t_0)\times \t(t_0+D)$ (this is the plane parallel to the face we mentioned). We need to prove that the angle between this projection and $\t(t_0)$ is $O(\delta)$. Note that this angle $\theta$ is comparable (since $\cos (\frac\pi2-\theta)=\sin \theta\sim \theta$) to 
$$|({\t(t_0)}-{\t(t_1)})\cdot \frac{\t(t_0)\times \t(t_0+D)}{|\t(t_0)\times \t(t_0+D)|}|\sim \frac1D |\t(t_1)\cdot(\t(t_0)\times \t(t_0+D))|.$$  
Finally
$$|\t(t_1)\cdot(\t(t_0)\times \t(t_0+D))|\sim |\det  \begin{bmatrix}1&2t_1&3t_1^2\\ 1&2t_0&3t_0^2\\1&2(t_0+D)&3(t_0+D)^2\end{bmatrix}|\lesssim D^3.$$
We conclude that $\theta\lesssim D^2\lesssim \delta$, as desired.

\end{proof}
\smallskip

We can prove the following analog of Proposition \ref{E6}, that we will find useful in the future.

\begin{pr}[Linear Kakeya for Vinogradov plates]
	\label{ccc1}
	 Let $\S$ be a collection of Vinogradov $(\delta,1,1)$-plates, with at most $m$ plates associated with each $I\in\I_\delta$ . Then
	$$\|\sum_{S\in\S}1_S\|_2\lessapprox m^{\frac12}(\sum_{S\in\S}|S|)^{\frac12}.$$
	
\end{pr}
\begin{proof}
The angle between any $S_1,S_2\in\S$  is (comparable to) $ j\delta$, for some $0\le j\lesssim \delta^{-1}$. Moreover, if $S_1$ is fixed, there can only be $m$ plates $S_2\in\S$ that make a fixed angle (comparable to $j\delta$) with $S_1$. In this case, Lemma \ref{vv4} shows that  $$|S_1\cap S_2|\lesssim \frac{\delta}{j+1}.$$
We may thus write
$$\sum_{S_1\in\S}\sum_{S_2\in\S}|S_1\cap S_2|\lesssim m\sum_{S_1\in\S}\sum_{j=1}^{O(\delta^{-1})}j^{-1}|S_1|\lesssim m\log(\delta^{-1})\sum_{S_1\in\S}|S_1|. $$
\end{proof}
We will now  seek for a stronger Kakeya-type estimate for Vinogradov plates, in a trilinear framework.

\begin{de}
Given families $\Fc_1,\Fc_2,\Fc_3$, $\Qc$ of sets in $\R^3$
we introduce the corresponding collection of trilinear $(r_1,r_2,r_3)$-rich sets
$$\Qc_{r_1,r_2,r_3}(\Fc_1,\Fc_2,\Fc_3)=\{q\in\Qc:\;q\text{ is a subset of   at least }r_i \text{ sets from each family }\Fc_i\}.$$
\end{de}
In the forthcoming arguments the collections $\Fc_i$ will consist of either plates or planks that will display a certain transversality (or broadness), while the collection $\Qc$ will consist of (pairwise disjoint) $\Delta$-cubes of various scales $\Delta$. The distinction between``$q$ is a subset of " and ``$q$ intersects" will bear no significance to us and will be ignored.
\bigskip

Here is our main result in this section.

\begin{te}\label{induction}
Let $\frac13<\alpha\le \frac23$. Define $W=\delta^{3\alpha-2}$, and note that $\delta\le W^{-1}\le 1$. Assume that we have a collection  $\S$ of Vinogradov $(\delta,1,1)$-plates with the following structure. For each $I\in\I_\delta$ we denote by $\S_I$ the plates in $\S$ associated with $I$
	\begin{enumerate}
		\item broad structure: there are collections of intervals  $\I_1'\subset  \I_{\delta}([0,1/6])$, $\I_2'\subset  \I_{\delta}([1/3,1/2])$ and $\I_3'\subset  \I_{\delta}([2/3,1])$, such that $|\I_i'|\sim M^{(i)}$ and each plate in $\S$ is associated with some interval in  $$\I'=\I_1'\cup\I_2'\cup\I_3'.$$
		We denote by $\S_1,\S_2,\S_3$ the  collections of plates associated with intervals in $\I_1',\I_2',\I_3'$, respectively.
		\item periodicity: the plates in each  $\S_I$, $I\in\I'$, are periodic in the  $x$ direction with period $W^{-1}$.
		\item uniformity: for each $I\in\I_i$, tile $[0,1]^3$ with fat $(W^{-1},1,1)$-plates $\Sigma$ with direction $\t(I)$. We assume that each $\Sigma$ contains $\sim N^{(i)}$ thin plates $S\in\S_I$. Thus
				$$|\S_I|\sim N^{(i)}W\;\text{ and }\;|\S_i|\sim N^{(i)}M^{(i)}W.$$

	\end{enumerate}
    For $1\le r_i\le M^{(i)}$, let us denote by $\Qc_{r_1,r_2,r_3}(\S_1,\S_2,\S_3)$,  the collection of trilinear $(r_1,r_2,r_3)$-rich $\delta$-cubes determined by $\S_1,\S_2,\S_3$. Let $M=(M^{(1)}M^{(2)}M^{(3)})^{1/3}$, $N=(N^{(1)}N^{(2)}N^{(3)})^{1/3}$, $r=(r_1r_2r_3)^{1/3}$.

    For each $\epsilon>0$ we have the upper bound
	\begin{equation}
	\label{vvv1}
	|\Qc_{r_1,r_2,r_3}(\S_1,\S_2,\S_3)|\lesssim_\epsilon \delta^{-\epsilon} (\frac{NM}{r^2\delta})^{\frac{4-6\alpha}{3\alpha-1}}(\frac{NM}{r})^3 W.
	\end{equation}
	
\end{te}
\smallskip

The argument $(\S_1,\S_2,\S_3)$ in $\Qc_{r_1,r_2,r_3}(\S_1,\S_2,\S_3)$ will be dropped when the family of plates is clear from the context. When $r_1=r_2=r_3=r$, we denote $\Qc_{r_1,r_2,r_3}$ by $\Qc_{r}$ and refer to it as the family of trilinear $r$-rich cubes. 
\bigskip

The only relevance of the choice of intervals $[0,1/6]$, $[1/3,1/2]$ and $[2/3,1]$ is that they are pairwise disjoint. The arguments work equally well for arbitrary triples of such intervals.

Before we prove \eqref{vvv1}, it helps to assess its strength in relation to a previously known estimate. Let $\S_1, \S_2,\S_3$ be three families of plates in $[0,1]^3$. They need not be Vinogradov plates, but we require that the unit normal vectors $\n_1,\n_2,\n_3$ to any $S_1\in\S_1$, $S_2\in\S_2$, $S_3\in\S_3$ satisfy the transversality assumption
$$|\n_1\wedge\n_2\wedge\n_3|\gtrsim 1.$$
In particular, we have
$$|S_1\cap S_2\cap S_3|\sim \delta^3.$$
Combining  this with Chebyshev's inequality we derive the following known estimate for the trilinear $r$-rich $\delta$-cubes determined by $\S_1,\S_2,\S_3$
\begin{equation}
\label{vvv3}
|\Qc_r|\lesssim \frac{\delta^{-3}}{r^3}\int\sum_{S_1\in\S_1}1_{S_1}\sum_{S_2\in\S_2}1_{S_2}\sum_{S_3\in\S_3}1_{S_3}\lesssim \frac{|\S_1||\S_2||\S_3|}{r^3}.
\end{equation}
A simple computation shows that when $\alpha<\frac23$ the upper bound \eqref{vvv1} is  stronger than \eqref{vvv3}, that is
$$(\frac{NM}{r^2\delta})^{\frac{4-6\alpha}{3\alpha-1}}(\frac{NM}{r})^3 W\le (\frac{MNW}{r})^3,$$
precisely when $r\ge (NMW)^{1/2}$. Our improved bound for large $r$ will take advantage of periodicity, uniformity and the fact that we deal with Vinogradov plates. However, in the main argument we will also make use of the weaker bound \eqref{vvv3} at appropriate scales.
\medskip

When $\alpha=2/3$, we have that  $W=1$, so there is no periodicity. The upper bounds \eqref{vvv1} and \eqref{vvv3} are essentially identical, so Theorem \ref{induction} is verified in this case.
\smallskip

The upper bound \eqref{vvv1} is probably not sharp in general, but it suffices for our purposes. The various exponents were picked so that they fit into the induction scheme described in the next two subsections.
\bigskip

Let us now comment on the strategy for proving Theorem \ref{induction} in the range $\frac13<\alpha<\frac23$. We denote by  $C(\delta,\alpha)$  the smallest constant such that the upper bound
$$|\Qc_{r_1,r_2,r_3}|\le C(\delta,\alpha) (\frac{NM}{r^2\delta})^{\frac{4-6\alpha}{3\alpha-1}}(\frac{NM}{r})^3 W$$
holds for each family  of plates as in Theorem \ref{induction}. Our goal is to prove that $C(\delta,\alpha)\lesssim_\epsilon \delta^{-\epsilon}$. We achieve this via a two-step induction on scales. More precisely, we will prove that
\begin{equation}
\label{vvv4}
C(\delta,\alpha)\lessapprox C(\delta^{2-3\alpha},\frac1{3(2-3\alpha)})
\end{equation}
for  $\frac13<\alpha\le \frac12$, and that
\begin{equation}
\label{vvv5}
C(\delta,\alpha)\lessapprox \max(C(\delta^{3\alpha-1},\frac{9\alpha-4}{9\alpha-3}),1)
\end{equation}
for  $\frac12<\alpha< \frac23$.

When $\alpha=1/2$, inequality \eqref{vvv4} reads
$$C(\delta,\frac12)\lessapprox C(\delta^{1/2},\frac23).$$
This suffices to conclude that $C(\delta,\frac12)\lessapprox 1$, since we have established that $C(\delta,\frac23)\lesssim 1.$

\begin{lem}Assume \eqref{vvv4} and \eqref{vvv5} hold. Then $C(\delta,\alpha)\lesssim_\epsilon \delta^{-\epsilon}$ for each $\frac13<\alpha<\frac23$.
\end{lem}
\begin{proof}
 Note that $\frac{1}{3(2-3\alpha)}\in (\frac13,\frac23)$ when $\frac13<\alpha\le \frac12$	and that $\frac{9\alpha-4}{9\alpha-3}\in (\frac13,\frac23]$ when $\frac12<\alpha< \frac23$.
Using \eqref{vvv3} and the trivial bound $\frac{NM}{r^2\delta}\ge 1$ (recall that $r\le M\le \delta^{-1}$), it follows that
$C(\delta,\alpha)\le W^2=\delta^{6\alpha-4}$. Thus, we have the uniform bound
\begin{equation}
\label{efopur89u90ewo-=}
\sup_{\alpha\in(\frac13,\frac23)}C(\delta,\alpha)\le \delta^{-2},
\end{equation}
which will serve as the base of our induction.

Assume $\alpha\in (\frac13,\frac23)$. The exponents $2-3\alpha$ and $3\alpha-1$ are both in $(0,1)$ when $\frac13<\alpha\le \frac12$ and $\frac12<\alpha< \frac23$, respectively. Iterating \eqref{vvv4} and \eqref{vvv5} we arrive at inequalities of the form
$$\max(C(\delta,\alpha),1)\lessapprox \max(C(\delta^\beta,\alpha'),1).$$
We have two scenarios. If $\alpha'$ eventually becomes $\frac23$,  we stop and use that $C(\delta^\beta,\frac23)\lesssim 1$.
Otherwise, we claim that $\beta$ can be pushed arbitrarily close to $0$ while always keeping $\alpha'$ in $(\frac13,\frac23)$ (this together with \eqref{efopur89u90ewo-=} is enough to conclude that $C(\delta,\alpha)\lesssim_\epsilon\delta^{-\epsilon}$).  Indeed, note that $\beta$ is a product of factors of the form $f_1(\alpha')=2-3\alpha'$ and $f_2(\alpha')=3\alpha'-1$, with $\alpha'$ in the forward orbit of $\alpha$. The only way for $\beta$ to stay away from zero is if all $\alpha'$ are eventually converging to either $\frac13$ or $\frac23$ (since $f_1(\frac13)=f_2(\frac23)=1$). This however is impossible, since
$$\frac{9\alpha-4}{9\alpha-3}<\alpha<\frac{1}{3(2-\alpha)}.$$
This means that if $\alpha$ is close to $\frac13$ (or $\frac23$), its successor $\alpha'$ is further away from $\frac13$ (or $\frac23$).

\end{proof}
 \bigskip

We close this preliminary discussion with  an elementary lemma that will be used repeatedly throughout the forthcoming argument.
\begin{lem}
\label{tril-pigeon}
Let $\Fc_1,\Fc_2,\Fc_3$ and $\Bc$ be  families of sets in $\R^3$.
Assume that for each $1\le i\le 3$
$$\Fc_{i}=\bigcup_{1\le j\lessapprox 1}\Fc_{i,j}.$$
Then
$$|\Qc_r(\Fc_1,\Fc_2,\Fc_3)|\lessapprox \max_{j_1,j_2,j_3}\max_{r':\;r\approx r' }|\Qc_{r'}(\Fc_{1,j_1},\Fc_{2,j_2},\Fc_{3,j_3})|.$$
\end{lem}
The proof of the lemma is immediate and will be omitted.
\bigskip

\subsection{The case $\frac13<\alpha\le \frac12$}	In this subsection we deal with the first half of the induction approach.

\begin{te}
	\label{ind2}
	Let $\frac13<\alpha\le \frac12$.
	We have
	\begin{equation}
	\label{vvv6}
	C(\delta,\alpha)\lessapprox C(\delta^{2-3\alpha},\frac1{3(2-3\alpha)}).
	\end{equation}	
\end{te}
\bigskip

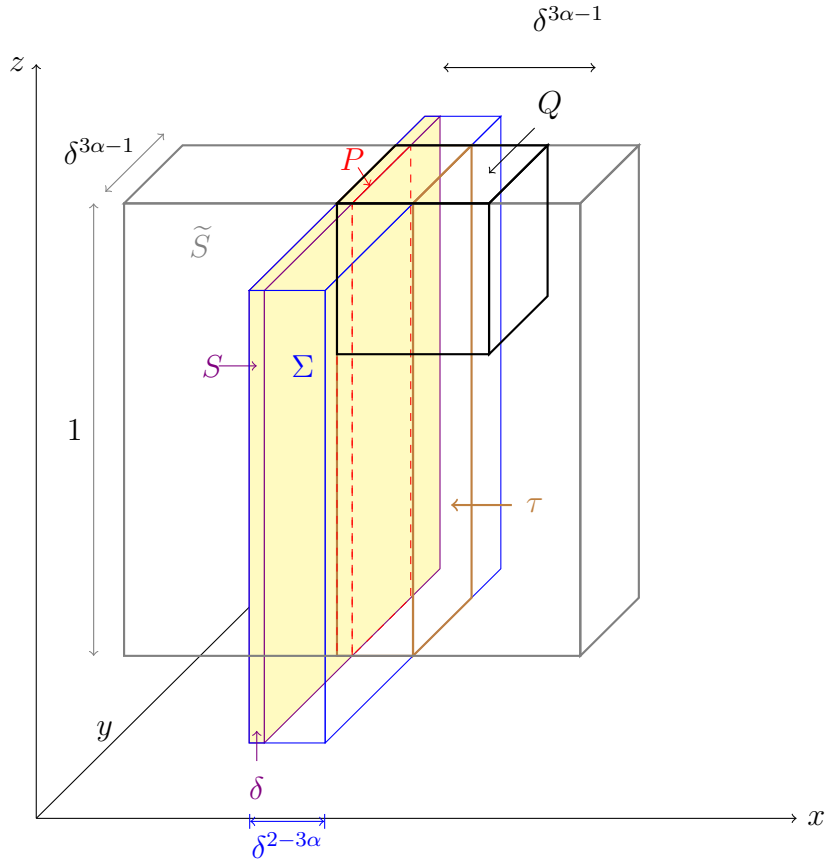
\textbf{Figure 2}	
\begin{center}

	\begin{tikzpicture}[scale=2]
\pgfmathsetmacro{\cubex}{3}
\pgfmathsetmacro{\cubey}{3}
\pgfmathsetmacro{\cubez}{3}

\coordinate (origin) at (-\cubex, -0.5-\cubey, 0);
\draw[->] (-\cubex, -0.5-\cubey, 0) -- ++ (5, 0, 0);
\node[right] at (5-\cubex, -0.5-\cubey, 0) {$x$};
\draw[->] (origin) -- ++ (0,5, 0);
\node[left] at ($ (origin) + (0, 5, 0)$) {$z$};
\draw[->] (origin)--++ (0, 0, -5);
\node[left] at ($ (origin)! 0.3! (-\cubex, -0.5-\cubey, -5)$) {$y$};

\pgfmathsetmacro{\Splatex}{0.1}
\pgfmathsetmacro{\Splatey}{3}
\pgfmathsetmacro{\Splatez}{3}
\coordinate (Sorigin) at (-\cubex/2,0,0); 
\draw[violet, fill=yellow!30] (Sorigin) -- ++(-\Splatex, 0,0)-- ++ (0,-\Splatey, 0) -- ++ (\Splatex, 0, 0)-- cycle;
\draw[violet, fill=yellow!30]  (Sorigin) -- ++ (0,0, -\Splatez) -- ++ (0,-\Splatey, 0) -- ++ (0,0, \Splatez)--cycle;
\draw[violet, fill=yellow!30]  (Sorigin) -- ++ (-\Splatex, 0,0) -- ++ (0,0,-\Splatez) -- ++ (\Splatex, 0,0) -- cycle;
\draw[violet, ->] ([yshift=-0.5cm, xshift=-0.3cm]Sorigin) -- ++(0.25, 0, 0);
\node[below, left, violet] at ([yshift=-0.5cm, xshift=-0.2cm]Sorigin) {$S$}; 

\pgfmathsetmacro{\PSplatex}{0.5}
\pgfmathsetmacro{\PSplatey}{3}
\pgfmathsetmacro{\PSplatez}{3}
\coordinate (PSorigin) at (-\cubex/2+0.4,0,0); 
\draw[blue] (PSorigin) -- ++(-\PSplatex, 0,0)-- ++ (0,-\PSplatey, 0) -- ++ (\PSplatex, 0, 0)-- cycle;
\draw[blue]  (PSorigin) -- ++ (0,0, -\PSplatez) -- ++ (0,-\PSplatey, 0) -- ++ (0,0, \PSplatez)--cycle;
\draw[blue]  (PSorigin) -- ++ (-\PSplatex, 0,0) -- ++ (0,0,-\PSplatez) -- ++ (\PSplatex, 0,0) -- cycle;
\node[below, left, blue] at ([yshift=-0.5cm]PSorigin) {$\Sigma$}; 

\coordinate (ssigma) at (0.2, 0.4, 0.2);
\coordinate (sigma1) at ($(PSorigin)	+(-\PSplatex, 0, -\PSplatez) + (ssigma)$);
\coordinate (sigma2) at ($(PSorigin)+(\PSplatex, 0, -\PSplatez)+(ssigma)$);
\draw[black, <->] (sigma1)--(sigma2);
\node[black] at ($(sigma1)! 0.7 ! (sigma2) +(ssigma)$) {$\delta^{3\alpha-1}$};

\coordinate (dd) at (0, -0.3, 0);
\coordinate (S1) at ($(Sorigin)+(0, -\Splatey, 0)+(dd)$);
\coordinate (S2) at ($(Sorigin)+(-\Splatex, -\Splatey, 0)+(dd)$);
\node[violet] at ($(S1)!0.5!(S2)$) {$\delta$};
\draw[violet, ->] ([yshift =0.18cm]$(S1)!0.5!(S2)$)  -- ++ (0, 0.2, 0);

\coordinate (ddd) at (0, -0.52, 0);
\coordinate (SS1) at ($(PSorigin)+(0, -\PSplatey, 0)+(ddd)$);
\coordinate (SS2) at ($(PSorigin)+(-\PSplatex, -\PSplatey, 0)+(ddd)$);
\draw[blue, |<->|] (SS1)--(SS2);
\node[blue, below]  at ($(SS1)! 0.5 ! (SS2)$) {$\delta^{2-3\alpha}$};

\pgfmathsetmacro{\taux}{0.5}
\pgfmathsetmacro{\tauy}{3}
\pgfmathsetmacro{\tauz}{1}
\coordinate (tauo) at (-\cubex/2+0.4, 0, -1.5);
\draw[thick, brown] (tauo) -- ++(-\taux, 0,0)-- ++ (0,-\tauy, 0) -- ++ (\taux, 0, 0)-- cycle;
\draw[thick, brown]  (tauo) -- ++ (0,0, -\tauz) -- ++ (0,-\tauy, 0) -- ++ (0,0, \tauz)--cycle;
\draw[thick, brown]  (tauo) -- ++ (-\taux, 0,0) -- ++ (0,0,-\tauz) -- ++ (\taux, 0,0) -- cycle;
\draw[thick, brown, ->] ([yshift=-2cm, xshift=0.65cm] tauo) -- ++(-0.4, 0, 0);
\node[thick, brown] at ([yshift=-2cm, xshift=0.8cm] tauo) {$\tau$};

\pgfmathsetmacro{\plankx}{0.1}
\pgfmathsetmacro{\planky}{3}
\pgfmathsetmacro{\plankz}{1}
\coordinate (planko) at (-\cubex/2, 0,-1.5);
\draw[red, dashed]  (planko) -- ++(-\plankx, 0,0)-- ++ (0,-\planky, 0) -- ++ (\plankx, 0, 0)-- cycle;
\draw[red, dashed]  (planko) -- ++ (0,0, -\plankz) -- ++ (0,-\planky, 0) -- ++ (0,0, \plankz)--cycle;
\draw[red]  (planko) -- ++ (-\plankx, 0,0) -- ++ (0,0,-\plankz) -- ++ (\plankx, 0,0) -- cycle;

\node[red, above] at ($(planko)+(0,0.15,0)$) {$P$};
\draw[red, ->] ($(planko) + (0, 0.2, -0.10)$)--($(planko)+(0, 0, -0.3)$);

\pgfmathsetmacro{\BSplatex}{3}
\pgfmathsetmacro{\BSplatey}{3}
\pgfmathsetmacro{\BSplatez}{1}
\coordinate (BSorigin) at (0, 0, -1.5);
\draw[thick, gray] (BSorigin) -- ++(-\BSplatex, 0,0)-- ++ (0,-\BSplatey, 0) -- ++ (\BSplatex, 0, 0)-- cycle;
\draw[thick, gray]  (BSorigin) -- ++ (0,0, -\BSplatez) -- ++ (0,-\BSplatey, 0) -- ++ (0,0, \BSplatez)--cycle;
\draw[thick, gray]  (BSorigin) -- ++ (-\BSplatex, 0,0) -- ++ (0,0,-\BSplatez) -- ++ (\BSplatex, 0,0) -- cycle;
\node[thick, gray] at ([xshift = -2.5cm, yshift=-0.25cm]BSorigin) {$\widetilde{S}$};

\coordinate(ss) at (-0.2, 0, -0.2);
\coordinate (m1) at ($(BSorigin)+(-\BSplatex, 0, 0) + (ss)$);
\coordinate (m2) at ($(BSorigin)+(-\BSplatex, 0, -\BSplatez) + (ss)$);
\draw[gray, <->] (m1)--(m2);
\node[left] at ($(m1)! 0.7 !(m2) $)  {$\delta^{3\alpha-1}$};

\coordinate (sss) at (-0.2, 0, 0);
\coordinate (mm1) at ($(BSorigin)+(-\BSplatex, 0, 0) + (sss)$);
\coordinate (mm2) at ($(BSorigin)+ (-\BSplatex, -\BSplatey, 0) +(sss)$);
\draw[gray, <->] (mm1)--(mm2);
\node[left] at ($(mm1)!0.5!(mm2)$) {$1$};


\pgfmathsetmacro{\Qx}{1}
\pgfmathsetmacro{\Qy}{1}
\pgfmathsetmacro{\Qz}{1}
\coordinate (Qorigin) at  (-\cubex/2+0.9, 0,-1.5);
\draw[thick] (Qorigin) -- ++(-\Qx, 0,0)-- ++ (0,-\Qy, 0) -- ++ (\Qx, 0, 0)-- cycle;
\draw[thick]  (Qorigin) -- ++ (0,0, -\Qz) -- ++ (0,-\Qy, 0) -- ++ (0,0, \Qz)--cycle;
\draw[thick]  (Qorigin) -- ++ (-\Qx, 0,0) -- ++ (0,0,-\Qz) -- ++ (\Qx, 0,0) -- cycle;

\draw[ ->] ([yshift=0.5cm, xshift=0.3cm]Qorigin) --++(-0.3,  -0.3,0);

\node[ thick] at ([yshift=0.65cm, xshift=0.4cm]Qorigin) {$Q$};

\end{tikzpicture}
\end{center}
\bigskip
\begin{proof}

	Let us fix a collection of plates as in Theorem \ref{induction}. To simplify the numerology in the exposition, we only analyze the diagonal case. More precisely, we assume that $M^{(i)}=M$,  $N^{(i)}=N$ and  $r_i=r$ for each $1\le i\le 3$. The reader is invited to check that our argument extends to the general (non-diagonal) case. In short, all rounds of pigeonholing for various parameters (e.g. $M_1,M_2,N_1,N_2$) as well as the estimates for them are done individually for each of the components $i\in\{1,2,3\}$. In Step 8, the number of cubes $Q$ is estimated using bilinear Kakeya. In the non-diagonal case, there are three such estimates available (for each pair of indices in $\{1,2,3\}$), and one uses the geometric average of these estimates to recover \eqref{vv21}.

	We need to prove that
	\begin{equation}
	\label{ccc2}
	|\Qc_r|\lessapprox C(\delta^{2-3\alpha},\frac1{3(2-3\alpha)}) (\frac{NM}{r^2\delta})^{\frac{4-6\alpha}{3\alpha-1}}(\frac{NM}{r})^3 W.
	\end{equation}
	The argument involves several stages. The first few steps pigeonhole key parameters that add more uniformity to the problem. The counting of $\delta$-cubes in $[0,1]^3$ is gradually localized to cubes $Q$ of scale $\delta^{3\alpha-1}$. Inside $Q$, the original Vinogradov plates $S$ are reduced to tiny plates. When $Q$ is rescaled to $[0,1]^3$, the tiny plates become Vinogradov $(\delta^{2-3\alpha},1,1)$-plates, and the induction hypothesis is applicable.
	
	Let us also say a few words about the most subtle part of our argument, that has to do with distinguishing between various  cubes $Q$. The hypotheses of our theorem imply that each $Q$ is intersected by the same number $\sim MN\frac{\delta^{3\alpha-1}}{\delta^{2-3\alpha}}$ of plates $S\in\S$. From this limited perspective, all $Q$ are the same. We replace plates with planks (the intuition behind this is suggested by  Lemma \ref{vv4secondversion}), at the expense of introducing new uniformity parameters. Planks have an extra direction, and this allows us to exploit the initial trilinear transversality in the form of the bilinear Kakeya inequality. This will provide us with a satisfactory upper bound on the number of ``heavy" cubes $Q$, those with many contributing directions.
	\\
	
	1. Pigeonholing the  parameters $M_1,M_2$
	\\

	We organize the large intervals ${J}\in \I_{\frac1W}([0,1/6])$ into different families, according to how many small intervals from $\I_1'$ they contain. A typical family will be associated with dyadic  parameters $M_1,M_2$, as follows. It will contain $\sim M_1$ intervals ${J}$, with each ${J}$ containing $\sim M_2$ intervals from $\I_1'$. Note that
	\begin{equation}
	\label{vv8}
	M_1M_2\lesssim M.
	\end{equation}
	We apply the same procedure to $\I_{\frac1W}([1/3,1/2])$ and $\I_{\frac1W}([2/3,1])$. Since there are $\lessapprox 1$ relevant dyadic values of $M_1,M_2$, it suffices to work with plates $S$ corresponding to a fixed choice of these parameters, for each of the three families.
	We caution that the parameters $M_1, M_2$ could in principle be different for each of the three families, however, we will only analyze the case when  they are the same. This assumption will simplify the numerology in the forthcoming argument. We will apply this type of simplification a few more times, without always mentioning it explicitly again.
	
	The reduction we used in this step follows via  an application of Lemma \ref{tril-pigeon}. Once we decide to work with the restricted families of plates corresponding to the  parameters $M_1,M_2$, the value of $r$ may become smaller, but only by some  logarithmic fraction. This is of course acceptable, due to the use of $\lessapprox$ in \eqref{ccc2}.
	\\
	
	2. Replacing the plates $S$ with planks $P$
	\\
	
	Fix  ${J}=[t,t+\frac1W]$, one of the $\sim M_1$  intervals selected in the previous step. Note that due to eccentricity considerations,  the plates $\Sigma$ can be thought of as being the same for all $I\subset {J}$. 	
	We denote by $\S_\Sigma$ the plates $S\in\cup_{I\subset {J}}\S_I$ lying inside $\Sigma$. Recall that only $\sim M_2$ intervals $I$ contribute.
	We tile each $\Sigma$ with  planks $P$ parallel to $\Sigma$,  with dimensions $(\delta,\delta^{3\alpha-1},1)$. call this collection $\P_\Sigma$. These are translates of the plank $P_J$ introduced in Lemma \ref{vv4secondversion}. Since $\frac1W\lesssim \delta^{1/2}$, this lemma is applicable in our context.
Recall that the long side of $P$ is in the direction $\t(t)\times \t(t+{\frac1W})$. The relevance of these directions comes (only somewhat loosely) in the use of bilinear Kakeya in Step 8. Bilinear Kakeya only demands $\sim 1$ separation between the (directions of the) two families of tubes, but does not demand separation between the tubes in either family.    What will also  matter is that the normals to these planks (which coincide with the normals to their parents $\Sigma$) will be sufficiently separated for planks associated with different intervals $J$. This will be needed in Step 9, when the induction hypothesis will be invoked for small pieces of these planks. 
	
	Given any $P\in\P_\Sigma$ and $S\in\S_\Sigma$, we can think of $P$ as either lying inside $S$, or being disjoint from it. This is a very harmless assumption. We discard the planks that do not intersect any $S\in\S_\Sigma$. We can thus think of the remaining planks as covering the same area as the plates $S\in\S_\Sigma$.
	\\
	
	3. Pigeonholing the  parameter $E_2$
	\\
	
	For each $\Sigma$ we partition the planks $P\in \P_\Sigma$ according to the number $E_2$ of plates $S\in\S_\Sigma$ they belong to. Note that $E_2\lesssim M_2$. Since there are $\lessapprox 1$ such dyadic values of $E_2$, it suffices to work with the planks corresponding to a fixed $E_2$. We will call them $E_2$-planks.
	
	Recall that we are interested in counting  trilinear $r$-rich $\delta$-cubes with respect to $\S$. We can now rethink this problem as counting the  trilinear $\frac{r}{E_2}$-rich $\delta$-cubes with respect to the family of $E_2$-planks $P$. Note that if $\frac{r}{E_2}>M_1$ then there cannot be any such cube in the collection. Thus we may assume that $\frac{r}{E_2}\le M_1$, in particular
	\begin{equation}
	\label{vv9}
	E_1:=\frac{r}{E_2}\le W.
	\end{equation}
	The reduction in this step relies on  another application of Lemma \ref{tril-pigeon}.
	From now on, each mentioning of planks $P$ will implicitly refer to $E_2$-planks.
	\\
	
	4. Pigeonholing the  parameters $N_1,N_2$ and the boxes $\tau$
	\\
	
	Split each  $\Sigma$ into parallel  $(\frac1W,\delta^{3\alpha-1},1)$-boxes $\tau$, in such a way that each $P\in\P_\Sigma$ fits inside some $\tau$. Note that $\frac1W\le \delta^{3\alpha-1}$, since $\alpha\le \frac12$. Invoking again dyadic considerations,  we may restrict attention to those $\tau$ containing $\sim N_1$ planks $P$, for some fixed dyadic parameter  $N_1$. Call these $\tau$ $N_1$-rich.
	
	Moreover, we may restrict attention to those $\Sigma$ containing $\sim \frac{N_2}{N_1}$ such $\tau$, for some fixed $N_2\ge N_1$. We will say that $\Sigma$ has $(N_1,N_2)$-configuration. There will be  $\sim N_2$ planks $P$ inside each  such $\Sigma$.

	Due to our periodicity assumption, the collections  $\P_\Sigma$ and $\S_\Sigma$  are the same (up to translation in the $x$ direction) for all $\Sigma$ associated with a fixed ${J}$. Thus, if some $\Sigma$ has $(N_1,N_2)$-configuration, then so does every other $\Sigma'$ parallel to $\Sigma$.  However, only part of the original collection of $M_1$ intervals ${J}$ will contribute to the $(N_1,N_2)$-family.

	The parameter $M_1$ can be thought of as getting smaller, reflecting the number of intervals ${J}$ that have survived these last two rounds of pigeonholing. But note that $M_2$, $E_1$, $E_2$ have not changed in the process, and that \eqref{vv8} continues to hold.
	
	To summarize, we only keep the  plates $\Sigma$ with $(N_1,N_2)$-configuration, the $N_1$-rich boxes $\tau\subset \Sigma$ and the $\sim N_1$ planks $P$ contained in each such $\tau$. All other plates, boxes and planks are discarded at this point. This step demands another application of Lemma \ref{tril-pigeon}.
	\smallskip
	
	At this point we are done with pigeonholing. What remains to be done is to estimate the various parameters, and to assemble the derived inequalities into the desired final estimate.
	\\
	
	5.  An upper bound for $N_1$ via a double counting argument
	\\
	
	Let us prove the following inequality
	\begin{equation}
	\label{vv6}
	N_1E_2\lesssim M_2N.
	\end{equation}There are $\sim M_2N$ plates $S$ inside a fat plate $\Sigma$, roughly $N$ for each of the $\sim M_2$ contributing directions.
	 The value $N_1E_2$ represents the number of plates $S$  that intersect the $N_1$ planks $P$ (recall that each $P$ is $E_2$-rich) in some fixed $N_1$-rich box  $\tau\subset\Sigma$. Now \eqref{vv6} follows from the fact  that a given $S$ can intersect at most $O(1)$ such planks $P$.
	\\

	6. An upper bound for $N_2$ via linear Kakeya
	\\
	
	As observed in the previous step, there are $\sim M_2N$ plates $S$ inside a fat plate $\Sigma$.
	Proposition \ref{ccc1} leads to the inequality
	$$\|\sum_{S\subset \Sigma}1_S\|_2^2\lessapprox\delta N^2M_2.$$
	Since there are $\sim N_2$ ($E_2$-rich) planks $P$ in $\Sigma$, combining this upper bound with Chebyshev's inequality leads to
	\begin{equation}
	\label{vv7}
	N_2\lessapprox \delta^{1-3\alpha}M_2(\frac{N}{E_2})^2.
	\end{equation}
	The factor $\delta^{1-3\alpha}$ represents the ratio between the volume of $S$ and the volume of $P$.
	\\

	7. The plates $\widetilde{S}$
	\\
	
	The boxes $\tau$ are periodic in the $x$ direction with period $\frac1W$. For each contributing ${J}$  we tile $[0,1]^3$ with $(1,\delta^{3\alpha-1},1)$-plates $\widetilde{S}$, so that each $\tau$ fits inside such an $\widetilde{S}$. We caution that these are not Vinogradov plates. There are $\sim \frac{N_2}{N_1}$ parallel plates $\widetilde{S}$ for each of the $\le M_1$ intervals ${J}$, and each $\widetilde{S}$ contains $\sim WN_1$ planks $P$.
	\\
	
	8. Counting $\delta^{3\alpha-1}$-cubes using bilinear Kakeya
	\\
	
	In this step we are about to exploit the bilinear transversality of the plates $\widetilde{S}$, a feature inherited from the original trilinear transversality of the plates $S$.
	
	We partition $[0,1]^3$ into a family of cubes $Q$ with side length $\sim \delta^{3\alpha-1}$. We classify these cubes according to a new dyadic parameter $\widetilde{M}\le M_1$ which represents the  number of plates $\widetilde{S}$ that intersects them,  for each of the three transverse directions. We are making again a harmless reduction to the diagonal case, restricting our focus to the case when $\widetilde{M}$ is the same for all three families.   In other words, we assume that  there are at least $\widetilde{M}$ contributing ${J}$ in each of the intervals $[0,1/6]$, $[1/3,1/2]$ and $[2/3,1]$. In our terminology, each such cube will be trilinear $\widetilde{M}$-rich.

	We count the trilinear $E_1$-rich $\delta$-cubes lying  inside some cube $Q$. Since there are $\lessapprox 1$ dyadic values of $\widetilde{M}$, it suffices to focus attention on the cubes $Q$ corresponding to a fixed $\widetilde{M}$ and to estimate the number of trilinear $E_1$-rich $\delta$-cubes they contain. This is a two-stage process.
	
	First, we count the number of such cubes $Q$.
	Since each $\widetilde{S}$ is parallel to the $x$ axis, this is a planar problem. We use the bilinear Kakeya inequality (Proposition \ref{vvv7}) for the projections of the plates $\widetilde{S}$ on the $yz$-plane. These are planar tubes, and recall that we have three such families of tubes. It is easy to see that the angle between tubes in two distinct families is $\gtrsim 1$.

	We dominate the number of such cubes $Q$ by
	\begin{equation}
	\label{vv21}
	\delta^{1-3\alpha}(\frac{N_2M_1}{N_1\widetilde{M}})^2.
	\end{equation}
	\\
	
	9. Counting trilinear $E_1$-rich $\delta$-cubes inside a cube $Q$ using the induction hypothesis
	\\
	
	Let us fix a $\delta^{3\alpha-1}$-cube $Q$ with parameter $\widetilde{M}$ introduced in the previous step. For each contributing $J$ there are $\sim \delta^{3\alpha-1}W$ boxes $\tau$ intersecting $Q$, that are $N_1$-rich. We localize the analysis to $Q$. The intersection of a plank $P\subset \tau$ with $Q$ is a tiny $(\delta,\delta^{3\alpha-1},\delta^{3\alpha-1})$-plate. There are $\sim \delta^{3\alpha-1}WN_1$ such tiny plates for each of the $\widetilde{M}$ directions, and they are $\frac1W$-periodic in the $x$ direction.
	
	We rescale $Q$ so that it becomes $[0,1]^3$. The tiny plates become $(\widetilde{\delta},1,1)$-plates, where
	$$\widetilde{\delta}=\delta^{2-3\alpha}.$$
	Note that they are Vinogradov plates, as they share the normal vectors of  their parent plates $\Sigma$. Also, it is easy to see that they have $\widetilde{\delta}$-separated directions, as the directions of the plates $\Sigma$ are themselves $\widetilde{\delta}$-separated.
	These new Vinogradov plates  are $\delta^{3-6\alpha}$-periodic in the $x$ direction. We introduce parameters $\widetilde{W}$ and $\widetilde{\alpha}$ satisfying
	$$\frac1{\widetilde{W}}=\delta^{3-6\alpha}=\widetilde{\delta}^{2-3\widetilde{\alpha}}.$$
	We have
	$$\widetilde{\alpha}=\frac1{3(2-3\alpha)}.$$
	We find that the number of trilinear $E_1$-rich $\delta$-cubes  inside $Q$ is bounded by
	$$\le C(\delta^{2-3\alpha},\frac1{3(2-3\alpha)})(\frac{N_1\widetilde{M}}{E_1^2\widetilde{\delta}})^{\frac{4-6\widetilde{\alpha}}{3\widetilde{\alpha}-1}}(\frac{N_1\widetilde{M}}{E_1})^3\widetilde{W}$$
	\begin{equation}
	\label{vv20}
	=C(\delta^{2-3\alpha},\frac1{3(2-3\alpha)})(\frac{N_1\widetilde{M}}{E_1^2\widetilde{\delta}})^{\frac{6-12{\alpha}}{3{\alpha}-1}}(\frac{N_1\widetilde{M}}{E_1})^3\widetilde{W}.
	\end{equation}

	10. Reaching the final estimate
	\\
	
	Combining \eqref{vv21} with \eqref{vv20} we conclude that the number of trilinear $E_1$-rich $\delta$-cubes lying inside trilinear $\widetilde{M}$-rich cubes $Q$ is dominated by
	$$C(\delta^{2-3\alpha},\frac1{3(2-3\alpha)})(\frac{N_1\widetilde{M}}{E_1^2\widetilde{\delta}})^{\frac{6-12{\alpha}}{3{\alpha}-1}}(\frac{N_1\widetilde{M}}{E_1})^3(\frac{N_2M_1}{N_1\widetilde{M}})^2\widetilde{W}\delta^{1-3\alpha}.$$
	We next use that $\widetilde{W}\delta^{1-3\alpha}=W$. Also, since the cumulative exponents of both $N_1$ and $\widetilde{M}$ in the expression from above are positive, we may invoke \eqref{vv6} and $\widetilde{M}\le M_1$ to dominate this expression by
	 $$C(\delta^{2-3\alpha},\frac1{3(2-3\alpha)})(\frac{M_2M_1N}{E_1^2E_2\widetilde{\delta}})^{\frac{6-12{\alpha}}{3{\alpha}-1}}(\frac{M_2M_1N}{E_1E_2})^3(\frac{N_2E_2}{M_2N})^2{W}.$$
	Using \eqref{vv8} and \eqref{vv9}, this is further  dominated by
	$$C(\delta^{2-3\alpha},\frac1{3(2-3\alpha)})(\frac{MN}{E_1r\widetilde{\delta}})^{\frac{6-12{\alpha}}{3{\alpha}-1}}(\frac{MN}{r})^3(\frac{N_2E_2}{M_2N})^2{W}.$$
	In order to verify  \eqref{ccc2}, it now  suffices to prove that
	$$(\frac{NM}{E_1r\widetilde{\delta}})^{\frac{6-12{\alpha}}{3{\alpha}-1}}(\frac{N_2E_2}{M_2N})^2\lessapprox (\frac{NM}{r^2\delta})^{\frac{4-6\alpha}{3\alpha-1}}.$$
	Since
	$$\frac{6-12\alpha}{3\alpha-1}+2=\frac{4-6\alpha}{3\alpha-1},$$
	this will follow once we verify that
	\begin{equation}
	\label{vv22}
	\frac{NM}{E_1r\widetilde{\delta}}\lesssim \frac{NM}{r^2\delta}
	\end{equation}
	and
	\begin{equation}
	\label{vv23}
	\frac{N_2E_2}{M_2N}\lessapprox  \frac{NM}{r^2\delta}.
	\end{equation}
	Note that \eqref{vv22} is equivalent to the immediate inequality $E_2\lesssim\frac{\widetilde{\delta}}{\delta}$.
\\
Using \eqref{vv7}, inequality \eqref{vv23} is reduced to proving
$$r^2\lesssim ME_2W.$$
This is also immediate since $r\le M$ and $r=E_1E_2\lesssim WE_2$.

\end{proof}
\bigskip

\subsection{The case $\frac12<\alpha<\frac23$} We now complete the second step in the proof of   Theorem \ref{induction}, by dealing with the case $\alpha\in(\frac12,\frac23)$.
\bigskip

	\begin{te}
	\label{ind1}
	Let $\frac12<\alpha<\frac23$.
	We have $$C(\delta,\alpha)\lessapprox \max(C(\delta^{3\alpha-1},\frac{9\alpha-4}{9\alpha-3}),1).$$	
	\end{te}

\begin{proof}
	Let us fix a collection of plates as in Theorem \ref{induction}. We need to prove that
	\begin{equation}
	\label{ccc3}
	|\Qc_r|\lessapprox \max(C(\delta^{3\alpha-1},\frac{9\alpha-4}{9\alpha-3}),1) (\frac{NM}{r^2\delta})^{\frac{4-6\alpha}{3\alpha-1}}(\frac{NM}{r})^3 W.
	\end{equation}
	
    The proof shares similarities with that of Theorem \ref{ind2}. Because of this, some details will be omitted this time. We again focus only on the diagonal case for the parameters $M^{(i)}$, $N^{(i)}$, $r_i$ and also for the new parameters arising in our argument.

	There will be two important scales in the argument, $W=\delta^{3\alpha-2}$ and $\widetilde{\delta}=\delta^{3\alpha-1}$.
	
	Recall that we have $\sim M$ relevant intervals $I\in\I_\delta$.  In addition to these, we will deal with  intervals $\widetilde{J}\in\I_{\delta^{3\alpha-1}}$ and with longer intervals $J\in\I_{\delta^{2-3\alpha}}$. Recall that the family of fat plates $\Sigma$ is the same for all intervals $I\in \I_\delta(J)$. We will say that the plate $\Sigma$ is associated with $J$.

	We will start by selecting intervals $\widetilde{J}$ that are uniform with respect to the number of relevant intervals  $I$ they contain.
	\bigskip
	
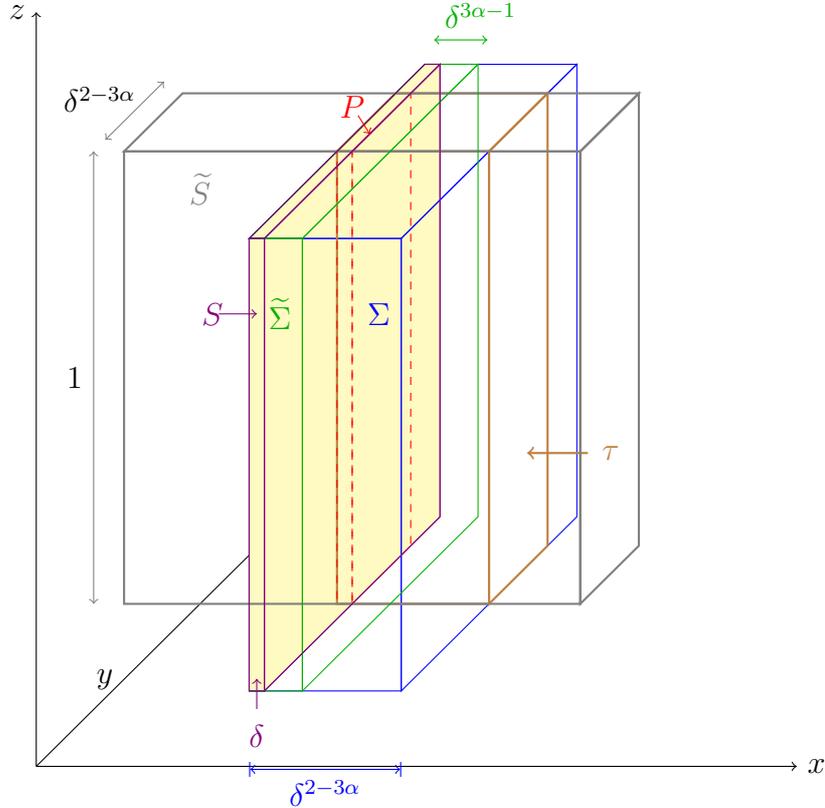
	\textbf{Figure 3}
	
	\begin{center}
		
		\begin{tikzpicture}[scale=2]
	\pgfmathsetmacro{\cubex}{3}
	\pgfmathsetmacro{\cubey}{3}
	\pgfmathsetmacro{\cubez}{3}
	
	\draw[->] (-\cubex, -0.5-\cubey, 0) -- ++ (5, 0, 0);
	\node[right] at (5-\cubex, -0.5-\cubey, 0) {$x$};
	\draw[->] (origin) -- ++ (0,5, 0);
	\node[left] at ($ (origin) + (0, 5, 0)$) {$z$};
	\draw[->] (origin)--++ (0, 0, -5);
	\node[left] at ($ (origin)! 0.3! (-\cubex, -0.5-\cubey, -5)$) {$y$};

	\pgfmathsetmacro{\Splatex}{0.1}
	\pgfmathsetmacro{\Splatey}{3}
	\pgfmathsetmacro{\Splatez}{3}
	\coordinate (Sorigin) at (-\cubex/2,0,0); 
	\draw[violet, fill=yellow!30] (Sorigin) -- ++(-\Splatex, 0,0)-- ++ (0,-\Splatey, 0) -- ++ (\Splatex, 0, 0)-- cycle;
	\draw[violet, fill=yellow!30]  (Sorigin) -- ++ (0,0, -\Splatez) -- ++ (0,-\Splatey, 0) -- ++ (0,0, \Splatez)--cycle;
	\draw[violet, fill=yellow!30]  (Sorigin) -- ++ (-\Splatex, 0,0) -- ++ (0,0,-\Splatez) -- ++ (\Splatex, 0,0) -- cycle;
	\draw[violet, ->] ([yshift=-0.5cm, xshift=-0.3cm]Sorigin) -- ++(0.25, 0, 0);
	\node[below, left, violet] at ([yshift=-0.5cm, xshift=-0.2cm]Sorigin) {$S$}; 
	
	\pgfmathsetmacro{\PSplatex}{1}
	\pgfmathsetmacro{\PSplatey}{3}
	\pgfmathsetmacro{\PSplatez}{3}
	\coordinate (PSorigin) at (-\cubex/2+0.9,0,0); 
	\draw[blue] (PSorigin) -- ++(-\PSplatex, 0,0)-- ++ (0,-\PSplatey, 0) -- ++ (\PSplatex, 0, 0)-- cycle;
	\draw[blue]  (PSorigin) -- ++ (0,0, -\PSplatez) -- ++ (0,-\PSplatey, 0) -- ++ (0,0, \PSplatez)--cycle;
	\draw[blue]  (PSorigin) -- ++ (-\PSplatex, 0,0) -- ++ (0,0,-\PSplatez) -- ++ (\PSplatex, 0,0) -- cycle;
	\node[below, left, blue] at ([yshift=-0.5cm]PSorigin) {$\Sigma$};

	\coordinate (dd) at (0, -0.3, 0);
	\coordinate (S1) at ($(Sorigin)+(0, -\Splatey, 0)+(dd)$);
	\coordinate (S2) at ($(Sorigin)+(-\Splatex, -\Splatey, 0)+(dd)$);
	\node[violet] at ($(S1)!0.5!(S2)$) {$\delta$};
	\draw[violet, ->] ([yshift =0.18cm]$(S1)!0.5!(S2)$)  -- ++ (0, 0.2, 0);
	
	\coordinate (ddd) at (0, -0.52, 0);
	\coordinate (SS1) at ($(PSorigin)+(0, -\PSplatey, 0)+(ddd)$);
	\coordinate (SS2) at ($(PSorigin)+(-\PSplatex, -\PSplatey, 0)+(ddd)$);
	\draw[blue, |<->|] (SS1)--(SS2);
	\node[blue, below]  at ($(SS1)! 0.5 ! (SS2)$) {$\delta^{2-3\alpha}$};

	\pgfmathsetmacro{\taux}{1}
	\pgfmathsetmacro{\tauy}{3}
	\pgfmathsetmacro{\tauz}{1}
	\coordinate (tauo) at (-\cubex/2+0.9, 0, -1.5);
	\draw[thick, brown] (tauo) -- ++(-\taux, 0,0)-- ++ (0,-\tauy, 0) -- ++ (\taux, 0, 0)-- cycle;
	\draw[thick, brown]  (tauo) -- ++ (0,0, -\tauz) -- ++ (0,-\tauy, 0) -- ++ (0,0, \tauz)--cycle;
	\draw[thick, brown]  (tauo) -- ++ (-\taux, 0,0) -- ++ (0,0,-\tauz) -- ++ (\taux, 0,0) -- cycle;
	\draw[thick, brown, ->] ([yshift=-2cm, xshift=0.65cm] tauo) -- ++(-0.4, 0, 0);
	\node[thick, brown] at ([yshift=-2cm, xshift=0.8cm] tauo) {$\tau$};
	
	\pgfmathsetmacro{\plankx}{0.1}
	\pgfmathsetmacro{\planky}{3}
	\pgfmathsetmacro{\plankz}{1}
	\coordinate (planko) at (-\cubex/2, 0,-1.5);
	\draw[red, dashed]  (planko) -- ++(-\plankx, 0,0)-- ++ (0,-\planky, 0) -- ++ (\plankx, 0, 0)-- cycle;
	\draw[red, dashed]  (planko) -- ++ (0,0, -\plankz) -- ++ (0,-\planky, 0) -- ++ (0,0, \plankz)--cycle;
	\draw[red]  (planko) -- ++ (-\plankx, 0,0) -- ++ (0,0,-\plankz) -- ++ (\plankx, 0,0) -- cycle;
	\node[red, above] at ($(planko)+(0,0.15,0)$) {$P$};
	\draw[red, ->] ($(planko) + (0, 0.2, -0.10)$)--($(planko)+(0, 0, -0.3)$);
	
	\pgfmathsetmacro{\BSplatex}{3}
	\pgfmathsetmacro{\BSplatey}{3}
	\pgfmathsetmacro{\BSplatez}{1}
	\coordinate (BSorigin) at (0, 0, -1.5);
	\draw[thick, gray] (BSorigin) -- ++(-\BSplatex, 0,0)-- ++ (0,-\BSplatey, 0) -- ++ (\BSplatex, 0, 0)-- cycle;
	\draw[thick, gray]  (BSorigin) -- ++ (0,0, -\BSplatez) -- ++ (0,-\BSplatey, 0) -- ++ (0,0, \BSplatez)--cycle;
	\draw[thick, gray]  (BSorigin) -- ++ (-\BSplatex, 0,0) -- ++ (0,0,-\BSplatez) -- ++ (\BSplatex, 0,0) -- cycle;
	\node[thick, gray] at ([xshift = -2.5cm, yshift=-0.25cm]BSorigin) {$\widetilde{S}$};
	
	\coordinate(ss) at (-0.2, 0, -0.2);
	\coordinate (m1) at ($(BSorigin)+(-\BSplatex, 0, 0) + (ss)$);
	\coordinate (m2) at ($(BSorigin)+(-\BSplatex, 0, -\BSplatez) + (ss)$);
	\draw[gray, <->] (m1)--(m2);
	\node[left] at ($(m1)! 0.7 !(m2) $)  {$\delta^{2-3\alpha}$};
	
	\coordinate (sss) at (-0.2, 0, 0);
	\coordinate (mm1) at ($(BSorigin)+(-\BSplatex, 0, 0) + (sss)$);
	\coordinate (mm2) at ($(BSorigin)+ (-\BSplatex, -\BSplatey, 0) +(sss)$);
	\draw[gray, <->] (mm1)--(mm2);
	\node[left] at ($(mm1)!0.5!(mm2)$) {$1$};
	
	\pgfmathsetmacro{\PPSplatex}{0.35}
	\pgfmathsetmacro{\PPSplatey}{3}
	\pgfmathsetmacro{\PPSplatez}{3}
	\coordinate (PPSorigin) at (-\cubex/2+0.25,0,0); 
	\draw[green!70!black] (PPSorigin) -- ++(-\PPSplatex, 0,0)-- ++ (0,-\PPSplatey, 0) -- ++ (\PPSplatex, 0, 0)-- cycle;
	\draw[green!70!black]  (PPSorigin) -- ++ (0,0, -\PPSplatez) -- ++ (0,-\PPSplatey, 0) -- ++ (0,0, \PPSplatez)--cycle;
	\draw[green!70!black]  (PPSorigin) -- ++ (-\PPSplatex, 0,0) -- ++ (0,0,-\PPSplatez) -- ++ (\PPSplatex, 0,0) -- cycle;
	\node[below, left, green!70!black] at ([yshift=-0.5cm]PPSorigin) {$\widetilde{\Sigma}$};

	\coordinate (ssigma) at (0.1, 0.2, 0.1);
	\coordinate (sigma1) at ($(PPSorigin)	+(-\PPSplatex, 0, -\PPSplatez) + (ssigma)$);
	\coordinate (sigma2) at ($(PPSorigin)+(0, 0, -\PPSplatez)+(ssigma)$);
	\draw[green!70!black, <->] (sigma1)--(sigma2);
	\node[green!70!black] at ($(sigma1)! 0.7 ! (sigma2) +(ssigma)$) {$\delta^{3\alpha-1}$};

	\draw[violet] (Sorigin) -- ++(-\Splatex, 0,0)-- ++ (0,-\Splatey, 0) -- ++ (\Splatex, 0, 0)-- cycle;
	\draw[violet]  (Sorigin) -- ++ (0,0, -\Splatez) -- ++ (0,-\Splatey, 0) -- ++ (0,0, \Splatez)--cycle;
	\draw[violet]  (Sorigin) -- ++ (-\Splatex, 0,0) -- ++ (0,0,-\Splatez) -- ++ (\Splatex, 0,0) -- cycle;


	%
	%

	\end{tikzpicture}
	
	\end{center}
	\bigskip

	1. Pigeonholing the  parameters $M_1,M_2$
	\\
	
	We fix $M_1,M_2$ with $M_1M_2\lesssim M$ such that there are $M_1$ intervals $\widetilde{J}$ of length $\delta^{3\alpha-1}$ each containing $\sim M_2$ relevant intervals $I$. We keep the corresponding plates $S$ and discard the other ones.
	\\
	
	2. Pigeonholing the parameters $N_2,M_{22}$
	\\
	
	Fix a fat plate $\Sigma$ associated with some $J$ and fix a contributing $\widetilde{J}\subset J$. The number of such $\widetilde{J}$ inside $J$ will not enter our computations.
	Tile $\Sigma$ with $(\delta^{3\alpha-1},1,1)$-plates $\widetilde{\Sigma}$ with direction $\t(\widetilde{J})$. Recall that $1/2<\alpha<2/3$, so we have $\delta^{3\alpha-1}<\delta^{2-3\alpha}$. Note that plates $\widetilde{\Sigma}$ associated with different $\widetilde{J}$ have different orientations.

	There are $\sim NM_2$ plates $S\subset \Sigma$ associated with intervals $I\subset \widetilde{J}$. Call them  $\S_{\widetilde{J},\Sigma}$.  Each $S\in \S_{\widetilde{J},\Sigma}$ will fit inside one such $\widetilde{\Sigma}$. We fix dyadic parameters $N_2,M_{22}$ satisfying  $N_2\le N$ and $M_{22}\le M_2$. We only keep those $\widetilde{\Sigma}$ that contain $\sim N_2M_{22}$ plates $S$, with $\sim N_2$ plates for each of $\sim M_{22}$ directions from among the $\sim M_2$ directions in $\widetilde{J}$.  Note that there are at most $\frac{N}{N_2}\frac{M_2}{M_{22}}$ such plates $\widetilde{\Sigma}\subset \Sigma$ for each contributing $\widetilde{J}$.
	
	At the end of this step, the initial parameter $M_1$ may be thought as becoming smaller, as we only keep those $\widetilde{J}$ with parameters $(N_2,M_{22})$. We also restrict the plates $S$ accordingly. We denote by $\S_{\widetilde{\Sigma}}$ the $\sim N_2M_{22}$ plates $S$ inside $\widetilde{\Sigma}$.

	\bigskip

	We estimate $|\Qc_r(\S_1,\S_2,\S_3)|$ in two ways, see steps 10 and 13. The first method is a direct argument that does not make use of induction.
	
	\bigskip

	3. Replacing plates $S$ with planks $P$. Pigeonholing the  parameter $E_2$.
	\\
	
	We cover the plates $S\in\S_{\widetilde{\Sigma}}$ with  $(\delta,\delta^{2-3\alpha},1)$-planks $P$ parallel to $\widetilde{\Sigma}$. The dimensions of these planks  are suggested by  Lemma \ref{vv4secondversion}. Indeed, note that in this case $D=\delta^{3\alpha-1}\lesssim\delta^{1/2}$ (since $\alpha\ge \frac12$), as required in Lemma \ref{vv4secondversion}. 
	We restrict the analysis to $E_2$-planks $P$, those intersecting $\sim E_2$ plates $S\in\S_{\widetilde{\Sigma}}$. We have
	\begin{equation}
	\label{vv17}
	E_2\lesssim M_{22}.
	\end{equation}
	 From now on, we investigate the incidences between the $E_2$-planks. We write
	$$E_1=\frac{r}{E_2}$$
	\\
	
	4. Estimating the number of planks $P$ inside $\Sigma$ using linear Kakeya
	\\
	
	Let $\widetilde{\Sigma}$ be associated with $\widetilde{J}$.
	
	The linear Kakeya estimate in Proposition \ref{ccc1} combined with Chebyshev's inequality shows that  the number of planks $P$ inside   $\widetilde{\Sigma}$ is bounded by $$ \lessapprox  (\frac{N_2}{E_2})^2M_{22}W.$$
	
	Recall that there are at most $\frac{N}{N_2}\frac{M_2}{M_{22}}$   plates $\widetilde{\Sigma}$ associated with $\widetilde{J}$ inside $\Sigma$. Combining these last two estimates, we find that the number $N_0$ of planks $P$ inside $\Sigma$ satisfies
	\begin{equation}
	\label{vv10}
	N_0\lessapprox \frac{N}{N_2}\frac{M_2}{M_{22}}(\frac{N_2}{E_2})^2M_{22}W=\frac{NN_2M_2W}{E_2^2}.
	\end{equation}
	From now on, we restrict attention to those $\Sigma$ associated with a dyadic number $N_0$.
	\\
	
	5. Pigeonholing $N_1$ and the boxes $\tau$
	\\
	
	Fix $\widetilde{J}\subset J$ and $\Sigma$ associated with $J$. We tile $\Sigma$ with $(\frac1W,\frac1W,1)$-boxes $\tau$ so that each plank $P$ associated with $\widetilde{J}$ fits inside some $\tau$. We only keep those $\tau$ which contain $\sim N_1$ planks $P$, and call them $N_1$-rich.
	
	We caution that the tiling into boxes $\tau$ is identical for all $\widetilde{J}\subset J$, as it  follows using simple geometry. This allows  repetitions of a given $\tau$. So each $\tau$ may be $N_1$-rich for some $\widetilde{J}$ and not rich for some other $\widetilde{J}$. Multiplicity brings no harm to  the forthcoming argument, as the bilinear Kakeya inequality works just as fine at this level of generality.
	\\
	
	6. An upper bound for $N_1$
	\\
	
	A double counting argument shows that
	
	\begin{equation}
	\label{vv11}
	N_1\lesssim \frac{NM_2}{E_2}.
	\end{equation}
	
	7. The plates $\widetilde{S}$
	\\
	
	The boxes $\tau$ are periodic in the $x$ direction. For each  $\tau$  we tile $[0,1]^3$ with $(1,\frac1W,1)$-plates $\widetilde{S}$, so that each $\tau$ fits inside some $\widetilde{S}$. Note that these are not Vinogradov plates. We only keep those $\widetilde{S}$ containing some $N_1$-rich $\tau$. There are $\lesssim \frac{N_0}{N_1}$ such parallel plates $\widetilde{S}$ for each of the $M_1$ intervals $\widetilde{J}$.
	
	In line with an earlier observation, we note that the tiling with plates $\widetilde{S}$ is the same for all $\widetilde{J}\subset J$. Consequently, a plate $\widetilde{S}$ is allowed to have multiplicity.
	\\
	
	8. Counting $\frac1W$-cubes using bilinear Kakeya
	\\
	
	We partition $[0,1]^3$ into a family of cubes $Q$ with side length $\frac1W$. We classify these cubes according to a new dyadic parameter $\widetilde{M}\le M_1$ which represents the  minimum number of plates $\widetilde{S}$ that intersects them,  for each of the three broad directions.

	Then by the bilinear Kakeya inequality for the  plates $\widetilde{S}$, the number of these cubes is
	$$\lesssim (\frac{{N}_0M_1}{{N}_1\widetilde{M}})^2W.$$
	\\
	
	9. Counting trilinear $E_1$-rich $\delta$-cubes inside a $\frac1W$-cube $Q$
	\\
	
	Inside each $\frac1W$-cube $Q$, we estimate the number of trilinear $E_1$--rich $\delta$-cubes for the $E_2$-planks  using  trilinear Kakeya by
	$$\lesssim (\frac{N_1 \widetilde{M}}{E_1})^3.$$
	
	Combining this with the estimate from the previous step, we conclude that the number of trilinear $E_1$-rich $\delta$-cubes lying inside trilinear $\widetilde{M}$-rich cubes $Q$ is
	
	\begin{equation}
	\label{vv12}
	\les ({N_0} M_1)^2 \frac{{N_1}\widetilde{M}}{E_1^3}W.
	\end{equation}
	\\
	
	10. The first upper bound for $\Qc_r(\S_1,\S_2,\S_3)$
	\\
	
	We combine \eqref{vv10}, \eqref{vv11} and \eqref{vv12} to write
	\begin{equation}\label{first}
	|\Qc_r(\S_1,\S_2,\S_3)|\lesssim (\frac{MN}{r})^3{(\frac{N_2}{E_2})^2}{W^3}.
	\end{equation}
	\bigskip
	
	Now we proceed to the second estimate for  $|\Qc_r(\S_1,\S_2,\S_3)|$. This involves two steps. First we estimate the number of larger cubes $\Omega$ which are intersected by many plates $\widetilde{\Sigma}$. In the second step, we use the classical trilinear Kakeya inequality to estimate the number of $\delta$-cubes inside each $\Omega$.
	\\

	11. Counting the trilinear $\widetilde{r}$-rich $\widetilde{\delta}$-cubes $\Omega$ using the induction hypothesis
	\\
	
	The plates  $\widetilde{\Sigma}$ have thickness $\widetilde{\delta}=\delta^{-1+3\alpha}$ and are periodic in $x$--direction with periodicity $$\frac{1}{\widetilde{W}}=\widetilde{\delta}^{-3\widetilde{\alpha}+2} = \delta^{-3\alpha+2}=\frac1W.$$Hence $\widetilde{\alpha}=\frac{9\alpha-4}{3(3\alpha-1)}$.
    Recall that the normal to $\widetilde{\Sigma}$ is $\t(\widetilde{J})$, so each $\widetilde{\Sigma}$ is a Vinogradov plate.

Let us fix a dyadic parameter
\begin{equation}
\label{vv16}
\widetilde{r}\gtrsim E_1.\end{equation} In this step we count the number of trilinear $\widetilde{r}$-rich $\widetilde{\delta}$-cubes $\Omega$ with respect to the plates $\widetilde{\Sigma}$. Recall that the number of parallel plates $\widetilde{\Sigma}$ in a box of width $\frac1{\widetilde{W}}$ is $\lesssim \frac{NM_2}{N_2M_{22}}$ and that there are at most $M_1$  directions for these plates.
\medskip

The number of $\widetilde{\delta}$-cubes is at most
	\begin{equation}
	\label{vv14}
	\lesssim C(\widetilde{\delta},\widetilde{\alpha}) (\frac{M_2M_1N}{M_{22}N_2\widetilde{r}^2\widetilde{\delta}})^{\frac{4-6\widetilde{\alpha}}{3\widetilde{\alpha}-1}}(\frac{MN}{\widetilde{r}M_{22}N_2})^3W\lesssim C(\widetilde{\delta},\widetilde{\alpha}) (\frac{MN}{M_{22}N_2\widetilde{r}^2\widetilde{\delta}})^{\frac{4-6\widetilde{\alpha}}{3\widetilde{\alpha}-1}}(\frac{MN}{\widetilde{r}M_{22}N_2})^3W.
	\end{equation}
	\\
	
	12. Counting trilinear $E_1$-rich $\delta$-cubes  inside  $\Omega$
	\\
	
	Fix $\Omega$ as in the previous step. We need an upper bound for the number of trilinear  $E_1$-rich cubes inside $\Omega$, with respect to the $E_2$-planks. Since $E_1E_2=r$,  this number is certainly smaller than the number of  trilinear $r$-rich $\delta$-cubes inside $\Omega$,  with respect to the original plates $S$. So we choose to estimate this latter number instead.
	
	 Recall that the thin plates $S$ are packed inside  fat  plates $\widetilde{\Sigma}$.  Each $\widetilde{\Sigma}$ contains $\sim M_{22}N_2$ plates $S$. Thus, there are $\sim \widetilde{r}M_{22}N_2$ plates $S$ intersecting $\Omega$

	We now apply the trilinear Kakeya inequality \eqref{vvv3}. The  number of trilinear $r$-rich $\delta$-cubes inside  $\Omega$ is bounded by
		\begin{equation}
	\label{vv15}
	(\frac{\widetilde{r}M_{22} N_2}{r})^3.
	\end{equation}
	\\
	
	13. The second upper bound for $\Qc_r(\S_1,\S_2,\S_3)$
	\\
	
	Combining \eqref{vv14} and \eqref{vv15} we find

	\begin{align*}
	|\Qc_r(\S_1,\S_2,\S_3)|&\lesssim  C(\widetilde{\delta},\widetilde{\alpha}) (\frac{MN}{M_{22}N_2\widetilde{r}^2\widetilde{\delta}})^{\frac{4-6\widetilde{\alpha}}{3\widetilde{\alpha}-1}}(\frac{MN}{\widetilde{r}M_{22}N_2})^3(\frac{\widetilde{r}M_{22} N_2}{r})^3W\\&=(\frac{MN}{M_{22}N_2\widetilde{r}^2\widetilde{\delta}})^{\frac{4-6\widetilde{\alpha}}{3\widetilde{\alpha}-1}}(\frac{MN}{r})^3W.
	\end{align*}
    Since $\frac{4-6\widetilde{\alpha}}{3\widetilde{\alpha}-1}=\frac{4-6\alpha}{6\alpha-3}$, we rewrite this as
	\begin{equation}\label{second}
	|\Qc_r(\S_1,\S_2,\S_3)|\lesssim C(\widetilde{\delta},\widetilde{\alpha})(\frac{MN}{M_{22}N_2\widetilde{r}^2\widetilde{\delta}})^{\frac{4-6\alpha}{6\alpha-3}}(\frac{MN}{r})^3W.
	\end{equation}
	\\
	
	14. Combining the two estimates for $\Qc_r(\S_1,\S_2,\S_3)$
	\\
	
	Taking a geometric average of the upper bounds \eqref{first} and ~\eqref{second} leads to
	\begin{align*}
	|\Qc_r(\S_1,\S_2,\S_3)|&\lessapprox W(\frac{MN}{r})^3 \min\left( C(\widetilde{\delta},\widetilde{\alpha})(\frac{MN}{M_{22}N_2\widetilde{r}^2\widetilde{\delta}})^{\frac{4-6\alpha}{6\alpha-3}},(\frac{N_2W}{E_2})^2\right)
	\\&\le \max(C(\widetilde{\delta},\widetilde{\alpha}),1)W(\frac{MN}{r})^3 (\frac{MN}{M_{22}N_2\widetilde{r}^2\widetilde{\delta}})^{\frac{4-6\alpha}{6\alpha-3}\frac{6\alpha-3}{3\alpha-1}}(\frac{N_2W}{E_2})^{\frac{2(2-3\alpha)}{3\alpha-1}}
	\\& =\max(C(\widetilde{\delta},\widetilde{\alpha}),1) W(\frac{MN}{r})^3(\frac{M}{M_{22}}\frac{WN}{E_2\widetilde{r}^2\widetilde{\delta}})^{\frac{4-6\alpha}{6\alpha-3}}\\&\lesssim \max(C(\widetilde{\delta},\widetilde{\alpha}),1) W(\frac{MN}{r})^3(\frac{NM}{r^2\delta})^{\frac{4-6\alpha}{6\alpha-3}}
	\end{align*}
	
In the last inequality we combined  \eqref{vv17} and \eqref{vv16} to write
$$\frac{1}{M_{22}E_2\widetilde{r}^2}\lesssim \frac{1}{r^2}.$$
The inequality \eqref{ccc3} is now verified.	

\end{proof}

\section{Refined $l^p$ decoupling  at canonical scale for the moment curve}

Let us consider the moment curve $\Gamma$  in $\R^3$
$$\gamma(t)=(t,t^2,t^3),\;t\in[0,1].$$
For each $\delta<1$ and each interval $H\subset [0,1]$ we introduce
the anisotropic neighborhood of the arc $\Gamma_H$
\smallskip

$$\Gamma_H(\delta)=\{(\xi_1,\xi_2,\xi_3):\;\xi_1\in H,\;|\xi_2-\xi_1^2|\le \delta^2,\; |\xi_3-3\xi_1\xi_2+2\xi_1^3|\le \delta^3\}.$$
\smallskip
\\
When $H=[0,1]$ we will  write $\Gamma(\delta)$.
\\
\\

We consider the partition of $\Gamma(\delta)$ into almost rectangular boxes $\Gamma_{I}(\delta)$, $I\in\I_\delta$, with dimensions $(\delta,\delta^2,\delta^3)$. We call these boxes $\theta$.

The following result is a close relative of the main theorem from \cite{BDG}. Its proof appears in \cite{Dembook}. This is a decoupling for boxes of canonical scale.

\begin{te}
	\label{J1}
	Assume that $F:\R^3\to \C$ has spectrum inside $\Gamma(\delta)$. Then for $2\le p\le 12$
	\begin{equation}
	\label{twi12}
	\|F\|_{L^{p}(\R^3)}\lesssim_\epsilon\delta^{-\epsilon}(\sum_{\theta}\|\Pc_\theta F\|^2_{L^{p}(\R^3)})^{1/2}.
	\end{equation}
\end{te}

We will use planks for  wave packet decompositions. These are  boxes introduced in the beginning of Chapter \ref{Ch:improvedplates}. The planks dual to the boxes $\theta$ are the subject of the following definition.

\begin{de}[Vinogradov planks]
A $(\delta^{-1},\delta^{-2},\delta^{-3})$-plank $P$ is called a Vinogradov plank associated with an interval $I\in\I_\delta$ if its long axis  points in the direction $\b(I)$ and the normal to the $(\delta^{-2},\delta^{-3})$-face is in the direction  $\t(I)$.

\end{de}

We recall from Chapter \ref{Ch:improvedplates} that the face with dimensions $(\delta^{-1},\delta^{-2})$ is parallel to the $xy$ plane. Each Vinogradov plank is an almost rectangular box.  

 We will frequently use the spatial scale $R=\delta^{-3}$, as in the following lemma.
\begin{lem}
\label{smallestpolar}	
Consider an interval $J\subset [0,1]$ of length $\sigma\ge \delta$. For each $I\in\I_\delta(J)$ let $P_I$ be a Vinogradov plank associated with $I$ and containing the origin.

There is  a rectangular box $B$ centered at the origin, containing all these planks $P_I$, and with dimensions $\sim(R\sigma^2,R\sigma,R)$ with respect to the axes $(\t(J),\n(J),\b(J))$.

Let $\sigma=\delta^{1/3}=R^{-1/9}$. If all planks $P_I$ are centered at the origin, the intersection of all $P_I$ is an almost rectangular box with dimensions $(R^{1/3},R^{4/9},R^{5/9})$.

\end{lem}
\begin{proof}
We will only prove the first part, the second part follows via a similar argument.
	
Consider the family of linear maps  $A_{\sigma,a}$ on $\R^3$  given by $(x',y',z')=A_{\sigma,a}(x,y,z)$ with
\begin{equation}
\label{rescalingdd}
\begin{cases}x'=\sigma(x+2ay+3a^2z)\\ y'=\sigma^2(y+3az)\\z'=\sigma^{3}z
\end{cases}.
\end{equation}
If $J=[a,a+\sigma]$, then $A_{1,a}(P_I)$ is a Vinogradov plank containing the origin,  associated with the interval $I-a$. Because of this, it suffices to assume $J=[0,\sigma]$. We take $B$ to be parallel to the $x,y,z$ axes.

Let $c\in I$. Recall that any point  $(x',y',z')\in P_I$  has coordinates $(O(R^{1/3}),O(R^{2/3}),O(R))$ with respect to the vectors $\t(c),\n(c),\b(c)$, whose $(x,y,z)$  coordinates up to scaling are
 $(1,2c,3c^2)=(1,O(\sigma),O(\sigma^2))$, $(2c+9c^3,9c^4-1,-3c-6c^3)=(O(\sigma),O(1),O(\sigma))$ and $(3c^2,-3c,1)=(O(\sigma^2),O(\sigma),1)$. It is now immediate that $x'=O(R\sigma^2)$, $y'=O(R\sigma)$ and $z'=O(R)$, so $P_I\subset B$.

\end{proof}

The proof of the following wave packet decomposition is standard,  see for example  Exercise 2.7 in \cite{Dembook}.

\begin{te}[Wave packet decomposition at scale $\delta$]
\label{WPdeco}	
Fix a scale $\delta<1$ and $F$ with spectrum in  $\Gamma(\delta)$. There is a decomposition
$$F=\sum_\theta\Pc_\theta F=\sum_{P\in\P_\delta(F)}F_P$$ where $\P_\delta(F)$ is a collection of Vinogradov $(\delta^{-1},\delta^{-2},\delta^{-3})$-planks, such that
\\
\\
(W1)\;\; each $\widehat{F_P}$ is supported on  $2\Gamma_{I}(\delta)$ for some $I\in\I_\delta$, and $P$ is associated with $I$. We denote by $\P_I(F)$ the corresponding planks, so $\Pc_\theta F=\sum_{P\in\P_I}F_P$ if $\theta=\Gamma_I(\delta)$.
\\
\\
(W2)\;\; $F_P$ is spatially concentrated near $P$, in the sense that for each $M\ge 1$
$$|F_P(x,y,z)|\lesssim_M\|F_P\|_\infty\chi_P^M(x,y,z).$$
Moreover, for each $p\ge 1$
$$\|F_P\|_p\sim \|F_P\|_\infty |P|^{1/p}.$$
\\
\\
(W3)\;\; for each $p\ge 2$ and each $\P_1\subset \P_2\subset \P_I(F)$ such that $\|F_P\|_\infty\sim const$ for $P\in\P_1$, we have
$$\|\sum_{P\in\P_1}F_P\|_{L^p(\R^3)}\lesssim \|\sum_{P\in\P_2}F_P\|_{L^p(\R^3)}$$
\\
\\
(W4)\;\; for each $p\ge 2$ and each $\P_1\subset  \P_I(F)$ such that $\|F_P\|_\infty\sim const$ for $P\in\P_1$, we have
$$\|\sum_{P\in\P_1}F_P\|_{L^p(\R^3)}\sim (\sum_{P\in\P_1}\|F_P\|^p_{L^p(\R^3)})^{1/p}.$$
\end{te}
A fair enough representation of $F_P$ is
$$F_P(x,y,z)\approx \|F_P\|_\infty 1_P(x,y,z)e((x,y,z)\cdot (c,c^2,c^3))$$
where $c$ is some (irrelevant) point in $I$.
\smallskip

The following is the extension of Theorem \ref{t4} to the moment curve.
\begin{te}
\label{ccc24}	
Assume that the spectrum of $F$ is inside $\Gamma_{R^{-1/3}}$.
	Let $\Qc$ be a collection of pairwise disjoint cubes  $q$ in $\R^3$, with side length $R^{1/3}$.
	Assume that each $q$ intersects at most $M$ fat planks $R^\Delta P$ with $P\in\P_{R^{-1/3}}(F)$, for some $M\ge 1$ and $\Delta>0$.
	
	Then for each $2\le p\le 12$ and $\epsilon>0$ we have	
	
	$$\|F\|_{L^p(\cup_{q\in\Qc}\chi_q)}\lesssim_{\Delta,\epsilon}R^{\epsilon}M^{\frac12-\frac1p}(\sum_{P\in\P_{R^{-1/3}}(F)}\|F_P\|^p_{L^p(\R^3)})^{\frac1p}.$$
\end{te}	
\begin{proof}
The proof is very similar to that of Theorem \ref{t4}. We sketch it briefly, ignoring the Schwartz-type technicalities relevant to the scale $\Delta$.

The argument involves induction on the scale $R$. Let us assume we have verified the claim for the smaller scale $R^{2/3}$. To verify it for scale $R$, let $F$ be as in the hypothesis of our theorem.

Partition $[0,1]$ into intervals $J$ of length $R^{-1/9}$. For each $J$, cover $\R^3$ with rectangular boxes $B$  with dimensions $\sim(R^{7/9},R^{8/9},R)$ with respect to the axes $(\t(J),\n(J),\b(J))$. The relevance of this choice comes from Lemma \ref{smallestpolar} with $\sigma=R^{-1/9}$. Each plank $P\in\P_{R^{-1/3}}(F)$ will lie inside some $B$. Call $\P_B$ the collection of all these planks. We will say that $B$ is associated with $J$.

Let $M_1,M_2$ be dyadic parameters with $M_1M_2\le M$.
We may restrict attention to the family of those cubes $q\in\Qc$ which are intersected by $\sim M_1$ boxes $B$, and by $M_2$ planks $P$ from each family  $\P_B$. We write $q\sim B$ to denote this special relation. The remaining boxes $B$ will contribute negligibly to $q$ and will be ignored.

First, using (the local version of) Theorem \ref{J1} (with $\delta$ replaced with $\delta^{1/3}$), we may write for each such $q$
$$\|F\|_{L^p(\chi_q)}\lesssim_{\epsilon}R^\epsilon M_1^{\frac12-\frac1p} (\sum_{B:\;q\sim B}\|\sum_{P\in\P_B}F_P\|^p_{L^p(\chi_q)})^{1/p}.$$
We have used H\"older's inequality and the fact  that for each $J$ as above we have  $\cup_{I\in\I_{\delta}(J)}\Gamma_I(\delta)\subset \Gamma_J(\delta^{1/3})$. In particular, the spectrum of $\sum_{P\in\P_B}F_P$ lies inside $\Gamma_J(\delta^{1/3})$ whenever $B$ is associated with  $J$. Summing over $q$ we find
\begin{equation}
\label{fbvgerufiehuih}
\|F\|_{L^p(\sum \chi_q)}\lesssim_{\epsilon}R^\epsilon M_1^{\frac12-\frac1p} (\sum_B\|\sum_{P\in\P_B}F_P\|^p_{L^p(\sum_{q\sim B}\chi_q)})^{1/p}.
\end{equation}
Let us now fix $B$. Recall that each $q$ with $q\sim B$ is intersected by $\sim M_2$ planks $P\in\P_B$. We tile $B$ with rectangular boxes $\tau$ with the same orientation as $B$ and dimensions $(R^{1/3},R^{4/9},R^{5/9})$. This choice is suggested by the second part of  Lemma \ref{smallestpolar}. We replace the family of cubes $q$ with the family of boxes $\tau$ which covers them.

Assume that $B$ is associated  with the interval $J=[a,a+R^{-1/9}]$. Let $A$ be the map in \eqref{rescalingdd} corresponding to $a$ and $\sigma=R^{-1/9}$. Note that $A$ maps each $\tau$ to an $R^{2/9}$-cube $\tilde{q}$ and each plank $P\in\P_B$ to a Vinogradov $(R^{2/9},R^{4/9},R^{2/3})$-plank $\tilde{P}$. Each $\tilde{q}$ will intersect $\lesssim M_2$ such planks  $\tilde{P}$.

We aim to use the induction hypothesis for $\tilde{F}=(\sum_{P\in\P_B}F_P)\circ A^{-1}$ at scale $R^{2/3}$. We write $\tilde{F}_{\tilde{P}}=F_P\circ A^{-1}$. Thus
\begin{align*}
\|\sum_{P\in\P_B}F_P\|_{L^p(\sum_{q\sim B}\chi_q)} ^p&= R^{2/3}\|\tilde{F}\|_{L^p(\sum_{\tilde{q}}\chi_{\tilde{q}})}^p\\&\lesssim_\epsilon R^{2/3+\epsilon}M_2^{\frac{p}{2}-1}\sum_{\tilde{P}}\|\tilde{F}_{\tilde{P}}\|^p_{L^p(\R^3)}\\&=R^{\epsilon}M_2^{\frac{p}{2}-1}\sum_{{P}\in\P_B}\|{F}_{P}\|^p_{L^p(\R^3)}.
\end{align*}
It now suffices to combine this with \eqref{fbvgerufiehuih}.

\end{proof}

\section{Proof of Theorem \ref{n=3partialrange} in the range $0<\beta\le 1$}

\bigskip

We begin with a trilinear-to-linear reduction.
\smallskip

Fix $\frac13<\alpha\le \frac12$. We denote by $Q_R$ an arbitrary cube in $\R^3$ with side length $R$.
Let $\Dec(R,p,\alpha)$ be the smallest constant such that the inequality
$$\|\sum_{j=1}^{R^\alpha}a_je(x\frac{j}{R^\alpha}+y\frac{j^2}{R^{2\alpha}}+z\frac{j^3}{R^{3\alpha}})\|_{L^p_{\sharp}(Q_R)}\le \Dec(R,p,\alpha) R^{\frac\alpha2}$$
holds true for each cube $Q_R$ as above and each $a_j\in \C$ with $|a_j|=1$. Our task is to prove that for $p= 6+\frac2\alpha$ we have
$$\Dec(R,p,\alpha)\lesssim_\epsilon R^{\epsilon}.$$
We will achieve this by relating $\Dec(R,p,\alpha)$ to its trilinear counterpart. This argument is standard, but we include it for reader's convenience.
\medskip

Fix a  parameter $K=O(1)$, to be chosen large enough. Consider the partition of $[0,1]$ into $K$ intervals $I\in\Ic$ of length $\frac1K$.
Let $\TD(R,p,\alpha)$ be the smallest constant such that the inequality
$$\|(\prod_{i=1}^3\sum_{\frac{j}{R^\alpha}\in I_i}a_je(x\frac{j}{R^\alpha}+y\frac{j^2}{R^{2\alpha}}+z\frac{j^3}{R^{3\alpha}}))^{1/3}\|_{L^p_{\sharp}(Q_R)}\le \TD(R,p,\alpha) R^{\frac\alpha2}$$
holds true for each cube $Q_R$, each $a_j\in \C$ with $|a_j|=1$ and each triple of pairwise non-adjacent intervals $I_1,I_2,I_3\in\Ic$.

\begin{pr}
	\label{p1}	
	Assume 
	\begin{equation}
	\label{fyryufioefiuyrdodpesivuuifiep[oyrtuio}
	4<\frac{p}{2}+\frac1\alpha.
	\end{equation}  
	There exists a constant  $C_K$ independent of $R$ such that  $$\Dec(R,p,\alpha)\le \Dec(\frac{R}{K^{\frac1\alpha}},p,\alpha)+ C_K\TD(R,p,\alpha).$$
\end{pr}
\begin{proof}
	
	Fix $a_j\in\C$ with unit modulus and fix $Q_R$. For each interval $I$, let
	$$\mathfrak{E}_I(x,y,z)=\sum_{\frac{j}{R^\alpha}\in I}a_je(x\frac{j}{R^\alpha}+y\frac{j^2}{R^{2\alpha}}+z\frac{j^3}{R^{3\alpha}}).$$
	We write $I_1\not\sim I_2\not\sim I_3$ if $I_1,I_2,I_3\in\Ic$ are pairwise non-adjacent.
	Note that
	$$|\mathfrak{E}_{(0,1]}(x,y,z)|\le 100\max_{I\in\Ic}|\mathfrak{E}_I(x,y,z)|+K^{100}\max_{I_1\not\sim I_2\not\sim I_3}|\mathfrak{E}_{I_1}(x,y,z)\mathfrak{E}_{I_2}(x,y,z)\mathfrak{E}_{I_3}(x,y,z)|^{\frac13},$$
	so
	\begin{equation}
	\label{eeeeeee1}
	|\mathfrak{E}_{(0,1]}(x,y,z)|^p\lesssim \sum_{I\in\Ic}|\mathfrak{E}_I(x,y,z)|^p+K^{O(1)}\sum_{I_1\not\sim I_2\not\sim I_3}|\mathfrak{E}_{I_1}(x,y,z)\mathfrak{E}_{I_2}(x,y,z)\mathfrak{E}_{I_3}(x,y,z)|^{\frac{p}3}.
	\end{equation}
	
	Let us analyze a term from the first sum. Fix $I$ with left endpoint $c+\frac1{R^\alpha}=\frac{j_0+1}{R^{\alpha}}$. Let $\Lambda$ consist of the points $\lambda=K(\frac{j}{R^\alpha}-c)$ with $\frac{j}{R^{\alpha}}\in I$. Define $b_\lambda=a_j$. Note that
	$$\mathfrak{E}_I(x,y,z)=e(cx+c^2y+c^3z)\sum_{\lambda\in\Lambda}b_\lambda e(\lambda\frac{x+2cy+3c^2z}{K}+\lambda^2\frac{y+3cz}{K^2}+\lambda^3\frac{z}{K^3}).$$
	Let $R'=RK^{-\frac1\alpha}$. The points in $\Lambda$ are of the form $\frac{j}{(R')^\alpha}$, with $1\le j\le (R')^\alpha$. The image $P_R$ of $Q_R$ under the map
	$$(x,y,z)\mapsto (\frac{x+2cy+3c^2z}{K},\frac{y+3cz}{K^2},\frac{z}{K^3})$$
	lies inside a rectangular box with dimensions $\frac{10R}{K},\frac{10R}{K^2},\frac{R}{K^3}$.
Thus, it can be covered with a finitely overlapping family $\Qc$ consisting of  roughly $ K^{\frac2\alpha-3}$ cubes $Q$ with diameter $R'$.  That is since the height $\frac{R}{K^3}$ of $P_R$ is smaller than $R'$, a consequence of our assumption $\alpha> \frac13$. It is worth noting that the squares $Q$ cover a larger area than that of $P_R$. This however does not turn out to be lossy.
	
	We thus have, after a change of variables
	\begin{align*}
	\int_{Q_R}|\mathfrak{E}_I(x)|^pdx&=K^6\int_{P_R}|\sum_{\lambda\in\Lambda}b_\lambda e(\lambda y_1+\lambda^2y_2+\lambda^3y_3)|^pdy_1dy_2dy_3\\&\le K^6\sum_{Q\in\Qc}\int_Q|\sum_{\lambda\in\Lambda}b_\lambda e(\lambda y_1+\lambda^2y_2+\lambda^3 y_3)|^pdy_1dy_2dy_3\\&\le K^6\sum_{Q}|Q| (R')^{\frac{p\alpha}{2}}\Dec(R',p,\alpha)^p\\&\lesssim K^6K^{\frac2\alpha-3}(R')^3(\frac{R^\alpha}{K})^{\frac{p}{2}}\Dec(R',p,\alpha)^p\\&=K^{3-\frac1\alpha-\frac{p}{2}}R^{3+\frac{\alpha p}{2}}\Dec(RK^{-\frac1\alpha},p,\alpha)^p.
	\end{align*}
	Along the way we have used the definition of $\Dec(R',p,\alpha)$ on each $Q$. 
	Since $4<\frac{p}{2}+\frac1\alpha$ we conclude by summing over all $I\in\Ic$ that
	\begin{equation}
	\label{eeeeeee2}
	\int_{Q_R}\sum_{I\in \Ic}|\mathfrak{E}_I|^p\le CR^{3+\frac{\alpha p}{2}}\Dec(RK^{-\frac1\alpha},p,\alpha)^p.
	\end{equation}
	The constant $C$ can be chosen as small as we wish, if $K$ is chosen large enough.  
 	\smallskip
	
	Next we analyze a term from the second sum in \eqref{eeeeeee1}. It is immediate that
\begin{equation}
\label{eeeeeee3}	
\int_{Q_R}|\mathfrak{E}_{I_1}\mathfrak{E}_{I_2}\mathfrak{E}_{I_3}|^{\frac{p}3}\le R^{3+\frac{\alpha p}{2}}\TD(R,p,\alpha)^p.
\end{equation}
	
The combination  of \eqref{eeeeeee1}-\eqref{eeeeeee3} concludes the argument.

\end{proof}
The following corollary follows by iterating the inequality in the previous proposition. 
\begin{co}	
	Assume $\TD(R,p,\alpha)\lesssim_\epsilon R^\epsilon$ for each $\epsilon>0$. Then $$\Dec(R,p,\alpha)\lesssim_\epsilon R^\epsilon.$$
\end{co}

\bigskip

Next, we focus on proving that $\TD(R,p,\alpha)\lesssim_\epsilon R^\epsilon$ if $p=6+\frac2\alpha$ and $\frac13<\alpha\le \frac12$.
\smallskip

Let $\eta:[-\frac1{10},\frac1{10}]^3\to\R$ be a Schwartz function
and let $\eta_R(\xi)=R^3\eta({\xi}{R})$.
To simplify notation, for an interval $H$ we write $\Pc_HF=\Pc_{H\times\R^2}F$.
For each $F:\R^3\to\C$ we write $F_1=\Pc_{[0,1/6]}F$, $F_2=\Pc_{[1/3,1/2]}F$ and $F_3=\Pc_{[2/3,1]}F$.

\smallskip

Our main result in this section is the following small cap decoupling for special functions with spectrum near the moment curve.
\begin{te}
\label{ccc7}	
Assume that $\frac13<\alpha\le \frac12$. Let $a_j\in\C$ with unit modulus. Define
\begin{equation}
\label{ccc5}
\widehat{F}(\xi)=\sum_{j=1}^{R^{\alpha}}a_j\eta_{R}(\xi_1-\frac{j}{R^\alpha},\xi_2-\frac{j^2}{R^{2\alpha}},\xi_3-\frac{j^3}{R^{3\alpha}}).
\end{equation}
Then for $p=6+\frac{2}{\alpha}$ we have, with an implicit constant independent of $a_j$
\begin{equation}
\label{ccc6}
\|(F_1F_2F_3)^{1/3}\|_{L^p([0,R]^3)}\lesssim_\epsilon R^{\alpha(\frac12-\frac1p)+\epsilon}(\sum_{J\in \I_{R^{-\alpha}}}\|\Pc_JF\|^p_{L^p(\R^3)})^{1/p}.
\end{equation}
\end{te}
Invoking standard Schwartz tail considerations, the integration domain $[0,R]^3$ can easily be replaced with $\R^3$.
Let us first observe the following immediate consequence.
\begin{co}
\label{ccc26}	
If $\frac13<\alpha\le \frac12$, $p=6+\frac2\alpha$, $|a_j|=1$ and $Q_R$ is an arbitrary cube with side length $R$	
$$\|\sum_{j=1}^{R^\alpha}a_je(x\frac{j}{R^\alpha}+y\frac{j^2}{R^{2\alpha}}+z\frac{j^3}{R^{3\alpha}})\|_{L^p_{\sharp}(Q_R)}\lesssim_{\epsilon}R^{\frac{\alpha}{2}+\epsilon}.$$
\end{co}
\begin{proof}
Let
$$\mathfrak{E}_{H,a}(x,y,z)=\sum_{\frac{j}{R^\alpha}\in H}a_je(x\frac{j}{R^\alpha}+y\frac{j^2}{R^{2\alpha}}+z\frac{j^3}{R^{3\alpha}}).$$
Invoking the trilinear-to-linear reduction explained earlier, it suffices to prove that
$$
\|(\mathfrak{E}_{[0,1/6],a}\mathfrak{E}_{[1/3,1/2],a}\mathfrak{E}_{[2/3,1],a})^{1/3}\|_{L^p_{\sharp}(Q_R)}\lesssim_{\epsilon}R^{\frac{\alpha}{2}+\epsilon}.$$	
The three intervals considered here may be replaced with any three non-adjacent intervals.  
To prove this inequality, assume $Q_R=(x_0,y_0,z_0)+[0,R]^3$. Let $b_j=a_je(x_0\frac{j}{R^\alpha}+y_0\frac{j^2}{R^{2\alpha}}+z_0\frac{j^3}{R^{3\alpha}})$ and $$\mathfrak{E}_{H,b}(x,y,z)=\sum_{\frac{j}{R^\alpha}\in H}b_je(x\frac{j}{R^\alpha}+y\frac{j^2}{R^{2\alpha}}+z\frac{j^3}{R^{3\alpha}}).$$
Note that $$
\|(\mathfrak{E}_{[0,1/6],a}\mathfrak{E}_{[1/3,1/2],a}\mathfrak{E}_{[2/3,1],a})^{1/3}\|_{L^p_{\sharp}(Q_R)}=
\|(\mathfrak{E}_{[0,1/6],b}\mathfrak{E}_{[1/3,1/2],b}\mathfrak{E}_{[2/3,1],b})^{1/3}\|_{L^p_{\sharp}([0,R]^3)}.$$

 We use \eqref{ccc6}, with a choice of $\eta$ satisfying $|\widehat{\eta}|\ge 1_{[0,1]^3}$ and with $a_j$ replaced with $b_j$.
Note that if $J=[\frac{j-\frac12}{R^\alpha},\frac{j+\frac12}{R^\alpha}]$ then $$\Pc_JF(x,y,z)=\widehat{\eta_R}(x,y,z)b_je(x\frac{j}{R^\alpha}+y\frac{j^2}{R^{2\alpha}}+z\frac{j^3}{R^{3\alpha}}).$$
Thus
$|F_1|\ge |\mathfrak{E}_{[0,1/6],b}|$, $|F_2|\ge |\mathfrak{E}_{[1/3,1/2],b}|$ and $|F_3|\ge |\mathfrak{E}_{[2/3,1],b}|$ on $[0,R]^3$.

\end{proof}
This corollary implies Theorem \ref{n=3partialrange} in the range $0<\beta\le 1$. Indeed, use $\beta=3-\frac1\alpha$, $N=R^{\alpha}$, rescaling and periodicity. 	
\medskip

The proof of Theorem \ref{ccc7} will be done in several steps.
Note that
$$F=\sum_{I\in\I_{R^{-1/3}}}\Pc_IF.$$
By splitting $F$ in two, we may assume that the sum contains no neighboring intervals $I$.

We start with a wave packet decomposition of $F$ at scale $R^{-1/3}$, as in Theorem \ref{WPdeco}
\begin{equation}
\label{ddd3}
F=\sum_{P\in\P_{R^{-1/3}}(F)}F_P.
\end{equation}
Since we are interested in estimating $(F_1F_2F_3)^{1/3}$ on $[0,R]^3$, we will assume that all planks in $\P_{R^{-1/3}}(F)$ are contained in $[0,R]^3$. For the rest of the argument we will replace the integration domain  $[0,R]^3$  with $\R^3$.
\smallskip

We partition $\P_{R^{-1/3}}(F)$ into families $\P^{(i)}$ with the following three properties. The parameters $A,N_0,N$ will depend on $i$. The uniformity assumptions in (S1) and (S3) are achieved via pigeonholing. A discussion about (S2) is included at the end of this section, see Remark \ref {ekjfewydeyfuopueuiytu38u}.
\\
\\
\textbf{Structure of $\P^{(i)}$:}
\\

(S1) (magnitude)\;$\|F_P\|_\infty\sim A$ for all $P\in\P^{(i)}$, for some dyadic parameter $A$
\\

(S2)\; ($x$ and $y$ periodicity) \; For each $I\in\I_{R^{-1/3}}$ we write $\P^{(i)}_I=\P^{(i)}\cap \P_{I}(F)$. We assume that either $\P_{I}^{(i)}=\emptyset$ or
\begin{equation}
\label{ccc8}
|\P_{I}^{(i)}|\sim N_0R^{2-3\alpha}
\end{equation}
for some dyadic integer $N_0$. If  the latter happens, we will refer to $I$ as ``contributing". The number of the contributing  intervals $I$, and thus the total number of planks in $\P^{(i)}$ will never enter our considerations.
\smallskip

If $\P_{I}^{(i)}\not=\emptyset$, then we assume that the planks in $\P_{I}^{(i)}$ are $R^\alpha$-periodic in the $x$-direction and $R^{2\alpha}$-periodic in the $y$-direction.

\smallskip

Let us understand better the structure of $\P^{(i)}_I$ in this latter case. Tile $[0,R]^3$ with $(R^{\alpha},R^{2\alpha},R)$-planks $\Sigma_0$ with axes parallel to those of a typical $P\in \P^{(i)}_I$. Each $P\in\P^{(i)}_I$ is contained in some unique $\Sigma_0$. Note that due to $R^\alpha$-periodicity in the $x$ direction and $R^{2\alpha}$-periodicity in the $y$ direction, all $\Sigma_0$ will contain the same number of planks $P\in\P^{(i)}_I$. This number must be $\sim N_0$, due to \eqref{ccc8}.
\\
	
(S3) \; For each contributing  $I$ we tile $[0,R]^3$ with $(R^{\alpha},R^{2/3},R)$-planks $\Sigma$ with axes parallel to those of a typical $P\in \P^{(i)}_I$.
Each $\Sigma$ is contained in some unique $\Sigma_0$. We will assume that there are either $\sim N$ or zero planks  $P\in\P^{(i)}_I$ inside each such $\Sigma$,  for some dyadic number $1\le N\le N_0$ independent of $I$. In the first case, we will refer to $\Sigma$ as ``contributing".  Note that $$N\lesssim R^{\alpha-\frac13}\;\text{ and }\; \frac{N_0}{N}\lesssim R^{2\alpha-\frac23}.$$
In summary, for each contributing $I$, each $\Sigma_0$ contains $\sim \frac{N_0}{N}$ contributing $\Sigma$. See Figure 4.

\bigskip

Write $F^{(i)}=\sum_{P\in\P^{(i)}}F_P$ so that $F=\sum_iF^{(i)}.$

 Standard considerations allow us to argue that only  $\lessapprox 1$ values of $i$ are significant (the small values of $A$ contribute negligibly), in particular
$$\|(F_1F_2F_3)^{1/3}\|_{L^p(\R^3)}\lessapprox \sup_{i_1,i_2,i_3}\|(F_1^{(i_1)}F_2^{(i_2)}F_3^{(i_3)})^{1/3}\|_{L^p(\R^3)}.$$
\smallskip

To ease the notation, we will analyze the case when $i_1=i_2=i_3=i$. Also, we will denote  $F^{(i)}$ by $g$, $\P^{(i)}$ by $\P$ and $\P^{(i)}_I$ by $\P_I$. Invoking  (W1) and our assumption on $F$ we have
\begin{equation}
\label{ccc11}
g=\sum_{I}\Pc_{2I}g=\sum_{I}\sum_{P\in\P_I}F_P.
\end{equation}
If $\P_I\not=\emptyset$ we will say that $I$ contributes to $g$.
Due to (W2)-(W4) in Theorem \ref{WPdeco}, (S1), (S2) and \eqref{ccc11},  we have
\begin{equation}
\label{ccc12}
\|\Pc_{2I}g\|_{L^p(\R^3)}\sim \begin{cases}A(N_0R^{4-3\alpha})^{1/p}, \text{ if }  I \text{ contributes to }g\\0, \text{ otherwise}\end{cases}.
\end{equation}
To prove Theorem \ref{ccc7} it will suffice to show that for $p=6+\frac2\alpha$
$$
\|(g_1g_2g_3)^{1/3}\|_{L^p(\R^3)}\lesssim_\epsilon R^{\alpha(\frac12-\frac1p)+\epsilon}(\sum_{J\in \I_{R^{-\alpha}}}\|\Pc_JF\|^p_{L^p(\R^3)})^{1/p}.
$$
This will immediately follow from combining two results, in line with our two-step decoupling philosophy.

The first one is about decoupling  $I$ into intervals  $J$. The proof of this combines $L^2$ orthogonality with $L^6$ decoupling, exploiting the fact that the support $\Gamma_I(R^{-1})$ of $\widehat{\Pc_IF}$ is essentially planar. Note that this result does not use trilinear transversality. In the grand scheme of the proof, this plays the same role as the role played by Corollary \ref{8} in our earlier argument for the parabola.
\begin{pr}
\label{ccc9}	
For each $I$ contributing to $F$ (that is, for half of the intervals $I\in\I_{R^{-1/3}}$) and each $p\ge 2$ we have	
$$A(N_0R^{4-3\alpha})^{1/p}\lesssim_\epsilon N_0^{1/p}\min(N^{-1/2},N_0^{-1/6})R^{(\alpha-\frac13)(1-\frac4p)+\epsilon}(\sum_{J\in \I_{R^{-\alpha}}(I)}\|\Pc_JF\|^p_{L^p(\R^3)})^{1/p}.$$
\end{pr}
In light of \eqref{ccc12}, for each $I$ that contributes to $g$ (the number of such $I$ will never enter our considerations), the above inequality is equivalent to 	
$$\|\Pc_{2I}g\|_{L^p(\R^3)}\lesssim_\epsilon N_0^{1/p}\min(N^{-1/2},N_0^{-1/6})R^{(\alpha-\frac13)(1-\frac4p)+\epsilon}(\sum_{J\in \I_{R^{-\alpha}}(I)}\|\Pc_JF\|^p_{L^p(\R^3)})^{1/p}.$$
The second result is about decoupling into intervals $I$ of canonical scale.

\begin{pr}
\label{ccc22}We have for $p=6+\frac2\alpha$
$$\|(g_1g_2g_3)^{1/3}\|_{L^p(\R^3)}\lesssim_\epsilon A(N_0R^{4-3\alpha})^{1/p} N_0^{-1/p}\max(N^{1/2},N_0^{1/6})R^{\frac{1}{3}(1-\frac3p)+\alpha(\frac3p-\frac12)+\epsilon} .$$
\end{pr}
\bigskip

In light of \eqref{ccc12}, if most $I$ were contributing to $g$, this inequality would be equivalent to
$$\|(g_1g_2g_3)^{1/3}\|_{L^p(\R^3)}\lesssim_\epsilon N_0^{-1/p}\max(N^{1/2},N_0^{1/6})R^{\frac{1}{3}(1-\frac4p)+\alpha(\frac3p-\frac12)+\epsilon}(\sum_{I\in\I_{R^{-1/3}}}\|\Pc_{2I}g\|^p_{L^p(\R^3)})^{1/p}.$$
We cannot prove this stronger inequality in the most general case, essentially because  the upper bound in Proposition \ref{ccc23} is not sensitive to the number of contributing $I$ (it does not get smaller if this number is smaller). The superficially weaker bound in Proposition \ref{ccc22} is compensated by the universal bound from Proposition \ref{ccc9}, which holds for all $I$ contributing to $F$ (including those that do not contribute to $g$). It is worth observing that we can carry on this type of argument precisely  because of the built-in uniformity of the function $F$, which manifests in the fact that $\|\Pc_JF\|^p_{L^p(\R^3)}$ is essentially independent of $J$.
\bigskip

\subsection{Proof of Proposition \ref{ccc9}}
Note  that for each $J\in\I_{R^{-\alpha}}(I)$
\begin{equation}
\label{ccc14}
\|\Pc_{J}F\|_{L^p(\R^3)}\sim R^{3/p},
\end{equation}
so $(\sum_{J\in \I_{R^{-\alpha}}(I)}\|\Pc_JF\|^p_{L^p(\R^3)})^{1/p}$ is essentially independent of $I$. We choose an $I$ that contributes to $g$.
We need to prove two upper bounds. We recast the first one into an $L^2$ inequality whose proof will follow from almost orthogonality.  Using \eqref{ccc12} and \eqref{ccc14}, the first upper bound
$$\|\Pc_{2I}g\|_{L^p(\R^3)}\lesssim_\epsilon N_0^{1/p}N^{-1/2}R^{(\alpha-\frac13)(1-\frac4p)+\epsilon}(\sum_{J\in \I_{R^{-\alpha}}(I)}\|\Pc_JF\|^p_{L^p(\R^3)})^{1/p}$$
is equivalent to
\begin{equation}
\label{ccc20}A\lesssim_\epsilon R^{\alpha-\frac{1}3+\epsilon}N^{-\frac{1}2}.
\end{equation}
 To prove this, we pick a cube $\Omega$ with side length ${R^{2/3}}$ that intersects significantly  some contributing $\Sigma$ (see (S3)). Both families of functions $(F_P)_{P\in\P_{R^{-1/3}}(F)}$ and  $(\Pc_JF)_{J\in\I_{R^{-\alpha}}(I)}$ are almost orthogonal on $\Omega$, hence
$$\|\Pc_{2I}g\|_{L^2(\Omega)}\lesssim \|\Pc_{I}F\|_{L^2(w_\Omega)}\lesssim (\sum_{J\in\I_{R^{-\alpha}}(I)}\|\Pc_JF\|_{L^2(w_\Omega)}^2)^{1/2}.$$
The structure assumption (S3) implies that the decomposition of $\Pc_{2I}g$ has $\sim NR^{\frac23-\alpha}$ planks $P$ that intersect  $\Omega$ significantly, more precisely  $|\Omega\cap P|\sim R^{\frac53}$. Thus
$$\|\Pc_{2I}g\|_{L^2(\Omega)}\gtrsim A(NR^{\frac{7}3-\alpha})^{1/2}.$$
Also, it is rather immediate that
$$(\sum_{J\in\I_{R^{-\alpha}}(I)}\|\Pc_JF\|_{L^2(w_\Omega)}^2)^{1/2}\sim R^{\frac{5}6+\frac\alpha2}.$$
The desired upper bound  \eqref{ccc20} follows by combining the last three inequalities.
\smallskip

Using again \eqref{ccc12} and \eqref{ccc14} alongside earlier reasoning, the second upper bound
$$\|\Pc_{2I}g\|_{L^p(\R^3)}\lesssim_\epsilon N_0^{1/p-1/6}R^{(\alpha-\frac13)(1-\frac4p)+\epsilon}(\sum_{J\in \I_{R^{-\alpha}}(I)}\|\Pc_JF\|^p_{L^p(\R^3)})^{1/p}$$
is similarly seen to be equivalent to the following estimate in $L^6$
\begin{equation}
\label{ccc21}
\|\Pc_{2I}g\|_{L^6(\R^3)}\lesssim_\epsilon R^\epsilon(\sum_{J\in\I_{R^{-\alpha}}(I)}\|\Pc_JF\|_{L^6(\R^3)}^2)^{1/2}.
\end{equation}
To justify \eqref{ccc21}, we start by recalling that the Fourier transform of $\Pc_{I}F$ is supported in the $\frac1R$-neighborhood of the arc $\Gamma_{I}$. This in turn lies inside the vertical parabolic cylinder
$$(\{(\xi,\xi^2):\;\xi\in I\}+O(R^{-1}))\times \R.$$
The intervals $J\in\I_{R^{-\alpha}}(I)$ have length at least $R^{-1/2}$. Planar $L^6$ decoupling (sometimes referred to as cylindrical decoupling) is thus available for $F$ and  gives
$$\|\Pc_{I}F\|_{L^6(\R^3)}\lesssim_\epsilon R^\epsilon(\sum_{J\in\I_{R^{-\alpha}}(I)}\|\Pc_JF\|_{L^6(\R^3)}^2)^{1/2}.$$

We combine (W3) and \eqref{ccc11} to write  $\|\Pc_{2I}g\|_{L^6(\R^3)}\lesssim \|\Pc_{I}F\|_{L^6(\R^3)}$.
The desired upper bound \eqref{ccc21} follows by combining the last two inequalities.
\bigskip

\subsection{Proof of Proposition \ref{ccc22}}
Recall that $\P$ are the planks of $g$. Call $\P_1,\P_2,\P_3$ the planks of $g_1,g_2,g_3$.
We will interpolate Theorem \ref{ccc24} ($p=12$) with the following standard reformulation of the trilinear restriction estimate for curves.
\begin{pr}
\label{ccc39}	
Let $q$ be any cube in $\R^3$ with side length $R^{1/3}$. Then
$$\|(g_1g_2g_3)^{1/3}\|_{L^6(q)}\lesssim (\|\sum_{P\in\P_1}|F_P|^2\|_{L^3(\chi_q)}\|\sum_{P\in\P_2}|F_P|^2\|_{L^3(\chi_q)}\|\sum_{P\in\P_3}|F_P|^2\|_{L^3(\chi_q)})^{1/6}$$
\end{pr}
For $r_1,r_2,r_3\ge 1$, let $\Qc_{r_1,r_2,r_3}$ be the collection of $R^{1/3}$-cubes $q$ in $[0,R]^3$ that intersect $\sim r_1$, $\sim r_2$ and $\sim r_3$ planks from $\P_1$, $\P_2$ and $\P_3$, respectively.
Invoking dyadic considerations and Schwartz-type decay we may write
$$\|(g_1g_2g_3)^{1/3}\|_{L^{p}(\R^3)}\lessapprox\max_{1\le r_1,r_2,r_3\lesssim R^{1/3}}\|(g_1g_2g_3)^{1/3}\|_{L^{p}(\cup_{q\in\Qc_{r_1,r_2,r_3}}q)}.$$
To reduce unnecessary technicalities, we only analyze the diagonal contribution  $r_1=r_2=r_3=r$. We denote $\Qc_{r_1,r_2,r_3}$ by $\Qc_r$.

We use  Proposition \ref{ccc39} to derive a first estimate
\begin{align*}
\label{ccc25}
\|(g_1g_2g_3)^{1/3}\|_{L^6(\cup_{q\in\Qc_r}q)}&\lesssim (\|\sum_{P\in\P_1}|F_P|^2\|_{L^3(\sum \chi_q)}\|\sum_{P\in\P_2}|F_P|^2\|_{L^3(\sum \chi_q)}\|\sum_{P\in\P_3}|F_P|^2\|_{L^3(\sum \chi_q)})^{1/6}\\&\lesssim A
(\|\sum_{P\in\P_1}\chi_P\|_{L^3(\sum \chi_q)}\|\sum_{P\in\P_2}\chi_P\|_{L^3(\sum \chi_q)}\|\sum_{P\in\P_3}\chi_P\|_{L^3(\sum \chi_q)})^{1/6},\;(\text{by W2})\\&\approx A(|\Qc_r|Rr^3)^{1/6}\\&=A(N_0R^{4-3\alpha})^{1/6}(\frac{R^{3\alpha-3}|\Qc_r|r^3}{N_0})^{\frac{1}{6}}.\end{align*}
To simplify technicalities, we have replaced $\chi_P$ with $1_P$ and $\chi_q$ with $1_q$.
We make a similar simplification when we apply Theorem \ref{ccc24} to each of $g_1,g_2,g_3$
\begin{align*}
\|(g_1g_2g_3)^{1/3}\|_{L^{12}(\cup_{q\in\Qc_r}q)}&\lesssim_\epsilon r^{\frac{5}{12}}R^\epsilon(\sum_{I\in\I_{R^{-1/3}}}\|\Pc_{2I}g\|^{12}_{L^{12}(\R^3)})^{1/12}\\&\lesssim A(N_0R^{4-3\alpha})^{1/12}r^{\frac{5}{12}}R^{\frac1{36}+\epsilon}.
\end{align*}
We combine the last two inequalities with H\"older's inequality to write for each $6\le p\le 12$
\begin{align*}
\|(g_1g_2g_3)^{1/3}\|_{L^{p}(\cup_{q\in\Qc_r}q)}&\lesssim_\epsilon R^\epsilon A(N_0R^{4-3\alpha})^{1/p}(r^{\frac{5}{12}}R^{\frac1{36}})^{2-\frac{12}{p}}(\frac{R^{3\alpha-3}|\Qc_r|r^3}{N_0})^{\frac2p-\frac{1}{6}}\\&= R^\epsilon A(N_0R^{4-3\alpha})^{1/p}R^{(\alpha-1)(\frac6p-\frac12)+\frac1{18}-\frac1{3p}}r^{\frac1p+\frac13}N_0^{\frac16-\frac2p}|\Qc_r|^{\frac2p-\frac16}.
\end{align*}
It remains to be shown that for $p=6+\frac2\alpha$ and each $r\lesssim R^{1/3}$
$$R^{(\alpha-1)(\frac6p-\frac12)+\frac1{18}-\frac1{3p}}r^{\frac1p+\frac13}N_0^{\frac16-\frac2p}|\Qc_r|^{\frac2p-\frac16}
\lessapprox
N_0^{-\frac1p}\max(N^{1/2},N_0^{1/6})R^{\frac{1}{3}(1-\frac3p)+\alpha(\frac3p-\frac12)}.$$
This boils down to
$$|\Qc_r|^{\frac{3\alpha-1}{6(3\alpha+1)}}\lessapprox N_0^{-\frac1{6(3\alpha+1)}}\max(N^{1/2},N_0^{1/6})\frac{R^{\frac{36\alpha-4-27\alpha^2}{9(6\alpha+2)}}}{r^{\frac{9\alpha+2}{6(3\alpha+1)}}}$$
or equivalently
$$|\Qc_r|\lessapprox \frac{R^{\frac{2-\alpha}{3\alpha-1}+\frac{10}{3}-3\alpha}}{r^{\frac{9\alpha+2}{3\alpha-1}}} N_0^{-\frac{1}{3\alpha-1}} \max (N^{\frac{1}{2}}, N_0^{\frac{1}{6}})^{\frac{6(3\alpha+1)}{3\alpha-1}}.$$
This inequality will be proved in the next subsection. See Proposition \ref{ccc23}.
\bigskip

\subsection{Plank incidences}

\begin{pr}
\label{ccc23}	
	Suppose that $\mathbb{P}$ satisfies the structural requirements (S2) and (S3) introduced at the  beginning of this section. Let $\P_1,\P_2,\P_3$ be the planks in $\P$ associated with intervals $I$ in $[0,1/6]$, $[1/3,1/2]$ and $[2/3,1]$. Let $\Qc_r(\P_1,\P_2,\P_3)$ denote the collection of trilinear $r$-rich $R^{1/3}$-cubes $q$ with respect to  $\P_1,\P_2,\P_3$. Then for each $1\le r\lesssim R^{1/3}$
	\begin{equation}
	\label{vv1}
	|\Qc_r(\P_1,\P_2,\P_3)|\lessapprox \frac{R^{\frac{2-\alpha}{3\alpha-1}+\frac{10}{3}-3\alpha}}{r^{\frac{9\alpha+2}{3\alpha-1}}} N_0^{-\frac{1}{3\alpha-1}} \max (N^{\frac{1}{2}}, N_0^{\frac{1}{6}})^{\frac{6(3\alpha+1)}{3\alpha-1}}.
	\end{equation}
\end{pr}

\bigskip

\textbf{Figure 4}

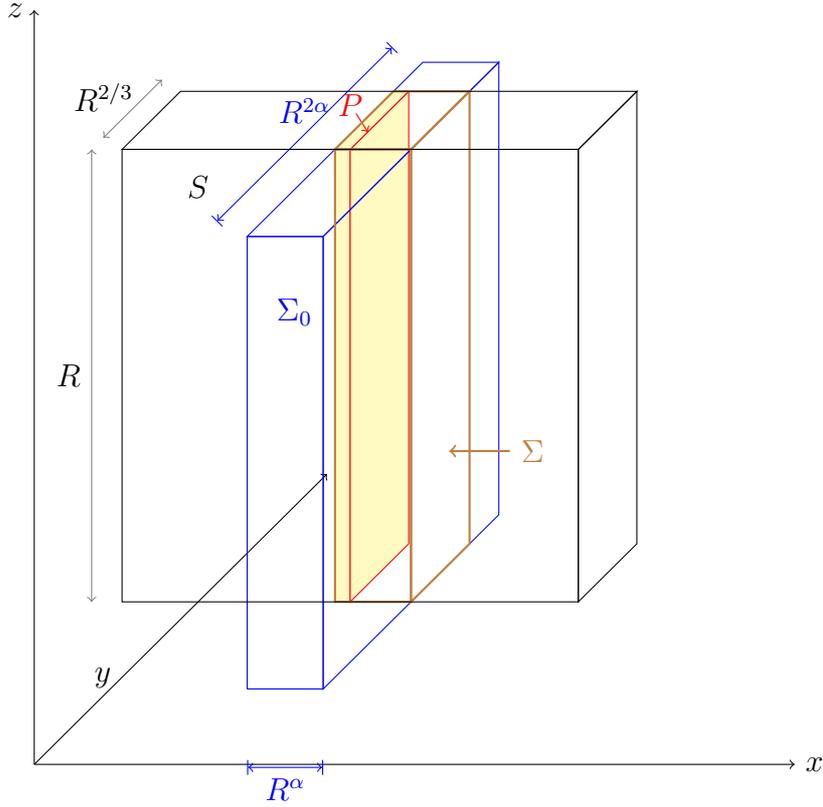
\begin{center}
	
	\begin{tikzpicture}[scale=2]
\pgfmathsetmacro{\cubex}{3}
\pgfmathsetmacro{\cubey}{3}
\pgfmathsetmacro{\cubez}{3}

\coordinate (origin) at (-\cubex, -0.5-\cubey, 0);
\draw[->] (-\cubex, -0.5-\cubey, 0) -- ++ (5, 0, 0);
\node[right] at (5-\cubex, -0.5-\cubey, 0) {$x$};
\draw[->] (origin) -- ++ (0,5, 0);
\node[left] at ($ (origin) + (0, 5, 0)$) {$z$};
\draw[->] (origin)--++ (0, 0, -5);
\node[left] at ($ (origin)! 0.3! (-\cubex, -0.5-\cubey, -5)$) {$y$};

\pgfmathsetmacro{\plankx}{0.1}
\pgfmathsetmacro{\planky}{3}
\pgfmathsetmacro{\plankz}{1}
\coordinate (planko) at (-\cubex/2, 0,-1.5);
\draw[red, fill= yellow!30]  (planko) -- ++(-\plankx, 0,0)-- ++ (0,-\planky, 0) -- ++ (\plankx, 0, 0)-- cycle;
\draw[red, fill=yellow!30]  (planko) -- ++ (0,0, -\plankz) -- ++ (0,-\planky, 0) -- ++ (0,0, \plankz)--cycle;
\draw[red, fill=yellow!30]  (planko) -- ++ (-\plankx, 0,0) -- ++ (0,0,-\plankz) -- ++ (\plankx, 0,0) -- cycle;

\node[red, above] at ($(planko)+(0,0.15,0)$) {$P$};
\draw[red, ->] ($(planko) + (0, 0.2, -0.10)$)--($(planko)+(0, 0, -0.3)$);



\pgfmathsetmacro{\PSplatex}{0.5}
\pgfmathsetmacro{\PSplatey}{3}
\pgfmathsetmacro{\PSplatez}{3}
\coordinate (PSorigin) at (-\cubex/2+0.4,0,0); 
\draw[blue] (PSorigin) -- ++(-\PSplatex, 0,0)-- ++ (0,-\PSplatey, 0) -- ++ (\PSplatex, 0, 0)-- cycle;
\draw[blue]  (PSorigin) -- ++ (0,0, -\PSplatez) -- ++ (0,-\PSplatey, 0) -- ++ (0,0, \PSplatez)--cycle;
\draw[blue]  (PSorigin) -- ++ (-\PSplatex, 0,0) -- ++ (0,0,-\PSplatez) -- ++ (\PSplatex, 0,0) -- cycle;
\node[below, left, blue] at ([yshift=-0.5cm]PSorigin) {$\Sigma_0$}; 



\coordinate (ddd) at (0, -0.52, 0);
\coordinate (SS1) at ($(PSorigin)+(0, -\PSplatey, 0)+(ddd)$);
\coordinate (SS2) at ($(PSorigin)+(-\PSplatex, -\PSplatey, 0)+(ddd)$);
\draw[blue, |<->|] (SS1)--(SS2);
\node[blue, below]  at ($(SS1)! 0.5 ! (SS2)$) {$R^{\alpha}$};

\coordinate (dddd) at (-0.2, 0.1,0);
\coordinate (SSS1) at ($(PSorigin)+ (-\PSplatex, 0, 0) +(dddd)$);
\coordinate (SSS2) at ($(PSorigin)+ (-\PSplatex, 0, -\PSplatez)+(dddd)$);
\draw[blue, |<->|] (SSS1)--(SSS2);
\node[blue, above] at ($(SSS1)! 0.5 ! (SSS2)$) {$R^{2\alpha}$};

\pgfmathsetmacro{\taux}{0.5}
\pgfmathsetmacro{\tauy}{3}
\pgfmathsetmacro{\tauz}{1}
\coordinate (tauo) at (-\cubex/2+0.4, 0, -1.5);
\draw[thick, brown] (tauo) -- ++(-\taux, 0,0)-- ++ (0,-\tauy, 0) -- ++ (\taux, 0, 0)-- cycle;
\draw[thick, brown]  (tauo) -- ++ (0,0, -\tauz) -- ++ (0,-\tauy, 0) -- ++ (0,0, \tauz)--cycle;
\draw[thick, brown]  (tauo) -- ++ (-\taux, 0,0) -- ++ (0,0,-\tauz) -- ++ (\taux, 0,0) -- cycle;
\draw[thick, brown, ->] ([yshift=-2cm, xshift=0.65cm] tauo) -- ++(-0.4, 0, 0);
\node[thick, brown] at ([yshift=-2cm, xshift=0.8cm] tauo) {$\Sigma$};

\pgfmathsetmacro{\BSplatex}{3}
\pgfmathsetmacro{\BSplatey}{3}
\pgfmathsetmacro{\BSplatez}{1}
\coordinate (BSorigin) at (0, 0, -1.5);
\draw (BSorigin) -- ++(-\BSplatex, 0,0)-- ++ (0,-\BSplatey, 0) -- ++ (\BSplatex, 0, 0)-- cycle;
\draw (BSorigin) -- ++ (0,0, -\BSplatez) -- ++ (0,-\BSplatey, 0) -- ++ (0,0, \BSplatez)--cycle;
\draw  (BSorigin) -- ++ (-\BSplatex, 0,0) -- ++ (0,0,-\BSplatez) -- ++ (\BSplatex, 0,0) -- cycle;
\node at ([xshift = -2.5cm, yshift=-0.25cm]BSorigin) {$S$};

\coordinate(ss) at (-0.2, 0, -0.2);
\coordinate (m1) at ($(BSorigin)+(-\BSplatex, 0, 0) + (ss)$);
\coordinate (m2) at ($(BSorigin)+(-\BSplatex, 0, -\BSplatez) + (ss)$);
\draw[gray, <->] (m1)--(m2);
\node[left] at ($(m1)! 0.7 !(m2) $)  {$ R^{2/3}$};

\coordinate (sss) at (-0.2, 0, 0);
\coordinate (mm1) at ($(BSorigin)+(-\BSplatex, 0, 0) + (sss)$);
\coordinate (mm2) at ($(BSorigin)+ (-\BSplatex, -\BSplatey, 0) +(sss)$);
\draw[gray, <->] (mm1)--(mm2);
\node[left] at ($(mm1)!0.5!(mm2)$) {$R$};


%
%
%

\end{tikzpicture}
\end{center}
\bigskip	
	
\begin{proof}		
	
For each contributing $I\in\I_{R^{-1/3}}$ we tile $[0,R]^3$ with $(R,R^{2/3},R)$-plates $S$ with normal vector  $\n(I)$. Note that each $S$ can be partitioned into planks $\Sigma$. Due to our assumption (S3) on $\Sigma$ and to $R^\alpha$-periodicity in the $x$ direction, we can split the plates $S$ into two categories. Those that contain $\sim NR^{1-\alpha}$ planks $P\in\P_I$ will be referred to as {\em heavy} and will be denoted by $\S_{heavy,I}$, while those that contain no  $P\in\P_I$ will be called {\em light}, and will play no role in the forthcoming argument. Let $\S_{heavy}$ be the union of all $\S_{heavy,I}$. It is immediate that
  $$|\S_{heavy}|\lesssim R^{\frac43-2\alpha} \frac{N_0}{N}.$$
This upper bound is only sharp if most $I$ are contributing, but it is always good enough for us.

	Let us partition $[0,R]^3$ into $R^{2/3}$-cubes  $Q$. Each small  $R^{1/3}$-cube $q$ in $\Qc_r(\P_1,\P_2,\P_3)$ lies inside such a large cube $Q$. Since each $q$ is trilinear $r$-rich  and since each plank $P$ lies inside some plate $S\in\S_{heavy}$, all relevant large cubes $Q$ can be assumed to be trilinear $M$-rich  with respect to the family $\S_{heavy}$, for some dyadic number $M\ge r$. There are $\lessapprox 1$ such values of $M$.
	
	For fixed $M\ge r$, we simply denote by $\Qc_{M}(\S_{heavy})$ the collection of these cubes $Q$. In preparation for an application of Corollary \ref{sdgesyhfioewyrt7yeidcyre7}, we define the separation parameter $W_M$ to be
	$W_M=R^{1-2\alpha}$ if $M \gtrapprox R^{1-2\alpha}\frac{N_0}{N}$ and $W_M=1$ otherwise.
	We estimate $\Qc_{M}(\S_{heavy})$ using purely planar considerations, since each plate $S$ is parallel to the $x$-axis. More precisely, we consider the projections of each $S\in\S_{heavy}$ onto the $yz$-plane -these are planar tubes- and find upper bounds for the $M$-rich squares associated with these tubes. If we multiply this upper bound with $R^{1/3}$ -the number of mutually parallel  cubes $Q$ in the $x$ direction-, we find an upper bound for $\Qc_{M}(\S_{heavy})$.
	When $M \gtrapprox R^{1-2\alpha}\frac{N_0}{N}$ we apply Corollary \ref{sdgesyhfioewyrt7yeidcyre7} with $W=W_M$. As a side remark, note that our tubes satisfy a slightly  stronger assumption, they are $R^{2\alpha}$- periodic. This additional structural assumption will not be needed, but it is interesting to ask whether Theorem \ref{12} admits an easier proof in this case.
	Otherwise we content ourselves with using bilinear Kakeya, Proposition \ref{vvv7}. In both cases the estimate reads
\begin{equation}
\label{dh4wiriowqueduireterudwifryu}	
	|\Qc_{M}(\S_{heavy})|\lessapprox \frac1{W_M}R^{1/3} (\frac{R^{\frac43-2\alpha} N_0}{N M})^2.
	\end{equation}
	
	Fix $Q\in \Qc_{M}(\S_{heavy})$. To estimate the number of cubes $q$ in $\Qc_{r}(\P_1,\P_2,\P_3)$ that lie inside $Q$ we apply  Theorem \ref{induction} with $\delta=R^{-1/3}$, and rescale. Indeed, note that each $P\cap Q$ is a rescaled Vinogradov plate.

	  Putting these bounds together we find that
	\begin{align*}
	|\Qc_r(\P_1,\P_2,\P_3)|&\lessapprox  |\Qc_{M}(\S_{heavy})|  (\frac{R^{1/3}NM}{r^2})^{\frac{4-6\alpha}{3\alpha-1}} (\frac{MN}{r})^3  R^{\frac{2}{3}-\alpha}\\
	&\lessapprox  \frac1{W_M}R^{1/3} (\frac{R^{\frac43-2\alpha} N_0}{N M})^2(\frac{R^{1/3}NM}{r^2})^{\frac{4-6\alpha}{3\alpha-1}} (\frac{MN}{r})^3  R^{\frac{2}{3}-\alpha}
	\\&=\frac1{W_M}N_0^2 N^{\frac{3-3\alpha}{3\alpha-1}}R^{\frac{11}{3}-5\alpha+\frac{\frac43-2\alpha}{3\alpha-1}}\frac{M^{\frac{3-3\alpha}{3\alpha-1}}}{r^{\frac{5-3\alpha}{3\alpha-1}}}.
\end{align*}

Now we discuss how this inequality  implies \eqref{vv1}.
\\

If $W_M=R^{1-2\alpha}$, a simple computation reveals that $$R^{\frac{8}{3}-3\alpha+\frac{\frac43-2\alpha}{3\alpha-1}}\frac{M^\frac{3-3\alpha}{3\alpha-1}}{r^{\frac{5-3\alpha}{3\alpha-1}}}\leq \frac{R^{\frac{2-\alpha}{3\alpha-1}+\frac{10}{3}-3\alpha}}{r^{\frac{9\alpha+2}{3\alpha-1}}}$$ because $r, M \leq R^{1/3}$  and $\alpha>\frac13$. It remains to check that
 $$N_0^2 N^{\frac{3-3\alpha}{3\alpha-1}}\leq N_0^{-\frac{1}{3\alpha-1}} \max (N^{\frac{1}{2}}, N_0^{\frac{1}{6}})^{\frac{6(3\alpha+1)}{3\alpha-1}}.$$
 Raising to the power $3\alpha-1$ we write, using a  geometric average with exponents $\beta=\frac{2-3\alpha}{3\alpha+1}$ and $1-\beta=\frac{6\alpha-1}{3\alpha+1}$ in the first step, and the fact  that $\alpha\le \frac12$ in the second step
 \begin{align*}\max(N^{9\alpha+3},N_0^{3\alpha+1})&\ge N^{\frac{(9\alpha+3)(2-3\alpha)}{3\alpha+1}}N_0^{6\alpha-1}\\&\ge N^{3-3\alpha}N_0^{6\alpha-1}.
 \end{align*}
Thus  \eqref{vv1} is verified in this case.
\\

We next analyze the remaining case $W_M=1$.




Using that $\max (N^{\frac{1}{2}}, N_0^{\frac{1}{6}})\ge N_0^{\frac16}$, it suffices to prove that
$$N_0^2 N^{\frac{3-3\alpha}{3\alpha-1}}R^{\frac{11}{3}-5\alpha+\frac{\frac43-2\alpha}{3\alpha-1}}\frac{M^{\frac{3-3\alpha}{3\alpha-1}}}{r^{\frac{5-3\alpha}{3\alpha-1}}}\lessapprox\frac{R^{\frac{2-\alpha}{3\alpha-1}+\frac{10}{3}-3\alpha}}{r^{\frac{9\alpha+2}{3\alpha-1}}} N_0^{-\frac{1}{3\alpha-1}} N_0^{\frac{3\alpha+1}{3\alpha-1}},$$
which, when rearranged,  reads
$$N_0^{\frac{3\alpha-2}{3\alpha-1}}N^{\frac{3-3\alpha}{3\alpha-1}}M^{\frac{3-3\alpha}{3\alpha-1}}r^{\frac{12\alpha-3}{3\alpha-1}}\lessapprox R^{2\alpha-\frac13+\frac{\alpha+\frac23}{3\alpha-1}}.$$
If we plug in the bound $r, M \lessapprox R^{1-2\alpha} \frac{N_0}{N}$, we reduce things to showing that
$$(\frac{N_0}{N})^{\frac{12\alpha-2}{3\alpha-1}}N^{\frac{1}{3\alpha-1}}\lesssim R^{2\alpha-\frac13+\frac{\alpha+\frac23-9\alpha(1-2\alpha)}{3\alpha-1}}.$$
This follows via a simple computation that uses  the bounds $N\lesssim R^{\alpha-\frac13}$ and $\frac{N_0}{N}\lesssim R^{2\alpha-\frac23}$.

\end{proof}

\begin{re}
\label{ekjfewydeyfuopueuiytu38u}
Let us now comment on the periodicity assumption in (S2) from the structure result for $\P^{(i)}$ introduced earlier in this section. We start with the heuristics on why this assumption is genuine, and then give hints about how a rigorous argument can be put into place. In a ``perfect world", the wave packets $F_P$ would be perfectly localized inside the planks $P$, with $|F_P|=A_P1_P$. Let us  assume for a moment that we are in this ideal setup (incidentally, this setup exists, if the Fourier transform is replaced with the Walsh-Fourier transform). Recall that
$$\Pc_IF(x,y,z)=\widehat{\eta_R}(x,y,z)\sum_{\frac{j}{R^\alpha}\in I} a_je(x\frac{j}{R^\alpha}+y\frac{j^2}{R^{2\alpha}}+z\frac{j^3}{R^{3\alpha}}).$$
Then $\Pc_IF(x,y,z)=F_P(x,y,z)$, where $P$ is the plank containing $(x,y,z)$. Note that, when restricting attention to $[0,R]^3$, we have $|\Pc_IF(x,y,z)|\sim |\Pc_IF(x+R^{\alpha},y+R^{2\alpha},z)|$. This is the key point where we  use the periodicity of our exponential sums. It implies that $A_{P_1}\sim A_{P_2}$, where $P_1$ and $P_2$ are the planks containing $(x,y,z)$ and $(x+R^{\alpha},y+R^{2\alpha},z)$, respectively. Because of this, the pigeonholing leading to the structure assumption (S1) for $\P^{(i)}$  will place the wave packets $F_{P_1}$ and $F_{P_2}$ in the same family. This is the periodicity we referred to in (S2). It is worth re-emphasizing  that throughout all arguments in this section, the actual planks $P$ are defined somewhat loosely, indistinguishable from their neighbors. Thus, the fact that $P_2$ may not be an honest translation of $P_1$ by $(R^\alpha,R^{2\alpha},0)$ is not a problem for our argument.

We now sketch the more rigorous argument on why periodicity may be enforced. We first observe that $\Pc_IF(x,y,z)$ receives contribution from not just the plank $P(x,y,z)$ containing the point, but also from nearby planks. Fix some $\epsilon>0$ and write $R^\epsilon P$ for the $R^\epsilon$ dilate of $P$ around its center. Recall $F_P$ has weight $A_P=\|F_P\|_\infty$. We assign the new weight $A_{P,new}$ to $F_P$ defined as follows
$$A_{P,new}=\max_{P'\subset R^\epsilon P}A_{P'}.$$
Within this more rigorous framework, the pigeonholing in (S1) is with respect to these new weights, and enforces $A_{P,new}\sim A$ for all $P$ within each family. There are a few things that we need to check in order to make sure that this change does not alter the structure of $\P^{(i)}$ and the proof of Theorem \ref{ccc7}. Let us start with (S2). Let $P_1$ and $P_2$ be the planks containing $(x,y,z)$ and $(x+R^{\alpha},y+R^{2\alpha},z)$, respectively. We need to prove that $\frac{A_{P_1,new}}{A_{P_2,new}}\in [R^{-O(\epsilon)},R^{O(\epsilon)}]$. This is a bit weaker than ${A_{P_1,new}}\sim {A_{P_2,new}}$ and will lead to $R^\epsilon$ losses, but these are harmless, as $\epsilon$ can be chosen to be arbitrarily small. The verification (left to the reader) involves $L^2$ orthogonality and the fact that for each $P$ and $P'=P+(R^\alpha,R^{2\alpha},0)$ we have
$$\|\Pc_I F\|_{L^2(R^\epsilon P)}\sim \|\Pc_I F\|_{L^2(R^\epsilon P')}.$$
To summarize, the structure of $\P^{(i)}$ is preserved. 

Recall that the proof of Theorem \ref{ccc7} was reduced to verifying Proposition \ref{ccc9} and Proposition \ref{ccc22}. We claim that both results continue to hold true, with the slight modification for the weights. First, the reader will note that the argument for  Proposition \ref{ccc22} only used the upper bound $\|F_P\|_\infty\lesssim A$. Since we have increased the weights ($A_{P,new}\ge A_P)$, Proposition \ref{ccc22} remains true in the new context. 

To make sure that Proposition \ref{ccc9}  continues to hold, we need to verify that  \eqref{ccc12}  (or rather its slight weakening allowing for $R^\epsilon$ losses) remains true. The upper bound is clear since we increased the weights. For the lower bound, note first that  among the $\sim N_0 R^{2-3\alpha}$ wave packets $F_P$ contributing to $\Pc_{2I}g$, all having $A_{P,new}\sim A$, there are at least  $\sim N_0 R^{2-3\alpha-O(\epsilon)}$ of them satisfying $A_P\sim A$. Call this family $\P^*$. Due to (W2)-(W4) in Theorem \ref{WPdeco} we may write
$$\|\Pc_{2I}g\|_{L^p(\R^3)}\gtrsim \|\sum_{P\in\P^*}F_P\|_{L^p(\R^3)}\sim (\sum_{P\in\P^*}\|F_P\|^p_{L^p(\R^3)})^{1/p}\gtrsim A(N_0R^{4-3\alpha-O(\epsilon)})^{1/p}.$$

\end{re}

\smallskip

\section{Proof of Theorem \ref{n=3partialrange} in the range $1<\beta\le \frac32$}
Throughout this section, fix $\frac12<\alpha\le \frac23$. Let $$T_R=[0,R]\times [0,R^{2\alpha}]\times [0,R].$$
Let $\eta:[-\frac1{10},\frac1{10}]^3\to\R$ be a Schwartz function and let
$$\eta_{T_R}(\xi)=R^{2+2\alpha}\eta(R\xi_1,R^{2\alpha}\xi_2,R\xi_3).$$

Our main result in this section is the following small cap decoupling for special functions with spectrum near the moment curve. This complements Theorem \ref{ccc7}, it covers the case of decoupling  into  even smaller arcs.

\begin{te}
	\label{dd1}	
	Let $a_j\in\C$ with unit modulus. Define
	$$
	\widehat{F}(\xi)=\sum_{j=1}^{R^{\alpha}}a_j\eta_{T_R}(\xi_1-\frac{j}{R^\alpha},\xi_2-\frac{j^2}{R^{2\alpha}},\xi_3-\frac{j^3}{R^{3\alpha}}).
	$$
	Then for $p=6+\frac{2}{\alpha}$ we have
	$$
	\|(F_1F_2F_3)^{1/3}\|_{L^p(T_R)}\lesssim_\epsilon R^{\alpha(\frac12-\frac1p)+\epsilon}(\sum_{J\in \I_{R^{-\alpha}}}\|\Pc_JF\|^p_{L^p(\R^3)})^{1/p}.
	$$
\end{te}
The proof of the following corollary is essentially identical to the one of Corollary \ref{ccc26}.  
\begin{co}
	If $p=6+\frac2\alpha$	and assume $a_j\in\C$ have unit modulus. Then for each translate $\tilde{T_R}$ of $T_R$ we have
	\begin{equation}
	\label{ddd2}
	\|\sum_{j=1}^{R^\alpha}a_je(x\frac{j}{R^\alpha}+y\frac{j^2}{R^{2\alpha}}+z\frac{j^3}{R^{3\alpha}})\|_{L^p_{\sharp}(\tilde{T_R})}\lesssim_{\epsilon}R^{\frac{\alpha}{2}+\epsilon}.
	\end{equation}
\end{co}
We mention a few key points on how to adapt the trilinear-to-linear reduction from the previous section to this case. The linear and trilinear decoupling constants will be with respect to averages over arbitrary translates $\tilde{D_R}$ of $D_R=[0,R^{2\alpha}]\times[0,R^{2\alpha}]\times [0,R]$. More precisely, we let  $\Dec(R,p,\alpha)$ be the smallest constant such that the inequality
$$\|\sum_{j=1}^{R^\alpha}a_je(x\frac{j}{R^\alpha}+y\frac{j^2}{R^{2\alpha}}+z\frac{j^3}{R^{3\alpha}})\|_{L^p_{\sharp}(\tilde{D_R})}\le \Dec(R,p,\alpha) R^{\frac\alpha2}$$
holds true for each such $\tilde{D_R}$  and each $a_j\in \C$ with $|a_j|=1$. We note that by making the first component of $D_R$ larger than the first component of $T_R$, the new averages do not change, due to periodicity. We make this superficial change just for convenience, as explained below.

   The proof of the analogue of Proposition \ref{p1} for this new case follows the same lines.  The constraint \eqref{fyryufioefiuyrdodpesivuuifiep[oyrtuio}  remains the same
$$4<\frac{p}{2}+\frac1\alpha.$$
This is clearly satisfied in our context. The image of $\tilde{D_R}$ under the map
$$(x,y,z)\mapsto (\frac{x+2cy+3c^2z}{K},\frac{y+3cz}{K^2},\frac{z}{K^3})$$
lies inside a rectangular box with dimensions $\sim (\frac{R^{2\alpha}}{K},\frac{R^{2\alpha}}{K^2},\frac{R}{K^3})$. Thus, it  can be covered with $\sim K$  translates of $D_{R'}$, where as before $R'=RK^{-\frac1\alpha}$. It is important that the needed number of such translates is independent of $R$. This is possible due to our choice of the slightly larger domain $D_R$ (compared to $T_R$). We leave the details to the reader.

\medskip

 Let us now justify our choice of $T_R$ in Theorem \ref{dd1}. Since
$$|\sum_{j=1}^{R^\alpha}e(x\frac{j}{R^\alpha}+y\frac{j^2}{R^{2\alpha}}+z\frac{j^3}{R^{3\alpha}})|\gtrsim R^{\alpha}$$
for $x\in\bigcup_{k=1}^{R^{1-\alpha}}[kR^\alpha,kR^\alpha+c]$, $|y|,|z|\le c$ (for $c$ a small enough constant),
we see that \eqref{ddd2} is false if $T_R$ is replaced with any smaller box $[0,R]\times [0,R^\beta]\times [0,R]$, $\beta<2\alpha$.

The use of the larger domain $T_R$ in Theorem \ref{dd1} does not alter its small-cap-decoupling nature. The important thing is that we did not enlarge the domain of integration for the variable $z$. This way, the result is strong enough to capture the desired application. Indeed, \eqref{ddd2} immediately implies Theorem \ref{n=3partialrange} in the range $1<\beta\le \frac32$. It suffices to use $\beta=3-\frac1\alpha$, $N=R^{\alpha}$, rescaling and periodicity.
\smallskip

The proof of Theorem \ref{dd1} will be done in several stages.
We use the wave packet decomposition \eqref{ddd3} for $F$ and assume that the planks $P\in\P_{R^{-1/3}}(F)$ are inside $T_R$ and  are  $R^\alpha$-periodic in the $x$ direction. Note however that there is no periodicity in the $y$-direction. We will replace the integration domain $T_R$ with $\R^3$.

We split $\P_{R^{-1/3}}(F)$ into collections $\P^{(i)}$ with the following properties.
Note that there is a new parameter $X$, which makes  (S2) below slightly more substantial than its earlier counterpart for the case $\alpha\in (\frac13,\frac12]$.
\\
\\
\textbf{Structure of $\P_i$:}
\\

(S1) \;$\|F_P\|_\infty\sim A$ for all $P\in\P^{(i)}$, for some dyadic parameter $A$
\\

(S2) \;  We cover $T_R$ with $R^{2\alpha-1}$ cubes $Q$ with side length $R$. For each $I\in\I_{R^{-1/3}}$ and $Q$ we denote by  $\P^{(i)}_{I,Q}$ those planks in $\P^{(i)}\cap \P_{I}(F)$ that lie inside $Q$.

We assume that  for some dyadic integers $N_0$ and $X$ the following holds: for each $I$ and each $Q$ we either have $\P^{(i)}_{I,Q}=\emptyset$ or
\begin{equation}
\label{ddd5}
|\P^{(i)}_{I,Q}|\sim N_0R^{1-\alpha}.
\end{equation}
Moreover, for each $I$ the number of those $Q$ satisfying \eqref{ddd5} is either $\sim X$ or $0$. We call ``heavy" those $I$ in the first category. The number of  heavy intervals $I$, and thus the total number of planks in $\P^{(i)}$ will not enter our considerations.

If \eqref{ddd5} holds, we will refer to $I$ as ``contributing" to $Q$. The collection of those $\sim X$ cubes $Q$ to which a given $I$ contributes may vary with $I$. Similarly, the number of those $I$ contributing to a given $Q$ will be a function of $Q$ that will not concern us.

Let us understand better the structure of $\P^{(i)}_{I,Q}$ in case when $I$ contributes to $Q$. We tile $Q$ with $(R^{\alpha},R,R)$-plates $\Sigma_0$ with axes parallel to those of a typical $P\in \P_I(F)$. Each $P\in\P^{(i)}_{I,Q}$ is contained in some unique $\Sigma_0$. Note that due to $R^\alpha$-periodicity in the $x$ direction,  all $\Sigma_0$ will contain the same number of planks $P\in\P^{(i)}_I$. This number must be $\sim N_0$, due to \eqref{ddd5}.
\\

(S3) \; For each   $I$ contributing to  $Q$, we tile $Q$  with $(R^{\alpha},R^{2/3},R)$-planks $\Sigma$ with axes parallel to those of a typical $P\in \P_I(F)$.
Each $\Sigma$ is contained in some unique $\Sigma_0$. We will assume that there are either $\sim N$ or zero planks  $P\in\P^{(i)}_{I,Q}$ inside each such $\Sigma$,  for some dyadic number $1\le N\le N_0$ independent of $I$. In the first case, we will refer to $\Sigma$ as ``contributing".  Note that $$N\lesssim R^{\alpha-\frac13}\;\text{ and }\; \frac{N_0}{N}\lesssim R^{\frac13}.$$
In summary, for each $I$ contributing to $Q$, each $\Sigma_0\subset Q$ contains $\sim \frac{N_0}{N}$ contributing planks $\Sigma$.

\bigskip

Let us fix an arbitrary $i$. To ease notation, we will denote  $\sum_{P\in\P^{(i)}}F_P$ by $g$, $\P^{(i)}$ by $\P$ and $\P^{(i)}_{I,Q}$ by $\P_{I,Q}$. We have as before for $p\ge 2$
$$
\|\Pc_{2I}g\|_{L^p(\R^3)}\sim \begin{cases}A(XN_0R^{3-\alpha})^{1/p}, \text{ if }  I \text{ is heavy}\\0, \text{ otherwise}\end{cases}.
$$
To prove Theorem \ref{dd1} it will suffice to show that for $p=6+\frac2\alpha$
$$
\|(g_1g_2g_3)^{1/3}\|_{L^p(\R^3)}\lesssim_\epsilon R^{\alpha(\frac12-\frac1p)+\epsilon}(\sum_{J\in \I_{R^{-\alpha}}}\|\Pc_JF\|^p_{L^p(\R^3)})^{1/p}.
$$
This will immediately follow from combining two results, similar to those from the previous section.
\begin{pr}
	\label{ddd9}
	Let $q=\frac{6\alpha}{3\alpha-1}$.	
	For each $I$ contributing to $F$ (that is, for half of the intervals $I\in\I_{R^{-1/3}}$) and each $p\ge 2$ we have	
	$$A(XN_0R^{3-\alpha})^{1/p}\lesssim_\epsilon$$$$ \min( N^{-1/2}(XN_0)^{1/p},(XN_0)^{\frac1p-\frac16}, N_0^{-1/q}(XN_0)^{1/p})R^{(\alpha-\frac13)(1-\frac4p)+\epsilon}(\sum_{J\in \I_{R^{-\alpha}}(I)}\|\Pc_JF\|^p_{L^p(\R^3)})^{1/p}.$$
\end{pr}

\begin{pr}
	\label{ddd15}We have for $p=6+\frac2\alpha$
	$$\|(g_1g_2g_3)^{1/3}\|_{L^p(\R^3)}\lesssim_\epsilon$$$$ A(XN_0R^{3-\alpha})^{1/p}\max( N^{1/2}(XN_0)^{-1/p},(XN_0)^{\frac16-\frac1p}, N_0^{1/q}(XN_0)^{-1/p})R^{\alpha(\frac3p-\frac12)+\frac13(1-\frac3p)+\epsilon}.$$
\end{pr}
\bigskip

\subsection{Proof of Proposition \ref{ddd9}}
Note  that for each $J\in\I_{R^{-\alpha}}(I)$
	\begin{equation}
	\label{ddd10}
	\|\Pc_{J}F\|_{L^p(\R^3)}\sim R^{\frac{2+2\alpha}{p}}.
	\end{equation}
We choose an $I$ that contributes to $g$ (also known as ``heavy") and prove the proposition with the left hand side replaced with
$\|\Pc_{2I}g\|_{L^p(\R^3)}$.
	
There are three upper bounds we need to prove, with the first two being essentially identical to the ones in Proposition \ref{ccc9}.

The first upper bound
$$A(XN_0R^{3-\alpha})^{1/p}\lesssim (XN_0)^{1/p}N^{-1/2}R^{(\alpha-\frac13)(1-\frac4p)}(\sum_{J\in \I_{R^{-\alpha}}(I)}\|\Pc_JF\|^p_{L^p(\R^3)})^{1/p}$$
is equivalent to
\begin{equation}
\label{ddd12}A\lesssim R^{\alpha-\frac{1}3}N^{-\frac{1}2}.
\end{equation}
To prove this, we pick a cube $\Omega$ with side length ${R^{2/3}}$ that intersects significantly  some contributing $\Sigma$ (see (S3)). Almost orthogonality implies
$$\|\Pc_{2I}g\|_{L^2(\Omega)}\lesssim \|\Pc_{I}F\|_{L^2(w_\Omega)}\lesssim (\sum_{J\in\I_{R^{-\alpha}}(I)}\|\Pc_JF\|_{L^2(w_\Omega)}^2)^{1/2}.$$
The structural assumption (S3) implies that the decomposition of $\Pc_{2I}g$ has $\sim NR^{\frac23-\alpha}$ planks $P$ that  intersect  $\Omega$ significantly, that is $|\Omega\cap P|\sim R^{\frac53}$. Thus
$$\|\Pc_{2I}g\|_{L^2(\Omega)}\gtrsim A(NR^{\frac{7}3-\alpha})^{1/2}.$$
Also, it is rather immediate that
$$(\sum_{J\in\I_{R^{-\alpha}}(I)}\|\Pc_JF\|_{L^2(w_\Omega)}^2)^{1/2}\sim R^{\frac{5}6+\frac\alpha2}.$$
The desired upper bound  \eqref{ddd12} follows by combining the last three inequalities.
\smallskip

The second upper bound
$$A(XN_0R^{3-\alpha})^{1/p}\lesssim_\epsilon (XN_0)^{\frac1p-\frac16}R^{(\alpha-\frac13)(1-\frac4p)+\epsilon}(\sum_{J\in \I_{R^{-\alpha}}(I)}\|\Pc_JF\|^p_{L^p(\R^3)})^{1/p}$$
is similarly seen to be equivalent to the following estimate in $L^6$
$$
\|\Pc_{2I}g\|_{L^6(\R^3)}\lesssim_\epsilon R^\epsilon(\sum_{J\in\I_{R^{-\alpha}}(I)}\|\Pc_JF\|_{L^6(\R^3)}^2)^{1/2}.
$$
This will follow (cf. (W3)) once we prove
$$
\|\Pc_{I}F\|_{L^6(\R^3)}\lesssim_\epsilon R^\epsilon(\sum_{J\in\I_{R^{-\alpha}}(I)}\|\Pc_JF\|_{L^6(\R^3)}^2)^{1/2}.
$$
If
$I=[\frac{j_0}{R^\alpha},\frac{j_0}{R^\alpha}+R^{-1/3}]$,
this boils down to the estimate
$$\|\sum_{j=j_0}^{j_0+R^{\alpha-\frac13}}e(x\frac{j}{R^\alpha}+y\frac{j^2}{R^{2\alpha}}+z\frac{j^3}{R^{3\alpha}})\|_{L^6_\sharp(T_R)}\lesssim_\epsilon R^{\frac{\alpha-\frac13}{2}+\epsilon}.$$
Changing variables, using periodicity in the first variable  and letting $a_j=e(z\frac{j^3}{R^{3\alpha}})$, this follows from
the uniform estimate over $z$ (consequence of Theorem \ref{t8})
$$\|\sum_{j=j_0}^{j_0+R^{\alpha-\frac13}}a_je(xj+y{j^2})\|_{L^6([0,1]^2)}\lesssim_\epsilon R^{\frac{\alpha-\frac13}{2}+\epsilon}.$$

Let us now pick a cube $Q$ with side length $R$ to which $I$ contributes.  The third upper bound
$$A(XN_0R^{3-\alpha})^{1/p}\lesssim_\epsilon (XN_0)^{\frac1p}N_0^{-\frac1q}R^{(\alpha-\frac13)(1-\frac4p)+\epsilon}(\sum_{J\in \I_{R^{-\alpha}}(I)}\|\Pc_JF\|^p_{L^p(\R^3)})^{1/p}$$
is equivalent with the $L^q$ inequality
$$
\|\Pc_{2I}g\|_{L^q(Q)}\lesssim_\epsilon R^\epsilon(\sum_{J\in\I_{R^{-\alpha}}(I)}\|\Pc_JF\|_{L^q(w_Q)}^2)^{1/2}.
$$
We will in fact prove the following superficially stronger inequality (cf. (W3))
$$
\|\Pc_{I}F\|_{L^q(w_Q)}\lesssim_\epsilon R^\epsilon(\sum_{J\in\I_{R^{-\alpha}}(I)}\|\Pc_JF\|_{L^q(w_Q)}^2)^{1/2}.
$$
Recalling the definition of $F$, if $I=[\frac{j_0}{R^\alpha},\frac{j_0}{R^\alpha}+R^{-1/3}]$ this is equivalent with
$$\|\sum_{j=j_0}^{j_0+R^{\alpha-\frac13}}e(x\frac{j}{R^\alpha}+y\frac{j^2}{R^{2\alpha}}+z\frac{j^3}{R^{3\alpha}})\|_{L^q_\sharp(Q)}\lesssim_\epsilon R^{\frac{\alpha-\frac13}{2}+\epsilon}.$$
When we make the change of variables $x=R^{1/3}x'$, $y=R^{2/3}y'$, $z=Rz'$, the cube $Q$ becomes an $(R^{2/3}, R^{1/3},1)$-rectangular box $B$ with the third side equal to some interval $H$.  We need to prove
$$\|\sum_{j=j_0}^{j_0+R^{\alpha-\frac13}}e(x'\frac{j}{R^{\alpha-\frac13}}+y'\frac{j^2}{R^{2(\alpha-\frac13)}}+z'\frac{j^3}{R^{3(\alpha-\frac13)}})\|_{L^q_\sharp(B)}\lesssim_\epsilon R^{\frac{\alpha-\frac13}{2}+\epsilon}.$$
Let us cover $B$ with boxes $B_{R_0}\times H$, with each $B_{R_0}$ a square with side length $R_0=R^{1/3}$. It further suffices to prove that for each $z'\in H$
$$\|\sum_{j=j_0}^{j_0+R^{\alpha-\frac13}}a_{j}e(x'\frac{j}{R^{\alpha-\frac13}}+y'\frac{j^2}{R^{2(\alpha-\frac13)}})\|_{L^q_\sharp(B_{R_0})}\lesssim_\epsilon R^{\frac{\alpha-\frac13}{2}+\epsilon},$$
where $a_j=e(z'\frac{j^3}{R^{3(\alpha-\frac13)}})$. This however is a consequence of Corollary \ref{t7}, with $R$ replaced by $R_0$ and  $\alpha$ replaced with $3\alpha-1$. Indeed, note that $2+\frac{2}{3\alpha-1}=\frac{6\alpha}{3\alpha-1}=q.$
\bigskip

\subsection{Proof of Theorem \ref{ddd15}}
Recall that $\P$ are the planks of $g$. Call $\P_1,\P_2,\P_3$ the planks of $g_1,g_2,g_3$.	
For $r\ge 1$, let $\Qc_r$ be the collection of  $R^{1/3}$-cubes  in $T_R$ that intersect  $\sim r$ planks from each of the families $\P_1,\P_2,\P_3$.

We use Proposition \ref{ccc39} as before to derive the first estimate
$$
\|(g_1g_2g_3)^{1/3}\|_{L^6(\cup_{q\in\Qc_r}q)}\lessapprox A(|\Qc_r|Rr^3)^{1/6}=A(XN_0R^{3-\alpha})^{1/6}(\frac{R^{\alpha-2}|\Qc_r|r^3}{XN_0})^{\frac{1}{6}}.
$$
We apply Theorem \ref{ccc24} to each of $g_1,g_2,g_3$
\begin{align*}
\|(g_1g_2g_3)^{1/3}\|_{L^{12}(\cup_{q\in\Qc_r}q)}&\lesssim_\epsilon r^{\frac{5}{12}}R^\epsilon(\sum_{I\in\I_{R^{-1/3}}}\|\Pc_{2I}g\|^{12}_{L^{12}(\R^3)})^{1/12}\\&\lesssim A(XN_0R^{3-\alpha})^{1/12}r^{\frac{5}{12}}R^{\frac1{36}+\epsilon}.
\end{align*}
We combine the last two inequalities with H\"older's inequality to write for each $6\le p\le 12$
\begin{align*}
\|(g_1g_2g_3)^{1/3}\|_{L^{p}(\cup_{q\in\Qc_r}q)}&\lesssim_\epsilon R^\epsilon A(XN_0R^{3-\alpha})^{1/p}(r^{\frac{5}{12}}R^{\frac1{36}})^{2-\frac{12}{p}}(\frac{R^{\alpha-2}|\Qc_r|r^3}{XN_0})^{\frac2p-\frac{1}{6}}\\&= R^\epsilon A(XN_0R^{3-\alpha})^{1/p}R^{(\alpha-1)(\frac2p-\frac16)+\frac2{9}-\frac7{3p}}r^{\frac1p+\frac13}(XN_0)^{\frac16-\frac2p}|\Qc_r|^{\frac2p-\frac16}.
\end{align*}
It remains to prove that
$$
R^{(\alpha-1)(\frac2p-\frac16)+\frac2{9}-\frac7{3p}}r^{\frac1p+\frac13}(XN_0)^{\frac16-\frac2p}|\Qc_r|^{\frac2p-\frac16}\lessapprox $$$$\max( N^{1/2}(XN_0)^{-1/p},(XN_0)^{\frac16-\frac1p}, N_0^{1/q}(XN_0)^{-1/p})R^{\alpha(\frac3p-\frac12)+\frac13(1-\frac3p)},$$
which after rearranging the terms becomes
$$|\Qc_r|^{\frac2p-\frac16}\lessapprox\max( N^{1/2}(XN_0)^{\frac1p-\frac16},(XN_0)^{\frac1p}, N_0^{1/q}(XN_0)^{\frac1p-\frac16})\frac{R^{\alpha(\frac1p-\frac13)-\frac1{18}+\frac{10}{3p}}}{r^{\frac1p+\frac13}}.$$
Using that $p=\frac{6\alpha+2}{\alpha}$ and $q=\frac{6\alpha}{3\alpha-1}$,  this is equivalent to
\begin{equation}
\label{dueydt87908090=-d0=}
|\Qc_r|\lessapprox \frac{R^\frac{-3\alpha^2+7\alpha-\frac13}{3\alpha-1}}{r^{\frac{9\alpha+2}{3\alpha-1}}}\max(N^{\frac{9\alpha+3}{3\alpha-1}}(XN_0)^{-\frac1{3\alpha-1}},(XN_0)^{\frac{3\alpha}{3\alpha-1}},   N_0^{\frac{3\alpha+1}{\alpha}}(XN_0)^{-\frac1{3\alpha-1}}).
\end{equation}
This upper bound will be proved in the next subsection.
\bigskip

\subsection{Plank incidences}
We first prove an intermediate estimate, using our earlier bounds for plate incidences.
\begin{lem}

Suppose that $\mathbb{P}$ satisfies requirements (S2) and (S3) introduced at the  beginning of this section. Let $\Qc_r(\P)$ denote the collection of trilinear $r$-rich $R^{1/3}$-cubes $q$ in $T_R$ with respect to $\P$ (or rather $\P_1,\P_2,\P_3$). Then for each $1\le r\lesssim R^{1/3}$
$$|\Qc_r(\P)|\lessapprox  X(\frac{R^{1/3}}{r})^{\frac{5-3\alpha}{3\alpha-1}}R^{1-\alpha}N_0^2N^\frac{3-3\alpha}{3\alpha-1}.$$	
\end{lem}	
\begin{proof}
We split the $R$-cubes $Q\subset T_R$ according to the number $M$ of intervals $I$ contributing to them. It suffices to focus on a fixed dyadic $M$. Let us assume that we have $Y$ such cubes. Note that $MY\lesssim XR^{1/3}$.

We use that $$YM^{\frac{3-3\alpha}{3\alpha-1}}\lesssim X(R^{1/3})^{\frac{3-3\alpha}{3\alpha-1}}$$
This is immediate if $Y\le X$ since $M\lesssim R^{1/3}$. Also, when $Y\ge X$, we use $M\lesssim \frac{X}YR^{1/3}$ and the fact that $\frac{3-3\alpha}{3\alpha-1}\ge 1$, since $\alpha\le \frac23$.

Thus, it suffices  to prove that for each of the $Y$ cubes $Q$
\begin{equation}
\label{eUHDEET665673E903E-03-=}
|\Qc_r(Q)|\lessapprox (\frac{R^{1/3}}{r})^{\frac{2}{3\alpha-1}}(\frac{M}{r})^{\frac{3-3\alpha}{3\alpha-1}}R^{1-\alpha}N_0^2N^\frac{3-3\alpha}{3\alpha-1},
\end{equation}
where $\Qc_r(Q)$ are those $q\in\Qc_r(\P)$ lying inside $Q$.

For each  $I$ that contributes to $Q$ we tile $Q$ with $(R,R^{2/3},R)$-plates $S$ with normal vector  $\n(I)$. Note that each $S$ can be partitioned into planks $\Sigma$. Due to our assumption (S3) on $\Sigma$ and to $R^\alpha$-periodicity in the $x$ direction, we can split the plates $S$ into two categories. Those that contain $\sim NR^{1-\alpha}$ planks $P\in\P_I$ will be referred to as {\em heavy} and will be denoted by $\S_{heavy,I}$, while those that contain no  $P\in\P_I$ will be called {\em light}, and will play no role in the forthcoming argument. Let $\S_{heavy}$ be the union of all $\S_{heavy,I}$. It is immediate that
$$|\S_{heavy}|\sim \frac{N_0 M}{N}.$$
Let $\widetilde{M}\le M$.
The number of bilinear $\widetilde{M}$-rich $R^{2/3}$-cubes $\Omega\subset Q$ with respect to $\S_{heavy}$
is
$O(R^{1/3}(\frac{N_0M}{N\widetilde{M}})^2)$, due to bilinear Kakeya. By (the rescaled version of) Theorem \ref{induction} $$|\Qc_r(\Omega)|\lessapprox (\frac{R^{1/3}N\widetilde{M}}{r^2})^{\frac{4-6\alpha}{3\alpha-1}}(\frac{N\widetilde{M}}{r})^3R^{\frac23-\alpha},$$ 
where  $\Qc_r(\Omega)$ are those $q\in\Qc_r(\P)$ lying inside $\Omega$.

Thus
$$|\Qc_r(Q)|\lessapprox \frac{M^\frac{3\alpha+1}{3\alpha-1}}{r^\frac{5-3\alpha}{3\alpha-1}}(R^{1/3})^{\frac{4-6\alpha}{3\alpha-1}}R^{1-\alpha}N_0^2N^\frac{3-3\alpha}{3\alpha-1}.$$
Finally, we compare this to \eqref{eUHDEET665673E903E-03-=}. The needed estimate
$$
\frac{M^\frac{3\alpha+1}{3\alpha-1}}{r^\frac{5-3\alpha}{3\alpha-1}}(R^{1/3})^{\frac{4-6\alpha}{3\alpha-1}}R^{1-\alpha}N_0^2N^\frac{3-3\alpha}{3\alpha-1}\lesssim (\frac{R^{1/3}}{r})^{\frac{2}{3\alpha-1}}(\frac{M}{r})^{\frac{3-3\alpha}{3\alpha-1}}R^{1-\alpha}N_0^2N^\frac{3-3\alpha}{3\alpha-1}$$
boils down to $M\lesssim R^{1/3}$.

\end{proof}
We  now finish the proof of \eqref{dueydt87908090=-d0=} using the bound from the previous lemma
$$|\Qc_r|\lessapprox  X(\frac{R^{1/3}}{r})^{\frac{5-3\alpha}{3\alpha-1}}R^{1-\alpha}N_0^2N^\frac{3-3\alpha}{3\alpha-1}=\frac{R^{\frac{-3\alpha^2+3\alpha+\frac23}{3\alpha-1}}}{r^\frac{5-3\alpha}{3\alpha-1}}XN_0^2N^\frac{3-3\alpha}{3\alpha-1}.$$
Let $\beta_1=\frac{1-\alpha}{3\alpha+1}$, $\beta_2=\frac{3\alpha}{3\alpha+1}$, $\beta_3=\frac{\alpha}{3\alpha+1}$.
A simple verification shows that
\begin{align*}
XN_0^2N^\frac{3-3\alpha}{3\alpha-1}&=[N^{\frac{9\alpha+3}{3\alpha-1}}(XN_0)^{-\frac1{3\alpha-1}}]^{\beta_1}[(XN_0)^{\frac{3\alpha}{3\alpha-1}}]^{\beta_2}[  N_0^{\frac{3\alpha+1}{\alpha}}(XN_0)^{-\frac1{3\alpha-1}}]^{\beta_3}\\&\le\max(N^{\frac{9\alpha+3}{3\alpha-1}}(XN_0)^{-\frac1{3\alpha-1}},(XN_0)^{\frac{3\alpha}{3\alpha-1}},   N_0^{\frac{3\alpha+1}{\alpha}}(XN_0)^{-\frac1{3\alpha-1}}).
\end{align*}
Also, the inequality
$$\frac{R^{\frac{-3\alpha^2+3\alpha+\frac23}{3\alpha-1}}}{r^\frac{5-3\alpha}{3\alpha-1}}\lesssim \frac{R^\frac{-3\alpha^2+7\alpha-\frac13}{3\alpha-1}}{r^{\frac{9\alpha+2}{3\alpha-1}}}$$
is equivalent to $r^{12\alpha-3}\lesssim R^{4\alpha-1}$, which in turn is a consequence of our assumption $r\lesssim R^{1/3}$.
\section{Proof of Theorem \ref{4}}
\label{cone}

Each $\theta\in\Theta_{{\C}o^2}(R^{-1})$ is essentially a rectangular box with dimensions $\sim (R^{-\frac12},R^{-1},1)$ with respect to  axes $(\e_\theta^1,\e_\theta^2,\e_\theta^3)$.
\medskip

Assume $\Pc_\theta F$ has wave packet decomposition (see \eqref{wapadec})
$$\Pc_\theta F=\sum_{P\in\P_\theta}w_PW_P.$$
The plank $P$ has dimensions $\sim (R^{\frac12},R,1)$  with respect to the axes $(\e_\theta^1,\e_\theta^2,\e_\theta^3)$.
\smallskip

Let $\Tc_\theta$ be a tiling of $\R^3$ with tubes $\tau$ with dimensions $\sim (R^{\frac12},R,R^{\frac12})$ oriented along the axes $(\e_\theta^1,\e_\theta^2,\e_\theta^3)$ of $\theta$. Each $P\in\P_\theta$ sits inside exactly one tube $\tau\in\Tc_\theta$, and  we will say that $\tau$ and $P$ have the same orientation.
\smallskip

For each $\theta$, let $\P_\theta'\subset \P_\theta$ be such that
$|w_P|\sim w$ for each $P\in\P_\theta'$ and such that each tube $\tau\in\Tc_\theta$ contains either $\sim N$ planks $P\in\P_\theta'$ with the same orientation,  or no such plank.
Let $$G=\sum_{\theta}\sum_{P\in\P'_\theta}w_PW_P.$$
Write
$$\P'=\cup_\theta\P_\theta'.$$

We will need the following Kakeya-type input.

\begin{lem}[Kakeya-type estimate for planks]
	\label{6}	
	We have
	$$\|\sum_{P\in\P'}1_P\|_2^{2}\le (\log R) N\|\sum_{P\in\P'}1_P\|_1.$$
\end{lem}
\begin{proof}
	Let $\Tc'$ be a collection of tubes $\tau$ containing all the planks in $\P'$, with each $\tau$ containing roughly $N$ planks (with the same orientation).
	The angle between two planks  is the same as the angle between the corresponding tubes. These angles are of the form $jR^{-1/2}$ with $1\le j\le R^{1/2}$.
	
	Two planks $P,P'$ with angle $\omega$ have intersection
	$$|P\cap P'|\lesssim R^{1/2}\omega^{-2}.$$
	
	For each fixed $\tau$ and $\omega$ we have
	$$|\{\tau':\; \tau'\cap \tau\not=\emptyset, \;\sphericalangle(\tau,\tau')= \omega\}|\lesssim R^{1/2}\omega.$$
	This is because all $\tau'$ in the collection are essentially coplanar and satisfy $|10\tau\cap 10\tau'|\sim R^{3/2}\omega^{-1}.$

	Using these  observations we write (the first sum runs over $\omega=jR^{-1/2}$, $1\le j\lesssim R^{1/2}$)
	\begin{align*}
	\|\sum_{P\in\P'}1_P\|_2^{2}&=\sum_{\omega=R^{-1/2}}^1\sum_{\tau\in \Tc'}\sum_{P\subset \tau}\sum_{\tau'\in \Tc'\atop{\sphericalangle(\tau,\tau')\sim \omega\atop{\tau\cap \tau'\not=\emptyset}}}\sum_{P'\subset \tau'}|P\cap P'|
	\\&\lesssim \sum_{\omega=R^{-1/2}}^1\sum_{\tau\in \Tc'}\sum_{P\subset \tau}R^{1/2}\omega NR^{1/2}\omega^{-2}\\&
	\sim |\P'|\sum_{\omega=R^{-1/2}}^1RN\omega^{-1}\\&
	\sim |\P'|\sum_{j=1}^{R^{1/2}}\frac{R^{3/2} N}{j}\\& \sim (\log R) NR^{3/2}|\P'|.
	\end{align*}
\end{proof}

The following result represents a refinement of the $l^4(L^4)$ decoupling for boxes of canonical scale covering the cone. It replaces the factor $R^{\frac18}$ in Theorem \ref{3} with the smaller $N^{\frac14}$.

\begin{pr}
	\label{7}	
	We have
	$$\|G\|_{L^4(\R^3)}\lesssim_\epsilon R^{\epsilon}N^{\frac14}(\sum_{\theta\in \Theta_{{\C}o^2}(R^{-1})}\|\Pc_\theta G\|^4_{L^4(\R^3)})^{\frac14}.$$		
\end{pr}
\begin{proof}
We may assume that $w\sim 1$.
	We first observe that
	$$(\sum_{\theta\in \Theta_{{\C}o^2}(R^{-1})}\|\Pc_\theta G\|^4_{L^4(\R^3)})^{\frac14}\sim R^{\frac38}|\P'|^{\frac14}.$$

	We use Theorem \ref{c3} to evaluate the left hand side
	\begin{align*}
	\|G\|_{L^4(\R^3)}&\lesssim_\epsilon R^{\epsilon}\|(\sum_{\theta}|\Pc_\theta G|^2)^{\frac12}\|_{L^4(\R^3)}\\&\lesssim_\epsilon R^\epsilon\|\sum_{P\in\P'}\chi_P\|_{L^2(\R^3)}^{\frac12}\\&\lesssim_\epsilon R^{\frac38+\epsilon}N^{\frac14}|\P'|^{\frac14}.
	\end{align*}
	In the last inequality we have used a standard variation of Lemma \ref{6}.
	
\end{proof}
Let us now see the proof of Theorem \ref{4}. Invoking interpolation (Exercise 9.21 in \cite{Dembook}) it will suffice to prove the case $p=4$, that is
$$\|F\|_{L^4(\R^3)}\lesssim_{\epsilon}R^{\frac14+\epsilon}(\sum_{\gamma\in \Gamma(R^{-1})}\|\Pc_\gamma F\|^4_{L^4(\R^3)})^{\frac14},$$
for each $F:\R^3\to\C$ with the Fourier transform supported inside $\Nc_{{\C}o^2}(R^{-1})$. Recall that
$$F=\sum_{P\in\P}w_PW_P.$$
where $\P=\cup_\theta\cup_{P\in\P_\theta}P$.
By normalizing we may assume that the largest coefficient $|w_P|$ is 1.
We split
$$\P=(\bigcup_{j=1}^{(\log R)^{O(1)}}\P_j)\cup\P_{small}.$$
The collection $\P_{small}$ contains all planks with coefficients $|w_P|\lesssim R^{-100}$. The contribution from the corresponding wave packets is easily seen to be negligible. The planks in each $\P_j$ have two properties. First, we have $|w_{P_j}|\sim w_j$ for some $w_j\in (0,\infty).$ Second, each $P\in \P_j$ sits inside some tube $\tau\in \Tc_j\subset \Tc$ (having the same orientation), and there are roughly $N_j\ge 1$ planks $P\in\P_j$ inside each tube $\tau\in \Tc_j$ with the same orientation as $\tau$.
\smallskip

Fix $j$ and write
$$G=\sum_{P\in\P_j}w_PW_P.$$
Invoking the triangle inequality, it will suffice to prove that
\begin{equation}
\label{9}
\|G\|_{L^4(\R^3)}\lesssim_{\epsilon}R^{\epsilon+\frac14}(\sum_{\gamma\in \Gamma(R^{-1})}\|\Pc_\gamma F\|^4_{L^4(\R^3)})^{\frac14}.
\end{equation}
First, we use Proposition \ref{7} to write
\begin{equation}
\label{10}
\|G\|_{L^4(\R^3)}\lesssim_\epsilon R^{\epsilon}N_j^{\frac14}(\sum_{\theta\in \Theta_{{\C}o^2}(R^{-1})}\|\Pc_\theta G\|^4_{L^4(\R^3)})^{\frac14}.
\end{equation}
Second, Corollary \ref{8} gives for each $\theta$
$$\|\Pc_\theta G\|^4_{L^4(\R^3)}\lesssim_\epsilon R^\epsilon (\frac{R}{N_j})^{\frac14}(\sum_{\gamma\in \Gamma(R^{-1})\atop{\gamma\subset \theta}}\|\Pc_\gamma F\|^4_{L^4(\R^3)})^{\frac14},$$
and summation leads to
\begin{equation}
\label{11}
(\sum_{\theta\in \Theta_{{\C}o^2}(R^{-1})}\|\Pc_\theta G\|^4_{L^4(\R^3)})^{\frac14}\lesssim_\epsilon R^\epsilon (\frac{R}{N_j})^{\frac14}(\sum_{\gamma\in \Gamma(R^{-1})}\|\Pc_\gamma F\|^4_{L^4(\R^3)})^{\frac14}.
\end{equation}
Now \eqref{9} follows from \eqref{10} and \eqref{11}.
\bigskip

\section{Appendix\\
	An Improved Fourth Derivative Estimate for Exponential Sums\\
	D.R. Heath-Brown\\Mathematical Institute, Oxford}

In this appendix we will show how Theorem \ref{n=3partialrange}
may be applied to
establish an improved version of the ``fourth derivative estimate'' for
exponential sums. The classical van der Corput $k$-th derivative
estimate (Titchmarsh \cite[Theorems 5.9, 5.11, \& 5.13]{Titch}, for
example), can be given as follows. Suppose that $k\ge 2$ is an integer, and
let $f(x):[0,N]\to\mathbb{R}$ have a continuous $k$-th
derivative on $(0,N)$ with $0<\lambda_k\le f^{(k)}(x)\le A\lambda_k$.
Then
\begin{equation}\label{rhb0}
\sum_{n\le N}e(f(n))\lesssim A^{2^{2-k}}N\lambda_k^{1/(2^k-2)}
+N^{1-2^{2-k}}\lambda_k^{-1/(2^k-2)},
\end{equation}
where the implied constant is independent of $k$.

By using the (essentially) optimal estimate for Vinogradov's mean
value, as proved by Bourgain, Demeter and Guth \cite{BDG}, one can obtain an
alternative bound
\begin{equation}\label{rhb3}
\sum_{n\le N}e(f(n))
\lesssim_{A,k,\varepsilon}N^{1+\varepsilon}
(\lambda_k^{1/k(k-1)}+N^{-1/k(k-1)}+N^{-2/k(k-1)}\lambda_k^{-2/k^2(k-1)}),
\end{equation}
for any fixed $\varepsilon>0$ (see Heath-Brown \cite[Theorem 1]{HB}).
In most situations this is superior to the classical estimate as soon
as $k\ge 4$, and
the object of this appendix is to show how Theorem \ref{n=3partialrange}
of the present
paper allows one to produce a further improvement in the case $k=4$. The
result we obtain is the following.

\begin{te}\label{rhbt1}
	Let $f(x):[0,N]\to\mathbb{R}$ have a continuous 4-th
	derivative on $(0,N)$ with $0<\lambda_4\le f^{(4)}(x)\le A\lambda_4$
	for some constant $A\ge 1$.  Write $\lambda_4=N^{-\varpi}$, and
	suppose that $N^{-2}\lesssim\lambda_4\lesssim N^{-1}$. Then
	\begin{equation}\label{rhb2}
	\sum_{n\le N}e(f(n))
	\lesssim_{A,\varepsilon} N^{1-\varpi/(4\varpi+8)+\varepsilon}+N^{8/9+\varepsilon},
	\end{equation}
	for any fixed $\varepsilon>0$.
\end{te}

In practice one would usually apply the third
derivative bound when $\lambda_4\lesssim N^{-2}$, giving a stronger result than
can be obtained from the fourth derivative estimates.
When $k=4$ and $\lambda_4=N^{-\varpi}$ the bound (\ref{rhb3}) yields
\begin{equation}\label{rhb3a}
\sum_{n\le N}e(f(n)) \lesssim_{A,\varepsilon} \left\{\begin{array}{cc}
N^{1-\varpi/12+\varepsilon}, & \lambda_4\gg N^{-1},\\ N^{11/12+\varepsilon}, &
N^{-2}\lesssim\lambda_4\lesssim N^{-1},\end{array}\right.
\end{equation}
while (\ref{rhb2}) produces
\begin{equation}\label{rhb4a}
\sum_{n\le N}e(f(n)) \lesssim_{A,\varepsilon} \left\{\begin{array}{cc}
N^{1-\varpi/(4\varpi+8)+\varepsilon}, & N^{-8/5}\lesssim\lambda_4\lesssim N^{-1},\\
N^{8/9+\varepsilon}, & N^{-2}\lesssim\lambda_4\lesssim N^{-8/5}.
\end{array}\right.
\end{equation}
Thus we get a significant saving when $N^{-2}\lesssim\lambda_4\lesssim N^{-1}$.

As an application of Theorem \ref{rhbt1} we will prove a bound for
the Lindel\"{o}f $\mu(\sigma)$ function associated to the Riemann
Zeta-function.
\begin{te}\label{rhbt2}
	We have
	\[\zeta(\tfrac{11}{15}+it)\lesssim_{\varepsilon}(|t|+1)^{1/15+\varepsilon}\]
	for any fixed $\varepsilon>0$, so that
	$\mu(\tfrac{11}{15})\le\tfrac{1}{15}$.
\end{te}

Strictly speaking, we do not claim that this bound is new. Indeed given the plethora of
published bounds and the convexity of $\mu(\sigma)$ it is not easy to say with
confidence that a given result is new. Moreover one can make further small
improvements on Theorem \ref{rhbt2} by using exponent pairs to sharpen the application of the third derivative bound in the argument below, and  by replacing the trivial bound (for small $N$)
by the case $k=5$ of (\ref{rhb3}). However our main purpose with Theorem \ref{rhbt2} is
to demonstrate a neat bound coming directly from the new fourth derivative estimate.
\bigskip

The proof of Theorem \ref{rhbt1}
begins by following the argument from \cite[Section 2]{HB}.
We assume that $k=4$, although the initial stages of the method
work for arbitrary $k\ge 3$.
We write $H=[(A\lambda_k)^{-1/k}]$ and for $\boldsymbol{\alpha}\in [0,1]^{k-1}$ we define
\[\nu(\boldsymbol{\alpha})=\#\{n\le N-H: ||f^{(j)}(n)/j!-\alpha_j||\le H^{-j}\mbox{
	for } 1\le j\le k-1\}.\]
If we set $\alpha^*=f^{(k-1)}(0)/(k-1)!$ then whenever
$\nu(\boldsymbol{\alpha})\not=0$ we must have
\[|\alpha_{k-1}-\alpha^*|\le
\frac{\left|f^{(k-1)}(n)-f^{(k-1)}(0)\right|}{(k-1)!}+H^{1-k}
\lesssim_{A,k} N\lambda_k+\lambda_k^{(k-1)/k}\]
for some $n\le N-H$. If we write this as
$|\alpha_{k-1}-\alpha^*|\le \xi$, say, then in our situation
we have    $\xi\lesssim_A N\lambda_4$, since $\lambda_4\gg N^{-2}$.  We may now
replace Lemma 1 of \cite{HB} by the estimate
\[\sum_{n\le N}e(f(n)) \lesssim_{A,\varepsilon} H+N^{1-1/s}\mathcal{N}^{1/2s}
\left\{H^{-2s+k(k-1)/2}J\right\}^{1/2s},\]
with
\[\mathcal{N}=\#\left\{m,n\le N:
\left|\left|\frac{f^{(j)}(m)}{j!}-\frac{f^{(j)}(n)}{j!}\right|\right|
\le 2H^{-j}\mbox{ for } 1\le j\le k-1\right\}\]
and
\[J=\int_0^1\ldots\int_0^1\int_{\alpha^*-\xi}^{\alpha^*+\xi}
\left|\sum_{n\le x}e(\alpha_1 n+\ldots +\alpha_{k-1}
n^{k-1})\right|^{2s}d\alpha_{k-1} d\alpha_{k-2}\ldots d\alpha_1\]
for some $x\le H$. Note that the argument of \cite[Section 2]{HB}
works for any real $s\ge 1$.

We now specialize to $k=4$, and plan to apply Theorem
\ref{n=3partialrange}. As remarked
in connection with Theorem \ref{n=3partialrange} the proof allows
us to replace the
range $[0,N^{-\beta}]$ by any interval of length $N^{-\beta}$. What is
less clear is whether the estimate of the theorem holds uniformly with
respect to $\beta$. To clarify this point we suppose that the theorem
yields the bounds
\begin{eqnarray*}
	&\int_{[0,1]^2\times[\tau,\tau+N^{-\beta}]}|\sum_{n\le N}
	e(\alpha_1 n+\alpha_2n^2 +\alpha_3n^3)|^{12-2\beta}d\alpha_1
	d\alpha_2 d\alpha_3&\\
	&\le\, C(\varepsilon,\beta)N^{6-2\beta+\varepsilon}&,
\end{eqnarray*}
for $0\le\beta\le\tfrac32$, uniformly in $\tau$. Set
$R=\lceil\tfrac32\varepsilon\rceil$ and $r=\lceil\tfrac23 R\beta\rceil$. If we
then write $\beta_r=3r/2R$ it follows that $0\le\beta_r\le 3/2$ and
$\beta_r-\varepsilon<\beta\le\beta_r$. We now observe firstly that
\[|\sum_{n\le N}e(\alpha_1 n+\alpha_2n^2 +\alpha_3n^3)|^{12-2\beta}
\le N^{2(\beta_r-\beta)}
|\sum_{n\le N}e(\alpha_1 n+\alpha_2n^2 +\alpha_3n^3)|^{12-2\beta_r},\]
and secondly that the interval $[\tau,\tau+N^{-\beta}]$ can be covered by at
most $N^{\varepsilon}$ intervals of length $N^{-\beta_r}$. Thus
\begin{eqnarray*}
	\lefteqn{\int_{[0,1]^2\times[\tau,\tau+N^{-\beta}]}|\sum_{n\le N}
		e(\alpha_1 n+\alpha_2n^2 +\alpha_3n^3)|^{12-2\beta}d\alpha_1
		d\alpha_2 d\alpha_3}\\
	&\le& N^{3\varepsilon}\sup_\sigma
	\int_{[0,1]^2\times[\sigma,\sigma+N^{-\beta_r}]}|\sum_{n\le N}
	e(\alpha_1 n+\alpha_2n^2 +\alpha_3n^3)|^{12-2\beta_r}d\alpha_1
	d\alpha_2 d\alpha_3\\
	&\le& C(\varepsilon,\beta_r)N^{6-2\beta_r+4\varepsilon}\\
	&\le& C(\varepsilon,\beta_r)N^{6-2\beta+4\varepsilon}.
\end{eqnarray*}
We therefore see that Theorem \ref{n=3partialrange} holds
(with $\varepsilon$ replaced by
$4\varepsilon$) with implied constant
$C(\varepsilon)=\max_{r\le R}C(\varepsilon,\beta_r)$ depending
only on $\varepsilon$.

We proceed to apply this uniform version of Theorem
\ref{n=3partialrange}.  We have
assumed that $\Lambda_4\le cN^{-1}$ for some constant $c$. With this
in mind we
define $\beta$ by the relation
\[H^{\beta}=\min\left\{c(N\lambda_4)^{-1}\,,\,H^{3/2}\right\}.\]
We then have $0\le\beta\le\tfrac32$ as required. Moreover,
\[H=[(A\lambda_4)^{-1/4}]=N^{\varpi/4+O(1/\log N)}\]
and $1\lesssim\varpi\lesssim 1$, whence
$N=H^{4/\varpi+O(1/\log N)}$ and
\[c(N\lambda_4)^{-1}=H^{4(\varpi-1)/\varpi+O(1/\log N)}.\]
We therefore see that
\[\beta=\min\left\{\frac{4(\varpi-1)}{\varpi}\,,\,\frac{3}{2}\right\}
+O\left(\frac{1}{\log N}\right).\]
Since $\xi\lesssim_A N\lambda_4\lesssim H^{-\beta}$ we can cover the range
$[\alpha^*-\xi,\alpha^*+\xi]$ with $O_A(1)$ intervals of length
$H^{-\beta}$. It then follows on taking $s=6-\beta$ that
\[J\lesssim_{\varepsilon} H^{2s-6+\varepsilon}.\]
Moreover \cite[Lemma 3]{HB} yields
$\mathcal{N}\lesssim_{A,\varepsilon}N^{1+\varepsilon}$ when
$N^{-2}\lesssim\lambda_4\lesssim N^{-1}$.  We therefore conclude
that
\[\sum_{n\le N}e(f(n)) \lesssim_{A,\varepsilon} \lambda_4^{-1/4}+N^{1-1/2s+\varepsilon}.\]
However
\[\frac{1}{2s}=\frac{1}{12-2\beta}=
\min\left\{\frac{\varpi}{4\varpi+8}\,,\,\frac{1}{9}\right\}
+O((\log N)^{-1}),\]
whence
\[\sum_{n\le N}e(f(n)) \lesssim_{A,\varepsilon}
N^{\varpi/4}+N^{1-\varpi/(4\varpi+8)+\varepsilon}+N^{8/9+\varepsilon}.\]
Theorem \ref{rhbt1} then follows.
\bigskip

To deduce Theorem \ref{rhbt2} it suffices by the approximate
functional equation (see Chapter 2 in \cite{GrCo} or \cite{Titch}) to show that
\[\sum_{N<n\le 2N}n^{it}\lesssim_{\varepsilon}N^{11/15}t^{1/15+\varepsilon}\]
for $N\le t^{1/2}$ and any fixed $\varepsilon>0$. The bound is trivial for
$N\le t^{1/4}$, and so we focus on the remaining range
$t^{1/4}\le N\le t^{1/2}$. When $f(x)=t(\log x)/2\pi$ one may apply
the third derivative estimate, taking $\lambda_3$ to have order
$tN^{-3}$.  The bound (\ref{rhb0}) then shows that
\[\sum_{N<n\le 2N}n^{it}\lesssim N^{1/2}t^{1/6}+Nt^{-1/6}.\]
This gives a satisfactory bound $O(N^{11/15}t^{1/15})$ when
$t^{3/7}\le N\le t^{1/2}$.  For the remaining range
$t^{1/4}\le N\le t^{3/7}$ we use our various
fourth derivative estimates, with $\lambda_4$ of order $tN^{-4}$.  When
$t^{1/4}\le N\le t^{1/3}$ the bound (\ref{rhb3a}) yields
\[\sum_{N<n\le 2N}n^{it}\lesssim_{\varepsilon}
N^{1+\varepsilon}\lambda_4^{1/12}\lesssim N^{2/3+\varepsilon}t^{1/12}
\lesssim N^{11/15+\varepsilon}t^{1/15}.\]
For $t^{5/12}\le N\le t^{3/7}$ we have $N^{-5/3}\lesssim\lambda_4\lesssim N^{-8/5}$
so that (\ref{rhb4a}) produces an estimate
\[\sum_{N<n\le 2N}n^{it}\lesssim_{\varepsilon}N^{8/9+\varepsilon}\le
N^{11/15+\varepsilon}t^{1/15}.\]
Finally, when $t^{1/3}\le N\le t^{5/12}$ we find that
$N^{-8/5}\lesssim\lambda_4\lesssim N^{-1}$.  In this case (\ref{rhb4a}) shows
that
\[\sum_{N<n\le 2N}n^{it}\lesssim_{\varepsilon} N^{1-\varpi/(4\varpi+8)+\varepsilon}.\]
If we write $t=N^{\tau}$ we will have $\tfrac{12}{5}\le\tau\le 3$, and
$\varpi=4-\tau+O(1/\log N)$.  It therefore suffices to show that
\[1-\frac{4-\tau}{24-4\tau}\le \frac{11}{15}+\frac{\tau}{15}\]
for $\tfrac{12}{5}\le\tau\le 3$, and this is readily
verified, completing the proof of Theorem \ref{rhbt2}. The reader will
note that the critical case is that in which $\lambda_4$ is of order
$N^{-5/3}$.

\bigskip
\bigskip

Mathematical Institute,

Radcliffe Observatory Quarter,

Woodstock Road,

Oxford

OX2 6GG

UK
\bigskip

{\tt rhb@maths.ox.ac.uk}

\end{document}